\documentclass[draft]{amsart}

\usepackage{cite,amstext,amsfonts,amsmath,amsthm,amscd,amssymb,euscript,mathrsfs,latexsym,gastex,color,cancel,soul}

\newtheorem{theorem}{Theorem}[section]

\newtheorem{proposition}[theorem]{Proposition}

\newtheorem{lemma}[theorem]{Lemma}

\newtheorem{corollary}[theorem]{Corollary}

\newtheorem{remark}[theorem]{Remark}

\theoremstyle{definition}

\newtheorem{problem}[theorem]{Problem}

\DeclareMathOperator{\con}{con}

\DeclareMathOperator{\ini}{ini}

\DeclareMathOperator{\mul}{mul}

\DeclareMathOperator{\occ}{occ}

\DeclareMathOperator{\simple}{sim}

\DeclareMathOperator{\var}{var}

\numberwithin{equation}{section}

\numberwithin{figure}{section}

\begin{document}

\title[Two weaker variants of congruence permutability]{Two weaker variants\\ 
of congruence permutability\\ 
for monoid varieties}
\thanks{The work was supported by the Ministry of Science and Higher Education of the Russian Federation (project FEUZ-2020-0016).}

\author[S. V. Gusev]{Sergey V. Gusev}
\address{Institute of Natural Sciences and Mathematics, Ural Federal University, Lenina str.~51, 620000 Ekaterinburg, Russia}
\email{sergey.gusb@gmail.com, bvernikov@gmail.com}

\author[B. M. Vernikov]{Boris M. Vernikov}



\begin{abstract}
We completely determine all varieties of monoids on whose free objects all fully invariant congruences or all fully invariant congruences contained in the least semilattice congruence permute. Along the way, we find several new monoid varieties with the distributive subvariety lattice (only a few examples of varieties with such a property are known so far). 
\end{abstract}

\maketitle

\section{Introduction and summary}
\label{Sec: introduction}

In this article, we study varieties of monoids as semigroups equipped with an additional \mbox{0-ary} operation that fixes the identity element. But we start with the universal-algebraic notions that closely connected with our research.

For congruences $\alpha$ and $\beta$ on an algebra $A$, we denote by $\alpha\beta$ their \emph{relational product}, that is, the relation
$$
\{(a,b)\in A\times A\mid a\,\alpha\,c\,\beta\,b\ \text{for some}\ c\in A\}.
$$
The congruences $\alpha$ and $\beta$ on an algebra $A$ are said to \emph{permute} if $\alpha\beta=\beta\alpha$. It is well-known that, exactly in this case, the relation $\alpha\beta$ is again a congruence on $A$ that coincides with the \emph{lattice join} $\alpha\vee\beta$ of $\alpha$ and $\beta$, that is, the least congruence containing both $\alpha$ and $\beta$.

The family of \emph{congruence permutable} varieties (that is, varieties on whose algebras all congruences permute) is very rich and important. In particular, it includes all varieties of groups and (not necessarily associative) rings. Unfortunately, proper semigroup or monoid varieties fail to belong to this family. Saying so, we refer to the following fact: a congruence permutable semigroup or monoid variety must consist entirely of groups. For semigroup varieties this claim first verified by Tully~\cite{Tully-64} and was then rediscovered and strengthened several times (see Evans~\cite[p.~35]{Evans-71}, Freese and Nation~\cite[Corollary on pp.~57--58]{Freese-Nation-73}, Jones~\cite[Theorem~1.2(iii)]{Jones-88} or Lipparini~\cite[Corollary~0]{Lipparini-95}). Analog of this fact is true for monoid varieties as well (see Section~\ref{Sec: generalizations}).

However, if we restrict to fully invariant congruences, the situation considerably improves since there already exist interesting varieties on whose semigroups all fully invariant congruences permute. For example, Pastijn~\cite{Pastijn-91} and Petrich and Reilly~\cite{Petrich-Reilly-90} have observed that fully invariant congruences on completely simple semigroups permute. 

Considering fully invariant congruences is most natural for free objects of varieties. Indeed, if $\mathbf V$ is a variety of algebras, then the lattice of fully invariant congruences on $\mathbf V$-free object over a countably infinite alphabet is known to be anti-isomorphic to the subvariety lattice of $\mathbf V$. Thus, any ``positive'' information about the fully invariant congruences on $\mathbf V$-free objects contributes to clarifying the structure of the subvariety lattice of $\mathbf V$. In particular, the permutability of fully invariant congruences on $\mathbf V$-free objects reflects in the very important Arguesian property of the corresponding subvariety lattice. (Recall that, by results of J\'onsson, any lattice of permuting equivalences is Arguesian~\cite{Jonsson-53} and the class of Arguesian lattices is self-dual~\cite{Jonsson-72}.) For brevity, we call a variety of algebras $\mathbf V$ $fi$-\emph{permutable} if every two fully invariant congruences on any $\mathbf V$-free object permute.

In order to introduce one more interesting class of varieties closely related with $fi$-permutable ones, we need some definitions and notation. Commutative idempotent semigroups usually are called \emph{semilattices}. We call commutative idempotent monoids \emph{semilattice monoids}. Recall that an element $a$ of a lattice $L$ is called \emph{neutral} if, for any $x,y\in L$, the sublattice of $L$ generated by $a$, $x$ and $y$ is distributive. Let $\mathbf{SL}_{\mathsf{sem}}$ [respectively, $\mathbf{SL}$] be the variety of all semilattices [semilattice monoids] and $\mathbb{SEM}$ [respectively, $\mathbb{MON}$] be the lattice of all semigroup [monoid] varieties. It is well known that $\mathbf{SL}_{\mathsf{sem}}$ [respectively, $\mathbf{SL}$] is an atom of the lattice $\mathbb{SEM}$ [respectively, $\mathbb{MON}$] and a neutral element of this lattice (see Volkov~\cite[Proposition~4.1]{Volkov-05} for the semigroup case and Gusev~\cite[Theorem~1.1]{Gusev-18} for the monoid one). These claims together with well-known properties of neutral elements in lattices (see~\cite[Theorem~254]{Gratzer-11}, for instance) imply that the lattice $\mathbb{SEM}$ [respectively, $\mathbb{MON}$] is decomposed into a subdirect product of the 2-element chain and the interval $[\mathbf{SL}_{\mathsf{sem}},\mathbf{SEM}]$ [the interval $[\mathbf{SL},\mathbf{MON}]$, respectively], where $\mathbf{SEM}$ [respectively, $\mathbf{MON}$] is the variety of all semigroups [monoids]. On every semigroup [monoid] $S$, there exists the least congruence $\sigma$ such that the quotient $S/\sigma$ is a semilattice [semilattice monoid]. The congruence $\sigma$ is called the \emph{least semilattice congruence} on $S$. The aforementioned observations show that it is natural to consider semigroup and monoid varieties $\mathbf V$ such that, on any $\mathbf V$-free object $F$, not all fully invariant congruences but only those of them that are contained in the least semilattice congruence on $F$ permute. We call a semigroup or monoid variety with such a property \emph{almost $fi$-permutable}. In view of the aforementioned properties of the varieties $\mathbf{SL}_{\mathsf{sem}}$ and $\mathbf{SL}$, an almost $fi$-permutable variety of semigroups or monoids has the Arguesian subvariety lattice. The class of almost $fi$-permutable semigroup varieties is quite wide; in particular, it contains all completely regular varieties~\cite{Pastijn-91,Petrich-Reilly-90}. As we will prove below, the same is true for almost $fi$-permutable monoid varieties (see Theorem~\ref{Th: almost fi-perm}).

A complete classification of $fi$-permutable and almost $fi$-permutable semigroup varieties is given by Vernikov and Volkov in~\cite{Vernikov-Volkov-97} and~\cite{Vernikov-Volkov-00}, respectively. The second result contains a minor inaccuracy that is fixed by Vernikov and Shaprynski\v{\i}~\cite{Vernikov-Shaprynskii-14}. Semigroup varieties with other multiplicative restrictions to fully invariant congruences on their free objects were examined in Vernikov~\cite{Vernikov-04a,Vernikov-04b,Vernikov-04c}, Vernikov and Shaprynski\v{\i}~\cite{Vernikov-Shaprynskii-14} and some other articles; further details see in Section~\ref{Sec: generalizations}. The present article is devoted to a complete determination of $fi$-permutable and almost $fi$-permutable monoid varieties. 

To formulate the main results of the article, we need some definitions and notation. Let $\mathfrak X$ be a countably infinite set called an \emph{alphabet}. As usual, we denote by $\mathfrak X^+$ [respectively, by $\mathfrak X^\ast$] the free semigroup [monoid] over the alphabet $\mathfrak X$; elements of $\mathfrak X^+$ and $\mathfrak X^\ast$ are called \emph{words}, while elements of $\mathfrak X$ are said to be \emph{letters}. Words unlike letters are written in bold. The two words forming an identity are connected by the symbol $\approx$, while the symbol $=$ denotes, among other things, the equality relation on $\mathfrak X^+$ or $\mathfrak X^\ast$. 

For a possibly empty set $W$ of words, we denote by $I(W)$ the set of all words that are not subwords of words from $W$. It is clear that $I(W)$ is an ideal of $\mathfrak X^\ast$. Let $S(W)$ denote the Rees quotient monoid $\mathfrak X^\ast/I(W)$. If $W=\{\mathbf w_1,\mathbf w_2,\dots,\mathbf w_k\}$, then we write $S(\mathbf w_1,\mathbf w_2,\dots,\mathbf w_k)$ rather than $S(\{\mathbf w_1,\mathbf w_2,\dots,\mathbf w_k\})$. The monoids of the kind $S(W)$ often appeared in the literature (see~\cite{Gusev-20+,Gusev-Vernikov-18,Jackson-05,Jackson-Lee-18,Lee-14}, for instance).

As usual, $\mathbb N$ denotes the set of all natural numbers. We denote by $\var M$ the monoid variety generated by a monoid $M$. Put
$$
\mathbf D_k=
\begin{cases}
\var S(xy)&\text{if}\ k=1,\\
\var S(xt_1xt_2\cdots xt_{k-1}x)&\text{if}\ k>1
\end{cases}
\quad\text{and}\quad\mathbf D_\infty=\bigvee_{k\in\mathbb N}\mathbf D_k.
$$

For any $n\in\mathbb N$, we denote by $S_n$ the full symmetric group on the set $\{1,2,\dots,n\}$. For convenience, we put $S_0=S_1$. Let $\mathbb N_0=\mathbb N\cup\{0\}$. For any $n,m\in\mathbb N_0$, $\pi,\tau\in S_n$ and $\rho\in S_{n+m}$, we define the words
\begin{align*}
\mathbf c_{n,m}[\rho]&=\biggl(\prod_{i=1}^{n} z_it_i\biggr)xyt\biggl(\prod_{i=n+1}^{n+m} z_it_i\biggr)x\biggl(\prod_{i=1}^{n+m} z_{i\rho}\biggr)y,\\[-3pt]
\mathbf c_{n,m}^\prime[\rho]&=\biggl(\prod_{i=1}^{n} z_it_i\biggr)yxt\biggl(\prod_{i=n+1}^{n+m} z_it_i\biggr)x\biggl(\prod_{i=1}^{n+m} z_{i\rho}\biggr)y,\\[-3pt]
\mathbf w_n[\pi,\tau]&=\biggl(\prod_{i=1}^n z_it_i\biggr)x\biggl(\prod_{i=1}^n z_{i\pi}z_{n+i\tau}\biggr)x\biggl(\prod_{i=n+1}^{2n} t_iz_i\biggr),\\[-3pt]
\mathbf w_n^\prime[\pi,\tau]&=\biggl(\prod_{i=1}^n z_it_i\biggr)x^2\biggl(\prod_{i=1}^n z_{i\pi}z_{n+i\tau}\biggr)\biggl(\prod_{i=n+1}^{2n} t_iz_i\biggr).
\end{align*}
We denote by $\mathbf d_{n,m}[\rho]$ and $\mathbf d_{n,m}^\prime[\rho]$ the words that are obtained from the words $\mathbf c_{n,m}[\rho]$ and $\mathbf c_{n,m}^\prime[\rho]$, respectively, when reading the last words from right to left. 

We fix notation for the following six identities:
\begin{align*}
\sigma_1:&\enskip xysxty\approx yxsxty,\\
\sigma_2:&\enskip xsytxy\approx xsytyx,\\
\sigma_3:&\enskip xsxyty\approx xsyxty,\\
\alpha_1:&\enskip xysxtxhy\approx yxsxtxhy,\\
\alpha_2:&\enskip xysxtyhx\approx yxsxtyhx,\\
\alpha_3:&\enskip xysytxhx\approx yxsytxhx.
\end{align*}
For any $i=1,2,3$, we denote by $\beta_i$ the identity dual to $\alpha_i$. The identity $\sigma_1$ [respectively, $\sigma_2$, $\sigma_3$] allows us to swap adjacent occurrences of two letters in a word whenever both the occurrences are non-last [both the occurrences are non-first, one occurrence is non-first and another one is non-last]. We will use these facts throughout the article many times without explicitly specifying this. 

Let $\var\Sigma$ denote the monoid variety given by an identity system $\Sigma$. We fix notation for the following monoid varieties:
\begin{align*}
&\mathbf K=\var\{x^2y\approx x^2yx,\,xyx\approx xyx^2,\,x^2y^2\approx y^2x^2\},\\
&\mathbf N=\var\{x^2\approx x^3,\,x^2y\approx yx^2,\,xyxzx\approx x^2yz,\,\sigma_2,\,\sigma_3\},\\
&\mathbf P_n=\var\{x^n\approx x^{n+1},\,x^ny\approx yx^n,\,x^2y\approx xyx\},\ \text{where}\ n\in\mathbb N,\\
&\mathbf Q_{r,s}=\var
\left\{
\begin{array}{l}
x^2\approx x^3,\,x^2y\approx yx^2,\\
\sigma_3,\,\alpha_i,\,\beta_j,\\
\mathbf c_{n,m}[\rho]\approx\mathbf c_{n,m}^\prime[\rho],\\
\mathbf d_{n,m}[\rho]\approx\mathbf d_{n,m}^\prime[\rho]
\end{array}
\middle\vert
\begin{array}{l}
1\le i,j\le 3,\\
i\ne r,\,j\ne s,\\
n,m\in\mathbb N_0,\\
\rho\in S_{n+m}
\end{array}
\right\}
,\ \text{where}\ 1\le r,s\le 3,\\
&\mathbf R=\var\{x^2\approx x^3,\,x^2y\approx yx^2,\,\sigma_1,\,\sigma_2,\,\mathbf w_n[\pi,\tau]\approx\mathbf w_n^\prime[\pi,\tau]\mid n\in\mathbb N,\,\pi,\tau\in S_n\}.
\end{align*}
By $\mathbf V^\delta$ we denote the monoid variety \emph{dual} to the variety $\mathbf V$ (in other words, $\mathbf V^\delta$ consists of monoids dual to members of $\mathbf V$).

Our first main result is the following

\begin{theorem}
\label{Th: fi-perm}
A variety of monoids $\mathbf V$ is $fi$-permutable if and only if one of the following holds:
\begin{itemize}
\item[(i)] $\mathbf V$ is a group variety;
\item[(ii)] $\mathbf V$ is a variety of idempotent monoids;
\item[(iii)] $\mathbf V$ is contained in one of the following varieties: $\mathbf D_\infty\vee\mathbf N$, $\mathbf D_\infty\vee\mathbf N^\delta$, $\mathbf K$, $\mathbf K^\delta$, $\mathbf P_n$, $\mathbf P_n^\delta$, $\mathbf Q_{r,s}$ or $\mathbf R$, where $n\in\mathbb N$ and $1\le r,s\le 3$.
\end{itemize}
\end{theorem}

A variety of semigroups [monoids] is called \emph{completely regular} if it consists of \emph{completely regular} semigroups [monoids], that is, unions of groups. The second main result of the article is the following

\begin{theorem}
\label{Th: almost fi-perm}
A variety of monoids $\mathbf V$ is almost $fi$-permutable if and only if $\mathbf V$ either is a completely regular variety or is contained in one of the varieties listed in the item~(iii) of Theorem~\ref{Th: fi-perm}.
\end{theorem}

As we have mentioned above, any [almost] $fi$-permutable semigroup variety has the Arguesian subvariety lattice. It follows from results of~\cite{Vernikov-Volkov-97,Vernikov-Volkov-00} that, beyond the completely simple [completely regular] case, subvariety lattices of such varieties possesses the much stronger distributive law. Results of the present article imply the following analogs of these claims. 

\begin{corollary}
\label{Cor: fi-perm distributive}
The subvariety lattice of any non-group $fi$-permutable variety of monoids is distributive.
\end{corollary}

\begin{corollary}
\label{Cor: almost fi-perm distributive}
The subvariety lattice of any non-completely regular almost $fi$-permutable variety of monoids is distributive.
\end{corollary}

The claims that the varieties $\mathbf D_\infty\vee\mathbf N$, $\mathbf D_\infty\vee\mathbf N^\delta$, $\mathbf P_n$, $\mathbf P_n^\delta$, $\mathbf Q_{r,s}$ and $\mathbf R$ with $n\in\mathbb N$ and $1\le r,s\le 3$ have distributive subvariety lattices are new. Moreover, we find below some monoid variety that properly contains $\mathbf Q_{r,s}$ and $\mathbf R$ and has the distributive subvariety lattice (see Proposition~\ref{Prop: L(A') is distributive} and proof of Corollaries~\ref{Cor: fi-perm distributive} and~\ref{Cor: almost fi-perm distributive} given in Section~\ref{Sec: proof of main results}). All these facts are of some independent interest because only a few examples of monoid varieties with the distributive subvariety lattice are known so far.

The main result of the article~\cite{Vernikov-Volkov-00} shows that the class of almost $fi$-permutable semigroup varieties is not closed under taking of subvarieties. This fact contrasts with the following assertion which immediately follows from Theorem~\ref{Th: almost fi-perm}.

\begin{corollary}
\label{Cor: almost fi-perm hereditary}
Every subvariety of any almost $fi$-permutable monoid variety is almost $fi$-permutable variety too.\qed
\end{corollary}

One more immediate corollary of Theorems~\ref{Th: fi-perm} and~\ref{Th: almost fi-perm} is the following claim.

\begin{corollary}
\label{Cor: almost fi-perm = fi-perm without cr}
If a non-completely regular monoid variety is almost $fi$-permutable, then it is $fi$-permutable.\qed
\end{corollary}

Results of the articles~\cite{Vernikov-Volkov-97} and~\cite{Vernikov-Volkov-00} show that the analog of the last claim for semigroup varieties is not the case.

The article consists of seven sections. Section~\ref{Sec: preliminaries} contains definitions, notation, certain known results and its simple corollaries. In Section~\ref{Sec: auxiliary results} we prove a number of auxiliary assertions. Section~\ref{Sec: non-fi-perm} contains several examples of non-permutative fully invariant congruences on the free monoid $\mathfrak X^\ast$, while in Section~\ref{Sec: fi-perm} we prove the $fi$-permutability of several concrete monoid varieties. Section~\ref{Sec: proof of main results} is devoted to verification of Theorems~\ref{Th: fi-perm} and~\ref{Th: almost fi-perm} and Corollaries~\ref{Cor: fi-perm distributive} and~\ref{Cor: almost fi-perm distributive}. Finally, in Section~\ref{Sec: generalizations} we discuss some generalizations of the notion of $fi$-permutability.

\section{Preliminaries}
\label{Sec: preliminaries}

We start with a general remark which can be straightforwardly checked.

\begin{lemma}
\label{Lem: lifting} 
Let $\alpha,\beta$ and $\nu$ be equivalences on a set $S$ such that $\alpha,\beta\supseteq\nu$. Then $\alpha$ and $\beta$ permute if and only if the equivalences $\alpha/\nu$ and $\beta/\nu$ on the quotient set $S/\nu$ permute.\qed
\end{lemma}

Lemma~\ref{Lem: lifting} shows that, when studying permuting fully invariant congruences, we may consider congruences on the free monoid $\mathfrak X^\ast$ that contain the fully invariant congruence $\nu$ on $\mathfrak X^\ast$ corresponding to a variety $\mathbf V$ instead of congruences on the $\mathbf V$-free object $\mathfrak X^\ast/\nu$. This is convenient for it is easier to deal with elements of $\mathfrak X^\ast$ (that is, words) than with elements of an arbitrary free objects of monoid varieties.

The well-known result by J\'onsson~\cite{Jonsson-53} (see also~\cite[Theorem~410]{Gratzer-11}, for instance) immediately implies the following

\begin{lemma}
\label{Lem: fi-perm implies modularity}
Every $fi$-permutable monoid variety has a modular and moreover, Arguesian subvariety lattice.\qed
\end{lemma}

We denote the empty word by $\lambda$. The \emph{content} of a word $\mathbf w$, that is, the set of all letters occurring in $\mathbf w$ is denoted by $\con(\mathbf w)$. The following assertion is well known (see~\cite[Lemma~2.1]{Gusev-Vernikov-18}, for instance).

\begin{lemma}
\label{Lem: group variety}
Let $\mathbf V$ be a monoid variety. The following are equivalent:
\begin{itemize}
\item[a)] $\mathbf V$ is a group variety;
\item[b)] $\mathbf V$ satisfies an identity $\mathbf u\approx\mathbf v$ with $\con(\mathbf u)\ne\con(\mathbf v)$;
\item[c)] $\mathbf{SL}\nsubseteq\mathbf V$.\qed
\end{itemize}
\end{lemma}

The following fact is well-known. It was explicitly noted, for example, in~\cite[Proposition~2.1]{Gusev-18} or~\cite[Subsection~1.1]{Jackson-Lee-18}.

\begin{lemma}
\label{Lem: embedding}
The map from the lattice $\mathbb{MON}$ of all monoid varieties to the lattice $\mathbb{SEM}$ of all semigroup varieties that maps a monoid variety generated by a monoid $M$ to the semigroup variety generated by the semigroup reduct of $M$ is an embedding from $\mathbb{MON}$ to $\mathbb{SEM}$.\qed
\end{lemma}

The fully invariant congruence on the free monoid $\mathfrak X^\ast$ corresponding to a variety $\mathbf V$ will be denoted by $\theta_{\mathbf V}$. It is evident that congruences $\alpha$ and $\beta$ on an algebra $A$ permute whenever they are comparable in the congruence lattice of $A$. Therefore, the following claim is true.

\begin{lemma}
\label{Lem: comparable is fi-permut}
If the varieties of monoids $\mathbf X$ and $\mathbf Y$ are comparable in the lattice $\mathbb{MON}$, then the congruences $\theta_{\mathbf X}$ and $\theta_{\mathbf Y}$ permute.\qed
\end{lemma}

\begin{lemma}
\label{Lem: cr is fi-perm}
If $\mathbf X$ and $\mathbf Y$ are completely regular monoid varieties and $\mathbf X,\mathbf Y\supseteq\mathbf{SL}$, then the congruences $\theta_{\mathbf X}$ and $\theta_{\mathbf Y}$ permute.
\end{lemma}

\begin{proof}
Let $(\mathbf u,\mathbf v)\in\theta_{\mathbf X}\theta_{\mathbf Y}$. Then $\mathbf u\,\theta_{\mathbf X}\,\mathbf w\,\theta_{\mathbf Y}\,\mathbf v$ for some word $\mathbf w$. Now Lemma~\ref{Lem: group variety} applies with the conclusion that $\con(\mathbf u)=\con(\mathbf w)=\con(\mathbf v)$. Therefore, either $\mathbf u=\mathbf v=\mathbf w=\lambda$ or all the words $\mathbf u$, $\mathbf v$ and $\mathbf w$ are non-empty. Clearly, $(\mathbf u,\mathbf v)\in\theta_{\mathbf Y}\theta_{\mathbf X}$ in the first case. Let us consider the second one.

Let $\varphi$ be the embedding from $\mathbb{MON}$ to $\mathbb{SEM}$ mentioned in Lemma~\ref{Lem: embedding}. Put $\mathbf X^\prime=\varphi(\mathbf X)$ and $\mathbf Y^\prime=\varphi(\mathbf Y)$. Let $\alpha$ and $\beta$ denote the fully invariant congruences on the free semigroup $\mathfrak X^+$ corresponding to the varieties $\mathbf X^\prime$ and $\mathbf Y^\prime$, respectively. Since every monoid satisfies the identities $x\approx 1\cdot x\approx x\cdot 1$ and the words $\mathbf u$, $\mathbf v$ and $\mathbf w$ are non-empty, we may assume that these words do not contain the symbol of \mbox{0-ary} operation. Hence the identities $\mathbf u\approx\mathbf w$ and $\mathbf w\approx\mathbf v$ hold in the varieties $\mathbf X^\prime$ and $\mathbf Y^\prime$, respectively. Thus $\mathbf u\,\alpha\,\mathbf w\,\beta\,\mathbf v$. It is verified independently by Pastijn~\cite{Pastijn-91} and Petrich and Reilly~\cite{Petrich-Reilly-90} that each completely regular semigroup variety is almost $fi$-permutable. Then the congruences $\alpha$ and $\beta$ permute, whence there is a word $\mathbf w^\prime\in\mathfrak X^+$ with $\mathbf u\,\beta\,\mathbf w^\prime\,\alpha\,\mathbf v$. Therefore, the identities $\mathbf u\approx\mathbf w^\prime$ and $\mathbf w^\prime\approx\mathbf v$ hold in the varieties $\mathbf Y$ and $\mathbf X$, respectively. Thus, $(\mathbf u,\mathbf v)\in\theta_{\mathbf Y}\theta_{\mathbf X}$, and we are done.
\end{proof}

A word $\mathbf w$ is called an \emph{isoterm} for a class of monoids if no monoid in the class satisfies any non-trivial identity of the form $\mathbf w\approx\mathbf w^\prime$. It is evident that if a word $\mathbf u$ is an isoterm for a monoid variety $\mathbf V$, then every subword of $\mathbf u$ is an isoterm for $\mathbf V$ too. Below we will use this fact many times, as a rule, without mentioning it explicitly. The following statement explains an important role that monoids of the form $S(W)$ play.
 
\begin{lemma}[{Jackson~\cite[Lemma~3.3]{Jackson-05}}]
\label{Lem: S(W) in V}
Let $\mathbf V$ be a monoid variety and $W$ be a set of words. Then $S(W)$ lies in $\mathbf V$ if and only if each word in $W$ is an isoterm for $\mathbf V$.\qed
\end{lemma}

For any $n\ge 2$, we put $\mathbf C_n=\var\{x^n\approx x^{n+1},\,xy\approx yx\}$. The following statement is well-known. It readily follows from~\cite[Corollary~6.1.5]{Almeida-94}, for instance.

\begin{lemma}
\label{Lem: x^n is isoterm}
Let $n\in\mathbb N$. For a monoid variety $\mathbf V$, the following are equivalent:
\begin{itemize}
\item[a)] $x^n$ is not an isoterm for $\mathbf V$;
\item[b)] $\mathbf V$ satisfies the identity
\begin{equation}
\label{x^n=x^m}
x^n\approx x^m
\end{equation}
for some $m>n$;
\item[c)] $\mathbf C_{n+1}\nsubseteq\mathbf V$.\qed
\end{itemize}
\end{lemma}

It is well known that a monoid variety is completely regular if and only if it satisfies an identity of the form 
\begin{equation}
\label{x=x^{n+1}}
x\approx x^{n+1}
\end{equation}
for some $n\in\mathbb N$. Then Lemma~\ref{Lem: x^n is isoterm} implies

\begin{corollary}[{Gusev and Vernikov~\cite[Corollary~2.6]{Gusev-Vernikov-18}}]
\label{Cor: cr}
A variety of monoids $\mathbf V$ is completely regular if and only if $\mathbf C_2\nsubseteq\mathbf V$.\qed
\end{corollary}

A variety of monoids is called \emph{aperiodic} if all its groups are singletons. It is well known that a variety is aperiodic if and only if it satisfies an identity of the form 
\begin{equation}
\label{x^n=x^{n+1}}
x^n\approx x^{n+1}
\end{equation}
for some $n\in\mathbb N$. 

\begin{lemma}
\label{Lem: x^n=x^m in X wedge Y}
Let $\mathbf X$ and $\mathbf Y$ be aperiodic monoid varieties. If $\mathbf X\wedge\mathbf Y$ satisfies a non-trivial identity of the form~\eqref{x^n=x^m}, then this identity holds in either $\mathbf X$ or $\mathbf Y$.
\end{lemma}

\begin{proof}
We may assume without loss of generality that $n<m$. Lemma~\ref{Lem: x^n is isoterm} implies that one of the varieties $\mathbf X$ and $\mathbf Y$, say, $\mathbf X$ satisfies an identity of the form $x^n\approx x^r$ for some $r>n$. But the variety $\mathbf X$ is aperiodic, whence it satisfies the identity~\eqref{x^n=x^{n+1}} and therefore, the identity~\eqref{x^n=x^m}.
\end{proof}
 
For any $k\in\mathbb N$, we fix notation for the following identity:
$$
\delta_k:\enskip xt_1xt_2x\cdots t_kx\approx x^2t_1t_2\cdots t_k.
$$
For convenience, we denote by $\delta_\infty$ the trivial identity. 

\begin{lemma}[{Lee~\cite[proof of Proposition~4.1]{Lee-14}}]
\label{Lem: basis for D_k}
Let $k\in\mathbb N\cup\{\infty\}$. Then 
$$
\mathbf D_k=\var\{x^2\approx x^3,\,x^2y\approx yx^2,\,\sigma_1,\,\sigma_2,\,\sigma_3,\,\delta_k\}.\eqno{\qed}
$$
\end{lemma}

\begin{lemma}
\label{Lem: does not contain D_{k+1}}
Let $\mathbf V$ be a monoid variety. If $\mathbf D_{k+1}\nsubseteq\mathbf V$ for some $k\in\mathbb N$, then $\mathbf V$ satisfies an identity of the form
\begin{equation}
\label{...t_ix...=...t_ix^{e_i}...}
xt_1xt_2x\cdots t_kx\approx x^{e_0}t_1x^{e_1}t_2x^{e_2}\cdots t_kx^{e_k},
\end{equation}
where $e_0,e_1,\dots,e_k\in\mathbb N_0$ and $e_i>1$ for some $i\in\{0,1,\dots,k\}$.
\end{lemma}

\begin{proof}
If $\mathbf V$ is non-completely regular, then this claim is proved in Gusev and Vernikov~\cite[Lemma~2.15]{Gusev-Vernikov-18}. Suppose now that $\mathbf V$ is completely regular. Then $\mathbf V$ satisfies the identity~\eqref{x=x^{n+1}} for some $n\in\mathbb N$. Hence the identity $xt_1xt_2x\cdots t_kx\approx x^{n+1}t_1xt_2x\cdots t_kx$ holds in $\mathbf V$.
\end{proof}

Put $\mathbf E=\var\{x^2\approx x^3,\,x^2y^2\approx y^2x^2,\,\delta_1\}$. 

\begin{lemma}
\label{Lem: does not contain E}
Let $\mathbf V$ be a monoid variety satisfying the identity~\eqref{x^n=x^{n+1}} with $n\ge 2$. If $\mathbf C_2\subseteq\mathbf V$ and $\mathbf E\nsubseteq\mathbf V$, then $\mathbf V$ satisfies the identity
\begin{equation}
\label{x^nyx^n=yx^n}
x^nyx^n\approx yx^n.
\end{equation}
\end{lemma}

\begin{proof}
If $\mathbf D_1\nsubseteq\mathbf V$, then $\mathbf V$ is commutative by Gusev and Vernikov~\cite[Lemma~2.14]{Gusev-Vernikov-18}. Then $\mathbf V$ satisfies the identities $x^nyx^n\approx yx^{2n}\approx yx^n$ and therefore, the identity~\eqref{x^nyx^n=yx^n}. Suppose now that $\mathbf D_1\subseteq\mathbf V$. Then $\mathbf V$ satisfies an identity of the form $x^pyx^q\approx yx^r$ for some $p,q\in\mathbb N$ and $r\ge 2$ by~\cite[Lemma~4.1 and Proposition~4.2]{Gusev-Vernikov-18}. One can substitute $x^n$ for $x$ in this identity and apply the identity~\eqref{x^n=x^{n+1}}. As a result, we obtain the identity~\eqref{x^nyx^n=yx^n}.
\end{proof}

The following statement immediately follows from Gusev and Vernikov~\cite[Proposition~4.2]{Gusev-Vernikov-18}.

\begin{lemma}
\label{Lem: yx^n=w in E}
If the variety $\mathbf E$ satisfies an identity of the form $yx^k\approx\mathbf w$ with $k\ge 2$, then $\mathbf w=yx^\ell$ for some $\ell\ge 2$.\qed
\end{lemma}

Put $\mathbf L=\var S(xtxysy)$ and $\mathbf M=\var S(xytxsy)$.

\begin{lemma}[{Gusev and Vernikov~\cite[Lemma~4.9]{Gusev-Vernikov-18}}]
\label{Lem: V contains D_2 but does not contain L or M or M^delta}
Let $\mathbf V$ be a monoid variety with $\mathbf D_2\subseteq\mathbf V$.
\begin{itemize}
\item[(i)] If $\mathbf M\nsubseteq\mathbf V$, then $\mathbf V$ satisfies the identity $\sigma_1$.
\item[(ii)] If $\mathbf M^\delta\nsubseteq\mathbf V$, then $\mathbf V$ satisfies the identity $\sigma_2$.
\item[(iii)] If $\mathbf L\nsubseteq\mathbf V$, then $\mathbf V$ satisfies the identity $\sigma_3$.\qed
\end{itemize}
\end{lemma}

Put $\mathbf O=\var\{\sigma_2,\,\sigma_3\}$. Following Lee~\cite{Lee-13}, we call an identity of the form
$$
\mathbf u_0\biggl(\prod_{i=1}^r t_i\mathbf u_i\biggr)\approx\mathbf v_0\biggl(\prod_{i=1}^r t_i\mathbf v_i\biggr),
$$
where $\{t_0,t_1,\dots,t_r\}=\simple\bigl(\prod_{i=0}^r t_i\mathbf u_i\bigr)=\simple\bigl(\prod_{i=0}^r t_i\mathbf v_i\bigr)$ \emph{efficient} if $\mathbf u_i\mathbf v_i\ne\lambda$ for any $i=0,1,\dots,r$. 

\begin{lemma}[{Lee~\cite[Lemma~8 and Remark~11]{Lee-13}}] 
\label{Lem: subvarieties of O}
Every non-commutative subvariety of the variety $\mathbf O$ can be given within $\mathbf O$ by a finite number of efficient identities of the form either
\begin{equation}
\label{one letter in a block}
x^{e_0}\biggl(\prod_{i=1}^r t_ix^{e_i}\biggr)\approx x^{f_0}\biggl(\prod_{i=1}^r t_ix^{f_i}\biggr),
\end{equation}
where $r,e_0,f_0,e_1,f_1,\dots,e_r,f_r\in\mathbb N_0$ or 
\begin{equation}
\label{two letters in a block}
x^{e_0}y^{f_0}\biggl(\prod_{i=1}^r t_ix^{e_i}y^{f_i}\biggr)\approx y^{f_0}x^{e_0}\biggl(\prod_{i=1}^r t_ix^{e_i}y^{f_i}\biggr),
\end{equation}
where $r\in\mathbb N_0$, $e_0,f_0\in\mathbb N$, $e_1,f_1,\dots,e_r,f_r\in\mathbb N_0$, $\sum_{i=0}^r e_i\ge 2$ and $\sum_{i=0}^r f_i\ge 2$.\qed
\end{lemma}

For a word $\mathbf w$ and a letter $x$, let $\occ_x(\mathbf w)$ denote the number of occurrences of $x$ in $\mathbf w$. A letter $x$ is called \emph{simple} [\emph{multiple}] \emph{in a word} $\mathbf w$ if $\occ_x(\mathbf w)=1$ [respectively, $\occ_x(\mathbf w)>1$]. The set of all simple [multiple] letters of a word $\mathbf w$ is denoted by $\simple(\mathbf w)$ [respectively, $\mul(\mathbf w)$]. Let $\mathbf w$ be a word and $\simple(\mathbf w)=\{t_1,t_2,\dots,t_m\}$. We will assume without loss of generality that $\mathbf w(t_1,t_2,\dots,t_m)=t_1t_2\cdots t_m$. Then $\mathbf w=\mathbf w_0t_1\mathbf w_1\cdots t_m\mathbf w_m$ for some words $\mathbf w_0,\mathbf w_1,\dots,\mathbf w_m$. The words $\mathbf w_0$, $\mathbf w_1$, \dots, $\mathbf w_m$ are called \emph{blocks} of the word $\mathbf w$. The representation of the word $\mathbf w$ as a product of alternating simple in $\mathbf w$ letters and blocks is called a \emph{decomposition} of the word $\mathbf w$.

The following statement follows from Lemma~\ref{Lem: S(W) in V} and the definition of the variety $\mathbf D_1$.

\begin{lemma}
\label{Lem: decompositions of u and v}
Let $\mathbf u\approx\mathbf v$ be an identity that holds in the variety $\mathbf D_1$. Suppose that 
\begin{equation}
\label{decomposition of u}
\mathbf u_0t_1\mathbf u_1\cdots t_m\mathbf u_m
\end{equation}
is the decomposition of the word $\mathbf u$. Then $\con(\mathbf u)=\con(\mathbf v)$ and the decomposition of the word $\mathbf v$ has the form
\begin{equation}
\label{decomposition of v}
\mathbf v_0t_1\mathbf v_1\cdots t_m\mathbf v_m
\end{equation}
for some words $\mathbf v_0,\mathbf v_1,\dots,\mathbf v_m$.\qed
\end{lemma}

As usual, the subvariety lattice of a variety $\mathbf V$ is denoted by $L(\mathbf V)$. 

\begin{lemma}[{Gusev and Vernikov~\cite[Proposition~6.1]{Gusev-Vernikov-18}}]
\label{Lem: L(K) is a chain}
The lattice $L(\mathbf K)$ is a chain.\qed
\end{lemma}

Let $\mathbf V\Sigma$ denote the variety given by the identity system $\Sigma$ within the variety $\mathbf V$.

\begin{lemma}
\label{Lem: smth imply distributivity}
Let $\mathbf V$ be a variety of algebras and $\mathbf W\subseteq\mathbf V$. Suppose that there is an identity system $\Sigma$ such that:
\begin{itemize}
\item[(i)] if $\mathbf W\subseteq\mathbf U\subseteq\mathbf V$, then $\mathbf U=\mathbf V\Phi$ for some identity system $\Phi\subseteq\Sigma$;
\item[(ii)] if $\mathbf U,\mathbf U^\prime\in[\mathbf W,\mathbf V]$ and $\mathbf U\wedge\mathbf U^\prime$ satisfies an identity $\sigma\in\Sigma$, then $\sigma$ holds in either $\mathbf U$ or $\mathbf U^\prime$.
\end{itemize}
Then the interval $[\mathbf W,\mathbf V]$ of the lattice $L(\mathbf V)$ is distributive.
\end{lemma}

\begin{proof}
Arguing by contradiction, we suppose that there are varieties $\mathbf X,\mathbf Y,\mathbf Z\in[\mathbf W,\mathbf V]$ such that the sublattice $L$ of the interval $[\mathbf W,\mathbf V]$ generated by $\mathbf X$, $\mathbf Y$ and $\mathbf Z$ is one of the two 5-element non-distributive lattices. We may assume without loss of generality that the varieties $\mathbf X$ and $\mathbf Y$ are atoms in the lattice $L$ and either $\mathbf Z$ also is an atom in $L$ or $\mathbf Y\subset\mathbf Z$. In either case there is an identity $\sigma$ that holds in $\mathbf Y$ and fails in $\mathbf Z$. By the claim~(i), we may assume that $\sigma\in\Sigma$. The identity $\sigma$ holds in the variety $\mathbf X\wedge\mathbf Y=\mathbf X\wedge\mathbf Z$. Since this identity fails in $\mathbf Z$, the claim~(ii) implies that it holds in $\mathbf X$. Therefore, $\sigma$ holds in $\mathbf X\vee\mathbf Y=\mathbf X\vee\mathbf Z$. But this contradicts the fact that $\sigma$ fails in $\mathbf Z$.
\end{proof}

\section{Auxiliary results}
\label{Sec: auxiliary results}

\subsection{Linear-balanced identities}
\label{Subsec: linear-balanced}
If $\mathbf w$ is a word and $X\subseteq\con(\mathbf w)$, then we denote by $\mathbf w(X)$ the word obtained from $\mathbf w$ by deleting all letters except letters from $X$. If $X=\{x_1,x_2,\dots,x_k\}$, then we write $\mathbf w(x_1,x_2,\dots,x_k)$ rather than $\mathbf w(\{x_1,x_2,\dots,x_k\})$. If $X\subseteq\con(\mathbf w)$ and $x_1,x_2,\dots,x_k\in\con(\mathbf w)\setminus X$, then we write $\mathbf w(x_1,x_2,\dots,x_k,X)$ instead of $\mathbf w(X\cup\{x_1,x_2,\dots,x_k\})$. Clearly, if $\mathbf u$ and $\mathbf v$ are words with $\con(\mathbf u)=\con(\mathbf v)$ and $X\subseteq\con(\mathbf u)$, then the identity $\mathbf u\approx\mathbf v$ implies the identity $\mathbf u(X)\approx\mathbf v(X)$. We will use this fact throughout the rest of the article many times without explicitly specifying this.

A non-empty word $\mathbf w$ is called \emph{linear} if $\occ_x(\mathbf w)\le 1$ for each letter $x$. Let $\mathbf u$ and $\mathbf v$ be words and~\eqref{decomposition of u} and~\eqref{decomposition of v} be decompositions of $\mathbf u$ and $\mathbf v$, respectively. A letter $x$ is called \emph{linear-balanced in the identity} $\mathbf u\approx\mathbf v$ if $x$ is multiple in $\mathbf u$ and $\occ_x(\mathbf u_i)=\occ_x(\mathbf v_i)\le 1$ for all $i=0,1,\dots,m$; the identity $\mathbf u\approx\mathbf v$ is called \emph{linear-balanced} if any letter $x\in\mul(\mathbf u)\cup\mul(\mathbf v)$ is linear-balanced in this identity. 

\begin{lemma}
\label{Lem: xt_1x...t_kx is isoterm}
Let $\mathbf V$ be a monoid variety such that the word $\bigl(\prod_{i=1}^k xt_i\bigr)x$ is an isoterm for $\mathbf V$, $\mathbf u$ be a word such that all its blocks are linear words and $\occ_x(\mathbf u)\le k+1$ for every letter $x$. Then every identity of the form $\mathbf u\approx\mathbf v$ that holds in the variety $\mathbf V$ is linear-balanced.
\end{lemma}

\begin{proof}
Suppose that an identity $\mathbf u\approx\mathbf v$ holds in $\mathbf V$. Let~\eqref{decomposition of u} be the decomposition of $\mathbf u$. Lemma~\ref{Lem: S(W) in V} implies that $\mathbf D_1\subseteq\mathbf V$. Then the decomposition of $\mathbf v$ has the form~\eqref{decomposition of v} and $\mul(\mathbf u)=\mul(\mathbf v)$ by Lemma~\ref{Lem: decompositions of u and v}. Let $x\in\mul(\mathbf u)$ and $r=\occ_x(\mathbf u)$. Then $x$ occurs in exactly $r$ blocks of $\mathbf u$, say, in blocks $\mathbf u_{i_1},\mathbf u_{i_2},\dots,\mathbf u_{i_r}$ with $i_1<i_2<\cdots<i_r$. Let $T=\{t_{i_1+1},t_{i_2+1},\dots,t_{i_{r-1}+1}\}$. Clearly, $\mathbf u(x,T)=\bigl(\prod_{s=1}^{r-1} xt_{i_s+1}\bigr)x$. Since $r\le k+1$ and the word $\bigl(\prod_{i=1}^k xt_i\bigr)x$ is an isoterm for $\mathbf V$, the word $\mathbf u(x,T)$ is an isoterm for $\mathbf V$ too. Therefore, $\mathbf u(x,T)=\mathbf v(x,T)$. It follows that $\occ_x(\mathbf v)=r$ and $x$ occurs at most once in each block of $\mathbf v$. Suppose that $x$ occurs in blocks $\mathbf v_{j_1},\mathbf v_{j_2},\dots,\mathbf v_{j_r}$ of $\mathbf v$ and $j_1<j_2<\cdots<j_r$. If $i_1<j_1$, then $\mathbf u(x,T)\ne\mathbf v(x,T)$ because the first occurrence of $x$ is preceded by $t_{i_1+1}$ in $\mathbf v$. We have a contradiction again. Therefore, $j_1\le i_1$. By symmetry, $i_1\le j_1$, whence $i_1=j_1$. Analogous considerations show that $i_s=j_s$ for all $s=2,3,\dots,r$. This means that the letter $x$ is linear-balanced in the identity $\mathbf u\approx\mathbf v$.
\end{proof}

\subsection{The variety $\mathbf A$ and its subvarieties}
\label{Subsec: A}
Put 
$$
\mathbf A=\var\{x^2\approx x^3,\,x^2y\approx yx^2\}.
$$ 
The following assertion follows from Lemma~3.3 of the work by Lee~\cite{Lee-14}, its proof and Lemma~4.2 of the same work.

\begin{lemma}
\label{Lem: Straubing identities in A}
Let $\mathbf u\approx\mathbf v$ be a non-trivial identity of the form~\eqref{one letter in a block} with $r,e_0,f_0,e_1$, $f_1,\dots,e_r,f_r\in\mathbb N_0$, $\sum_{i=0}^r e_i\ge 2$ and $\sum_{i=0}^r f_i\ge 2$. Then
\begin{itemize}
\item[(i)] if $e_i,f_j>1$ for some $0\le i,j\le r$, then the identity $\mathbf u\approx\mathbf v$ holds in $\mathbf A$;
\item[(ii)] if $e_0,e_1,\dots,e_r\le 1$ and $e=\sum_{i=0}^r e_i$, then $\mathbf A\{\mathbf u\approx\mathbf v\}=\mathbf A\{\delta_{e-1}\}$.\qed
\end{itemize}
\end{lemma}

We denote the trivial variety of monoids by $\mathbf T$.

\begin{lemma}
\label{Lem: L(A)}
\quad
\begin{itemize}
\item[(i)] The lattice $L(\mathbf A)$ is a set-theoretical union of the chain $\mathbf T\subset\mathbf{SL}\subset\mathbf C_2\subset\mathbf D_1\subset\mathbf D_2$ and the interval $[\mathbf D_2,\mathbf A]$.
\item[(ii)] The interval $[\mathbf D_2,\mathbf A]$ is a disjoint union of intervals of the form $[\mathbf D_k,\mathbf A\{\delta_k\}]$, where $2\le k\le\infty$.
\end{itemize}
\end{lemma}

\begin{proof}
(i) The lattice $L(\mathbf D_2)$ is the chain $\mathbf T\subset\mathbf{SL}\subset\mathbf C_2\subset\mathbf D_1\subset\mathbf D_2$ by Jackson~\cite[Fig.~1]{Jackson-05}. Let $\mathbf V$ be a subvariety of $\mathbf A$ with $\mathbf D_2\nsubseteq\mathbf V$. We have to check that $\mathbf V\subseteq\mathbf D_2$. In view of Lemma~\ref{Lem: does not contain D_{k+1}}, the variety $\mathbf V$ satisfies an identity of the form $xyx\approx x^kyx^\ell$, where either $k\ge 2$ or $\ell\ge 2$. Since the identities
\begin{align}
\label{xx=xxx}
x^2&\approx x^3,\\
\label{xxy=yxx}
x^2y&\approx yx^2
\end{align}
hold in the variety $\mathbf V$, this variety satisfies the identities $x^2y\approx xyx\approx yx^2$. Clearly, these identities imply the identities $\sigma_1$, $\sigma_2$, $\sigma_3$ and $\delta_1$. Now we can apply Lemma~\ref{Lem: basis for D_k} and conclude that $\mathbf X\subseteq\mathbf D_1\subset\mathbf D_2$.

\smallskip

(ii) Let $\mathbf V\in[\mathbf D_2,\mathbf A]$. If $\mathbf D_\infty\subseteq\mathbf V$, then $\mathbf V\in[\mathbf D_\infty,\mathbf A\{\delta_\infty\}]$. Otherwise, there is a natural number $k\ge 2$ such that $\mathbf D_k\subseteq\mathbf V$ but $\mathbf D_{k+1}\nsubseteq\mathbf V$. Now Lemma~\ref{Lem: does not contain D_{k+1}} applies with the conclusion that $\mathbf V$ satisfies an identity of the form~\eqref{...t_ix...=...t_ix^{e_i}...}, where $e_i>1$ for some $i$. Then $\mathbf A\{\eqref{...t_ix...=...t_ix^{e_i}...}\}=\mathbf A\{\delta_k\}$ by Lemma~\ref{Lem: Straubing identities in A}(ii). Hence $\mathbf V\in[\mathbf D_k,\mathbf A\{\delta_k\}]$.
\end{proof}

\begin{corollary}
\label{Cor: delta_k in X wedge Y}
Let $\mathbf X$ and $\mathbf Y$ be subvarieties of the variety $\mathbf A$. If $\mathbf X\wedge\mathbf Y$ satisfies the identity $\delta_k$ for some $2\le k\le\infty$, then this identity is true in either $\mathbf X$ or $\mathbf Y$.
\end{corollary}

\begin{proof}
Lemma~\ref{Lem: L(A)}(i) allows us to assume that $\mathbf X,\mathbf Y\in[\mathbf D_2,\mathbf A]$. Then Lemma~\ref{Lem: L(A)}(ii) implies that there are $r,s$ such that $2\le r,s\le\infty$, $\mathbf X\in[\mathbf D_r,\mathbf A\{\delta_r\}]$ and $\mathbf Y\in[\mathbf D_s,\mathbf A\{\delta_s\}]$. We may assume without loss of generality that $r\le s$. Then $\mathbf X\wedge\mathbf Y\in[\mathbf D_r,\mathbf A\{\delta_r\}]$. If $\mathbf X\wedge\mathbf Y$ satisfies the identity $\delta_k$, then $k\ge r$ and therefore, $\delta_k$ holds in $\mathbf X$.
\end{proof}

If $\mathbf w$ is a word and $X\subseteq\con(\mathbf w)$, then we denote by $\mathbf w_X$ the word obtained from $\mathbf w$ by deleting all letters from $X$. If $X=\{x\}$, then we write $\mathbf w_x$ rather than $\mathbf w_{\{x\}}$. Clearly, if $\mathbf u$ and $\mathbf v$ are words with $\con(\mathbf u)=\con(\mathbf v)$ and $X\subseteq\con(\mathbf u)$, then the identity $\mathbf u\approx\mathbf v$ implies the identity $\mathbf u_X\approx\mathbf v_X$. We will use this fact throughout the rest of the article many times without explicitly specifying this.

\begin{lemma}
\label{Lem: pxyq=pyxq}
Let $\mathbf V\subseteq\mathbf A$. If $\mathbf V$ does not contain the monoid $S(\mathbf pxy\mathbf q)$, where 
\begin{equation}
\label{pxyq=pyxq restrict}
\begin{array}{l}
\mathbf p=a_1t_1\cdots a_kt_k\ \text{and}\ \mathbf q=t_{k+1}a_{k+1}\cdots t_{k+\ell}a_{k+\ell}\ \text{for some}\ k,\ell\in\mathbb N_0\\
\text{and}\ a_1,a_2,\dots,a_{k+\ell}\ \text{are letters such that}\ \{a_1,a_2,\dots,a_{k+\ell}\}=\{x,y\},
\end{array}
\end{equation}
then $\mathbf V$ satisfies the identity
\begin{equation}
\label{pxyq=pyxq}
\mathbf pxy\mathbf q\approx\mathbf pyx\mathbf q.
\end{equation}
\end{lemma}

\begin{proof}
It is clear that the identity~\eqref{pxyq=pyxq} holds in the monoid $S(xyx)$ and therefore, in the variety $\mathbf D_2$. In view of Lemma~\ref{Lem: L(A)}(i), we may assume that $\mathbf D_2\subseteq\mathbf V$. Put $\mathbf u=\mathbf pxy\mathbf q$. We note that $\simple(\mathbf u)=\{t_1,t_2,\dots,t_{k+\ell}\}$. If the word $\mathbf u_y$ is not an isoterm for $\mathbf V$, then $\mathbf V$ satisfies a non-trivial identity of the form $\mathbf u_y\approx\mathbf u^\prime$. Lemma~\ref{Lem: decompositions of u and v} implies that $\mathbf u^\prime=x^{f_0}\bigl(\prod_{i=1}^{k+\ell} t_ix^{f_i}\bigr)$ for some $f_0,f_1,\dots,f_{k+\ell}\in\mathbb N_0$. Then we can apply Lemma~\ref{Lem: Straubing identities in A}(ii) with the conclusion that $\mathbf V$ satisfies the identity $\delta_{\occ_x(\mathbf u)-1}$. This identity implies the identities $\mathbf u\approx x^2\mathbf u_x\approx\mathbf pyx\mathbf q$, and we are done. Analogous considerations show that if the word $\mathbf u_x$ is not an isoterm for $\mathbf V$, then this variety satisfies the identity~\eqref{pxyq=pyxq}. 

Finally, suppose that both the words $\mathbf u_x$ and $\mathbf u_y$ are isoterms for $\mathbf V$. Lemma~\ref{Lem: S(W) in V} implies that $\mathbf V$ satisfies a non-trivial identity of the form $\mathbf u\approx\mathbf v$. In view of Lemma~\ref{Lem: xt_1x...t_kx is isoterm}, the identity $\mathbf u\approx\mathbf v$ is linear-balanced. This means that $\simple(\mathbf v)=\{t_1,t_2,\dots,t_{k+\ell}\}$ and blocks of the word $\mathbf v$ (in order of their appearance from left to right) are $a_1$, $a_2$, \dots, $a_k$, $\mathbf w$, $a_{k+1}$, \dots, $a_{k+\ell}$, where $\mathbf w\in\{xy,yx\}$. Since the identity $\mathbf u\approx\mathbf v$ is non-trivial, $\mathbf w=yx$, whence $\mathbf v=\mathbf pyx\mathbf q$.
\end{proof}

\begin{corollary}
\label{Cor: pxyq=pyxq in X wedge Y}
Let $\mathbf X$ and $\mathbf Y$ be subvarieties of the variety $\mathbf A$. If $\mathbf X\wedge\mathbf Y$ satisfies the identity~\eqref{pxyq=pyxq}, where the equalities~\eqref{pxyq=pyxq restrict} hold, then this identity is true in either $\mathbf X$ or $\mathbf Y$.
\end{corollary}

\begin{proof}
Lemma~\ref{Lem: S(W) in V} implies that one of the varieties $\mathbf X$ or $\mathbf Y$, say $\mathbf X$, does not contain the monoid $S(\mathbf pxy\mathbf q)$. Then Lemma~\ref{Lem: pxyq=pyxq} applies with the conclusion that $\mathbf X$ satisfies~\eqref{pxyq=pyxq}. 
\end{proof}

\subsection{Identities of the form $\mathbf w_n[\pi,\tau]\approx\mathbf w_n^\prime[\pi,\tau]$}
\label{Subsec: w_n[pi,tau]=w_n'[pi,tau]}

\begin{lemma}
\label{Lem: w_n[pi,tau]=w_n'[pi,tau] in S(w_n[xi,eta])}
Let $n\in\mathbb N$, $\pi,\tau,\xi,\eta\in S_n$ with $\mathbf w_n[\pi,\tau]\ne\mathbf w_n[\xi,\eta]$. Then the monoid $S(\mathbf w_n[\xi,\eta])$ satisfies the identity
\begin{equation}
\label{w_n[pi,tau]=w_n'[pi,tau]}
\mathbf w_n[\pi,\tau]\approx\mathbf w_n^\prime[\pi,\tau].
\end{equation}
\end{lemma}

\begin{proof}
We are going to show that if we substitute elements of the monoid $S(\mathbf w_n[\xi,\eta])$ for letters in the identity~\eqref{w_n[pi,tau]=w_n'[pi,tau]}, then we always obtain a right equality. 
 
If we substitute~1 for $x$ in the identity~\eqref{w_n[pi,tau]=w_n'[pi,tau]}, then no matter what is substituted for the other letters, the resulting identity will be trivial, whence it will be true in $S(\mathbf w_n[\xi,\eta])$. Suppose now that we substitute some other element of the monoid $S(\mathbf w_n[\xi,\eta])$ for $x$. Then value of the word $\mathbf w_n^\prime[\pi,\tau]$ equals~0 because $\mathbf w_n[\xi,\eta]$ is square free. One can denote value of the word $\mathbf w_n[\pi,\tau]$ under the substitutions by $\overline{\mathbf w_n[\pi,\tau]}$ and verify that $\overline{\mathbf w_n[\pi,\tau]}$ equals~0 as well. Non-zero elements of the monoid $S(\mathbf w_n[\xi,\eta])$ are subwords of the word $\mathbf w_n[\xi,\eta]$. Suppose that we substitute some word $\mathbf a\in S(\mathbf w_n[\xi,\eta])\setminus\{0\}$ for $x$ in $\mathbf w_n[\pi,\tau]$. If $z_i\in\con(\mathbf a)$ for some $1\le i\le 2n$, then $\overline{\mathbf w_n[\pi,\tau]}$ equals~0 because occurrences of the letter $z_i$ in the word $\mathbf w_n[\xi,\eta]$ lie in different blocks of this word, while the both occurrences of the letter $x$ in the word $\mathbf w_n[\pi,\tau]$ lie in the same block. Further, if $t_i\in\con(\mathbf a)$ for some $1\le i\le 2n$, then $\overline{\mathbf w_n[\pi,\tau]}$ equals~0 again because $t_i\in\simple(\mathbf w_n[\xi,\eta])$, while $x\in\mul(\mathbf w_n[\pi,\tau])$. It remains to consider the case when $\mathbf a=x^k$ for some $k\in\mathbb N$. We may assume that $k=1$ because $\mathbf a$ is not a subword of $\mathbf w_n[\xi,\eta]$ otherwise. Then $\overline{\mathbf w_n[\pi,\tau]}$ is a subword of the word $\mathbf w_n[\xi,\eta]$ only whenever $\mathbf w_n[\xi,\eta]=\mathbf w_n[\pi,\tau]$. But this is not the case.
\end{proof}

The arguments arising in the proof of Lemma~\ref{Lem: w_n[pi,tau]=w_n'[pi,tau] in S(w_n[xi,eta])} are very typical. There are many places below, where it will be necessary to establish that a particular identity is true in a monoid of the form $S(W)$. We will omit the corresponding calculations, since they are very similar in concept to the proof of Lemma~\ref{Lem: w_n[pi,tau]=w_n'[pi,tau] in S(w_n[xi,eta])} and simpler than this proof.

If $\mathbf u$ and $\mathbf v$ are words and $\Sigma$ is an identity or system of identities, then we will write $\mathbf u\stackrel{\Sigma}\approx\mathbf v$ in the case when the identity $\mathbf u\approx\mathbf v$ follows from $\Sigma$.

\begin{lemma}
\label{Lem: from pxqxr to px^2qr}
Let $\mathbf V$ be a monoid variety satisfying the identities~\eqref{xxy=yxx} and~\eqref{w_n[pi,tau]=w_n'[pi,tau]} for any $n\in\mathbb N$ and $\pi,\tau\in S_n$. If $\mathbf w=\mathbf px\mathbf qx\mathbf r$ and $\con(\mathbf q)\subseteq\mul(\mathbf w)$, then $\mathbf V$ satisfies the identity $\mathbf w\approx\mathbf px^2\mathbf q\mathbf r$.
\end{lemma}

\begin{proof}
For any $k$ and $m$ with $k+m>0$ and any $\rho\in S_{k+m}$, we put
\begin{align*}
\mathbf w_{k,m}[\rho]&=\biggl(\prod_{i=1}^k z_it_i\biggr)x\biggl(\prod_{i=1}^{k+m} z_{i\rho}\biggr)x\biggl(\prod_{i=k+1}^{k+m} t_iz_i\biggr),\\[-3pt]
\mathbf w_{k,m}^\prime[\rho]&=\biggl(\prod_{i=1}^k z_it_i\biggr)x^2\biggl(\prod_{i=1}^{k+m} z_{i\rho}\biggr)\biggl(\prod_{i=k+1}^{k+m} t_iz_i\biggr).
\end{align*}
The proof of Lemma~4.4 in Gusev and Vernikov~\cite{Gusev-Vernikov-18} implies that any identity of the form 
\begin{equation}
\label{w_{k,m}[rho]=w_{k,m}'[rho]}
\mathbf w_{k,m}[\rho]\approx\mathbf w_{k,m}^\prime[\rho]
\end{equation}
follows from an identity of the form~\eqref{w_n[pi,tau]=w_n'[pi,tau]} for some $n\in\mathbb N$ and $\pi,\tau\in S_n$. Therefore, $\mathbf V$ satisfies the identity~\eqref{w_{k,m}[rho]=w_{k,m}'[rho]} for any $k,m\in\mathbb N_0$ and $\rho\in S_{k+m}$.

We may assume that $\mathbf q\ne\lambda$ because the required conclusion is evident otherwise. Further considerations are divided into two cases.

\smallskip

\emph{Case 1}: $x\notin\con(\mathbf q)$. Suppose that the word $\mathbf q$ is linear and depends on pairwise different letters $z_1,z_2,\dots, z_n$. Then $z_i\in\con(\mathbf p\mathbf r)$ for each $i=1,2,\dots,n$. For any $i=1,2,\dots,n$, we fix one occurrence of the letter $z_i$ in the word $\mathbf p\mathbf r$. We may assume without loss of generality that $z_1,z_2,\dots,z_k\in\con(\mathbf p)$ and $z_{k+1},\dots,z_{k+m}\in\con(\mathbf r)$ for some $k$ and $m$ with $k+m=n$ (if it is not the case, we can rename letters). Then $\mathbf w=\mathbf pxz_{1\rho}z_{2\rho}\cdots z_{(k+m)\rho}x\mathbf r$ for an appropriate permutation $\rho\in S_{k+m}$, where
$$
\mathbf p=\mathbf w_0\biggl(\prod_{i=1}^k z_i\mathbf w_i\biggr)\quad\text{and}\quad\mathbf r=\biggl(\prod_{i=k+1}^{k+m}\mathbf w_iz_i\biggr)\mathbf w_{k+m+1}
$$
for some words $\mathbf w_0,\mathbf w_1,\dots,\mathbf w_{k+m+1}$. Therefore, $\mathbf V$ satisfies the identity
$$
\mathbf w\stackrel{\eqref{w_{k,m}[rho]=w_{k,m}'[rho]}}\approx\mathbf px^2z_{1\rho}z_{2\rho}\cdots z_{(k+m)\rho}\mathbf r,
$$
and we are done.

It remains to consider the case when the word $\mathbf q$ is non-linear. Then there is a letter $y_1\in\con(\mathbf q)$ such that $\mathbf q=\mathbf v_1y_1\mathbf v_2y_1\mathbf v_3$, where $y_1\notin\con(\mathbf v_2)$ and the word $\mathbf v_2$ is either empty or linear. Then the same arguments as in the previous paragraph show that $\mathbf V$ satisfies the identities
$$
\mathbf w=\mathbf px\mathbf v_1y_1\mathbf v_2y_1\mathbf v_3x\mathbf r\stackrel{\eqref{w_{k,m}[rho]=w_{k,m}'[rho]}}\approx\mathbf px\mathbf v_1y_1^2\mathbf v_2\mathbf v_3x\mathbf r\stackrel{\eqref{xxy=yxx}}\approx\mathbf py_1^2x\mathbf v_1\mathbf v_2\mathbf v_3x\mathbf r.
$$
In other words, we may delete the letter $y_1$ from the word $\mathbf q$. Repeating these considerations, we may delete from $\mathbf q$ all multiple letters. As a result, we obtain that $\mathbf V$ satisfies the identity $\mathbf w\approx\mathbf py_1^2y_2^2\cdots y_\ell^2x\mathbf q^\prime x\mathbf r$ for some letters $y_1,y_2,\dots,y_\ell$ and some word $\mathbf q^\prime$ that is either empty or linear. Then we may repeat considerations from the previous paragraph and conclude that $\mathbf V$ satisfies the identities
$$
\mathbf w\approx\mathbf py_1^2y_2^2\cdots y_\ell^2x\mathbf q^\prime x\mathbf r\approx\mathbf py_1^2y_2^2\cdots y_\ell^2x^2\mathbf q^\prime\mathbf r\stackrel{\eqref{xxy=yxx}}\approx\mathbf px^2y_1^2y_2^2\cdots y_\ell^2\mathbf q^\prime\mathbf r.
$$
It remains to return the letters $y_1,y_2,\dots,y_\ell$ to their original places using the identities~\eqref{xxy=yxx} and~\eqref{w_{k,m}[rho]=w_{k,m}'[rho]}. As a result, we obtain that $\mathbf V$ satisfies the identity $\mathbf w\approx\mathbf px^2\mathbf q\mathbf r$.

\smallskip

\emph{Case 2}: $x\in\con(\mathbf q)$. Then $\mathbf q=\mathbf q_0\prod_{i=1}^r(x\mathbf q_i)$, where $x\notin\con(\mathbf q_0\mathbf q_1\cdots\mathbf q_r)$. The same arguments as in Case~1 implies that we can step by step swap $x$ and $\mathbf q_r$, $\mathbf q_{r-1}$, \dots, $\mathbf q_0$. As a result, we get that $\mathbf V$ satisfies the identities
\begin{align*}
\mathbf w&=\mathbf px\mathbf q_0\biggl(\prod_{i=1}^{r-1} x\mathbf q_i\biggr)x\mathbf q_rx\mathbf r\approx\mathbf px\mathbf q_0\biggl(\prod_{i=1}^{r-1} x\mathbf q_i\biggr)x^2\mathbf q_r\mathbf r\approx\\[-3pt]
&\approx\mathbf px\mathbf q_0\biggl(\prod_{i=1}^{r-2} x\mathbf q_i\biggr)x^2\mathbf q_{r-1}x\mathbf q_r\mathbf r\approx\cdots\approx\mathbf px\mathbf q_0x^2\mathbf q_1\biggl(\prod_{i=2}^r x\mathbf q_i\biggr)\mathbf r\approx\mathbf px^2\mathbf q\mathbf r,
\end{align*}
and we are done.
\end{proof}

Put $\mathbf A^\ast=\mathbf A\{\mathbf w_n[\pi,\tau]\approx\mathbf w_n^\prime[\pi,\tau]\mid n\in\mathbb N,\,\pi,\tau\in S_n\}$.

\begin{lemma}
\label{Lem: reduction to linear-balanced}
Let $\mathbf X\in[\mathbf D_p,\mathbf A^\ast\{\delta_p\}]$ and $\mathbf Y\in[\mathbf D_q,\mathbf A^\ast\{\delta_q\}]$, where $2\le p,q\le\infty$. Suppose that $\mathbf X\wedge\mathbf Y$ satisfies an identity $\mathbf u\approx\mathbf v$. Then there are a linear-balanced identity $\mathbf u^\prime\approx\mathbf v^\prime$ that holds in $\mathbf X\wedge\mathbf Y$ and a word $\mathbf p$ such that $\mathbf A^\ast\{\delta_p\}$ and $\mathbf A^\ast\{\delta_q\}$ satisfy the identities $\mathbf u\approx\mathbf p\mathbf u^\prime$ and $\mathbf v\approx\mathbf p\mathbf v^\prime$, respectively, and $\con(\mathbf p)\cap\con(\mathbf u^\prime)=\varnothing$.
\end{lemma}

\begin{proof}
We will assume without loss of generality that $p\le q$. Let ~\eqref{decomposition of u} is the decomposition of $\mathbf u$. Then the decomposition of $\mathbf v$ has the form~\eqref{decomposition of v} by Lemma~\ref{Lem: decompositions of u and v}. Suppose that exactly $k$ letters are not linear-balanced in the identity $\mathbf u\approx\mathbf v$. We use induction on $k$.

\smallskip

\emph{Induction base}. If $k=0$, then the required statement is evident.

\smallskip

\emph{Induction step}. Suppose that $k>0$. Let $x$ be a letter that is not linear-balanced in the identity $\mathbf u\approx\mathbf v$. Put $e_i=\occ_x(\mathbf u_i)$ and $f_i=\occ_x(\mathbf v_i)$ for $i=0,1,\dots,m$. Let $e=\sum_{i=0}^m e_i$ and $f=\sum_{i=0}^m f_i$. Further considerations are naturally divided into two cases.

\smallskip

\emph{Case 1}: $e_i,f_j>1$ for some $i$ and $j$. Being a subvariety of $\mathbf A$, the variety $\mathbf A^\ast$ satisfies the identity~\eqref{xxy=yxx}. Therefore, Lemma~\ref{Lem: from pxqxr to px^2qr} implies that $\mathbf A^\ast\{\delta_p\}$ and $\mathbf A^\ast\{\delta_q\}$ satisfy the identities $\mathbf u\approx x^2\mathbf u_x$ and $\mathbf v\approx x^2\mathbf v_x$, respectively. 

The identity $\mathbf u_x\approx\mathbf v_x$ holds in $\mathbf X\wedge\mathbf Y$ and only $k-1$ letters are not linear-balanced in this identity. By the induction assumption, there are a linear-balanced identity $\mathbf u^\prime\approx\mathbf v^\prime$ that holds in $\mathbf X\wedge\mathbf Y$ and a word $\mathbf p$ such that $\mathbf A^\ast\{\delta_p\}$ and $\mathbf A^\ast\{\delta_q\}$ satisfy the identities $\mathbf u_x\approx\mathbf p\mathbf u^\prime$ and $\mathbf v_x\approx\mathbf p\mathbf v^\prime$, respectively, and $\con(\mathbf p)\cap\con(\mathbf u^\prime)=\varnothing$. Then $\mathbf A^\ast\{\delta_p\}$ and $\mathbf A^\ast\{\delta_q\}$ satisfy the identities $\mathbf u\approx x^2\mathbf u_x\approx x^2\mathbf p\mathbf u^\prime$ and $\mathbf v\approx x^2\mathbf v_x\approx x^2\mathbf p\mathbf v^\prime$, respectively, and we are done.

\smallskip

\emph{Case 2}: either $e_i\le 1$ for all $i=0,1,\dots,m$ or $f_i\le 1$ for all $i=0,1,\dots,m$. We may assume without loss of generality that the first claim is true. The identity 
$$
\mathbf u(x,t_1,t_2,\dots,t_m)\approx\mathbf v(x,t_1,t_2,\dots,t_m)
$$
is non-trivial because the letter $x$ is not linear-balanced in the identity $\mathbf u\approx\mathbf v$. Then Lemma~\ref{Lem: Straubing identities in A}(ii) implies that the variety $\mathbf X\wedge\mathbf Y$ satisfies the identity $\delta_{e-1}$. It is clear that $p\le e-1$ because $\delta_{e-1}$ is false in $\mathbf D_p$ otherwise. Hence $\mathbf A^\ast\{\delta_p\}$ satisfies the identity $\delta_{e-1}$.

Suppose that $f_i\le 1$ for all $i=0,1,\dots,m$. Then we may assume that $e\le f$ by symmetry. Put $\Psi=\{\eqref{xx=xxx},\,\delta_{e-1}\}$. Then $\mathbf A^\ast\{\delta_p\}$ satisfies the identity $\mathbf u\stackrel{\Psi}\approx\mathbf w$, where 
$$
\mathbf w=x^{f_0}(\mathbf u_0)_x\biggl(\prod_{i=1}^m t_ix^{f_i}(\mathbf u_i)_x\biggr).
$$
The identity $\mathbf w\approx\mathbf v$ holds in $\mathbf X\wedge\mathbf Y$ because $\mathbf w\approx\mathbf u$ holds in $\mathbf A^\ast\{\delta_p\}$ and $\mathbf u\approx\mathbf v$ holds in $\mathbf X\wedge\mathbf Y$. Besides that, the letter $x$ is linear-balanced in the identity $\mathbf w\approx\mathbf v$, whence only $k-1$ letters are not linear-balanced in this identity. By the induction assumption, there are linear-balanced identity $\mathbf u^\prime\approx\mathbf v^\prime$ that holds in $\mathbf X\wedge\mathbf Y$ and a word $\mathbf p$ such that $\mathbf A^\ast\{\delta_p\}$ and $\mathbf A^\ast\{\delta_q\}$ satisfy the identities $\mathbf w\approx\mathbf p\mathbf u^\prime$ and $\mathbf v\approx\mathbf p\mathbf v^\prime$, respectively, and $\con(\mathbf p)\cap\con(\mathbf u^\prime)=\varnothing$. Then $\mathbf A^\ast\{\delta_p\}$ and $\mathbf A^\ast\{\delta_q\}$ satisfy the identities $\mathbf u\approx\mathbf w\approx\mathbf p\mathbf u^\prime$ and $\mathbf v\approx\mathbf p\mathbf v^\prime$, respectively, and we are done.

Finally, suppose that $f_j>1$ for some $j\in\{0,1,\dots,m\}$. It is clear that $\mathbf A^\ast\{\delta_p\}$ satisfies the identity $\mathbf u\stackrel{\Psi}\approx x^2\mathbf u_x$. Further, Lemma~\ref{Lem: from pxqxr to px^2qr} implies that $\mathbf A^\ast\{\delta_q\}$ satisfies the identity $\mathbf v\approx x^2\mathbf v_x$. The identity $\mathbf u_x\approx\mathbf v_x$ holds in $\mathbf X\wedge\mathbf Y$ and only $k-1$ letters are not linear-balanced in this identity. This allows us to complete the proof by repeating literally arguments from the second paragraph of Case~1.
\end{proof}

Let $n\in\mathbb N$, $0\le k\le\ell\le n$ and $\pi,\tau\in S_n$. Put 
\begin{align*}
\mathbf w_n^{k,\ell}[\pi,\tau]=&\phantom{\cdot}\ \biggl(\prod_{i=1}^n z_it_i\biggr)\biggl(\prod_{i=1}^k z_{i\pi}z_{n+i\tau}\biggr)x\biggl(\prod_{i=k+1}^\ell z_{i\pi}z_{n+i\tau}\biggr)x\\[-3pt]
&\cdot\biggl(\prod_{i=\ell+1}^n z_{i\pi}z_{n+i\tau}\biggr)\biggl(\prod_{i=n+1}^{2n} t_iz_i\biggr).
\end{align*}
Note that $\mathbf w_n^{0,n}[\pi,\tau]=\mathbf w_n[\pi,\tau]$ and $\mathbf w_n^{0,0}[\pi,\tau]=\mathbf w_n^\prime[\pi,\tau]$.

\begin{lemma}
\label{Lem: S(w_n[pi,tau]) notin V}
Let $n\in\mathbb N$, $\pi,\tau\in S_n$ and $\mathbf X$ be a monoid variety such that $\mathbf L\subseteq\mathbf X\subseteq\mathbf A\{\sigma_1,\,\sigma_2\}$. If $S(\mathbf w_n[\pi,\tau])\notin\mathbf X$, then $\mathbf X$ satisfies a non-trivial identity of the form $\mathbf w_n[\pi,\tau]\approx\mathbf w_n^{k,\ell}[\pi,\tau]$ for some $0\le k\le\ell\le n$.
\end{lemma}

\begin{proof}
Suppose that $S(\mathbf w_n[\pi,\tau])\notin\mathbf X$. Then $\mathbf X$ satisfies a non-trivial identity of the form $\mathbf w_n[\pi,\tau]\approx\mathbf w$ by Lemma~\ref{Lem: S(W) in V}. Repeating literally arguments from the proof of Lemma~4.10 in Gusev and Vernikov~\cite{Gusev-Vernikov-18}, we can check that 
$$
\mathbf w_x=\biggl(\prod_{i=1}^n z_it_i\biggr)\biggl(\prod_{i=1}^n z_{i\pi}z_{n+i\tau}\biggr)\biggl(\prod_{i=n+1}^{2n} t_iz_i\biggr).
$$
Lemma~\ref{Lem: x^n is isoterm} and inclusions $\mathbf C_2\subset\mathbf L\subseteq\mathbf X$ imply that $x$ is an isoterm for $\mathbf V$, whence $\occ_x(\mathbf w)\ge 2$. Therefore, 
$$
\mathbf w=\biggl(\prod_{i=1}^n \mathbf p_{2i-1}z_i\mathbf p_{2i}t_i\biggr)\mathbf q_0\biggl(\prod_{i=1}^n z_{i\pi}\mathbf q_{2i-1}z_{n+i\tau}\mathbf q_{2i}\biggr)\biggl(\prod_{i=n+1}^{2n} t_i\mathbf r_{2i-2n-1}z_i\mathbf r_{2i-2n}\biggr),
$$
where $\mathbf p_1\cdots\mathbf p_{2n}\mathbf q_0\mathbf q_1\cdots\mathbf q_{2n}\mathbf r_1\cdots\mathbf r_{2n}=x^m$ with $m\ge 2$. If $m=2$, then we can complete the proof by the same arguments as in the proof of Lemma~4.10 in~\cite{Gusev-Vernikov-18}. Now we suppose that $m>2$. 

Suppose that $x$ appears at most once in all blocks of the word $\mathbf w$. Then $\mathbf X$ satisfies the identity $\delta_{m-1}$ by Lemma~\ref{Lem: Straubing identities in A}(ii). Therefore, $\mathbf X$ satisfies the identity
$$
\mathbf w_n[\pi,\tau]\stackrel{\Psi}\approx\biggl(\prod_{i=1}^n z_it_i\biggr)x^2\biggl(\prod_{i=1}^n z_{i\pi}z_{n+i\tau}\biggr)\biggl(\prod_{i=n+1}^{2n} t_iz_i\biggr)=\mathbf w_n^{0,0}[\pi,\tau],
$$
where $\Psi=\{\eqref{xx=xxx},\,\eqref{xxy=yxx},\,\delta_{m-1}\}$. Thus, we may assume that $x$ appears at least twice in some block of the word $\mathbf w$. Further considerations are divided into three cases.

\smallskip

\emph{Case 1}: $\occ_x(\mathbf p_{2j-1}\mathbf p_{2j})>1$ for some $j\in\{1,2,\dots,n\}$. If either $\occ_x(\mathbf p_{2j-1})>1$ or $\occ(\mathbf p_{2j})>1$, then we can use the identities~\eqref{xx=xxx} and~\eqref{xxy=yxx} and obtain the identity $\mathbf w_n[\pi,\tau]\approx\mathbf w_n^{0,0}[\pi,\tau]$. 

It remains to consider the case when $\mathbf p_{2j-1}=\mathbf p_{2j}=x$. Here we use the identity $\sigma_1$ and swap the word $\mathbf p_{2j-1}$ and the letter $z_j$. Then we obtain the situation considered in the previous paragraph.

\smallskip

\emph{Case 2}: $\occ_x(\mathbf r_{2j-1}\mathbf r_{2j})>1$ for some $j\in\{1,2,\dots,n\}$. This case is dual to the previous one.

\smallskip

\emph{Case 3}: $\occ_x(\mathbf q_0\mathbf q_1\cdots\mathbf q_{2n})>1$. If $\occ_x(\mathbf q_i)>1$ for some $0\le i\le 2n$, then, as well as in Case~1, we can use the identities~\eqref{xx=xxx} and~\eqref{xxy=yxx} and obtain the identity $\mathbf w_n[\pi,\tau]\approx\mathbf w_n^{0,0}[\pi,\tau]$.

Thus, we may assume that $\occ_x(\mathbf q_i)\le 1$ for $i=0,1,\dots,2n$. Then there are $s$ and $t$ such that $s<t$, $\mathbf q_s=\mathbf q_t=x$ and $\mathbf q_{s+1}\cdots\mathbf q_{t-1}=\lambda$. Since the letter $x$ appears at most twice in the word $\mathbf w$, we have that either $x\in\con\bigl(\prod_{i=1}^{2n}\mathbf p_i\cdot\prod_{i=0}^{s-1}\mathbf q_i\bigr)$ or $x\in\con\bigl(\prod_{i=t+1}^{2n}\mathbf q_i\cdot\prod_{i=1}^{2n}\mathbf r_i\bigr)$. By symmetry, it suffices to consider the former case. Then we apply the identities $\sigma_1$ and $\sigma_2$ and replace the word $\mathbf q_s$ so that it is next to the word $\mathbf q_t$. Then we obtain the situation considered in the previous paragraph.
\end{proof}

\begin{lemma}
\label{Lem: smth imply w_n[pi,tau]=w_n'[pi,tau]}
The identities $\sigma_3$ and~\eqref{xxy=yxx} imply the identity~\eqref{w_n[pi,tau]=w_n'[pi,tau]} for any $n\in\mathbb N$ and $\pi,\tau\in S_n$.
\end{lemma}

\begin{proof}
Indeed, we have
$$
\mathbf w_n[\pi,\tau]=\mathbf px\mathbf qx\mathbf r\stackrel{\sigma_3}\approx\mathbf px\mathbf a\mathbf bx\mathbf r\stackrel{\sigma_3}\approx\mathbf p\mathbf ax^2\mathbf b\mathbf r\stackrel{\eqref{xxy=yxx}}\approx\mathbf px^2\mathbf a\mathbf b\mathbf r\stackrel{\sigma_3}\approx\mathbf px^2\mathbf q\mathbf r=\mathbf w_n^\prime[\pi,\tau],
$$
where
$$
\mathbf p=\prod_{i=1}^n (z_it_i),\ \mathbf q=\prod_{i=1}^n (z_{i\pi}z_{n+i\tau}),\ \mathbf r=\prod_{i=n+1}^{2n} (t_iz_i),\ \mathbf a=\prod_{i=1}^n z_{i\pi}\quad\text{and}\quad\mathbf b=\prod_{i=1}^n z_{n+i\tau}.
$$
Lemma is proved.
\end{proof}

\subsection{Identities of the form $\mathbf c_{n,m,k}[\rho]\approx\mathbf c_{n,m,k}^\prime[\rho]$}
\label{Subsec: c_{n,m,k}[rho]= c_{n,m,k}'[rho]}
For any $n,m,k\in\mathbb N_0$ and $\rho\in S_{n+m+k}$, we define the words
\begin{align*}
\mathbf c_{n,m,k}[\rho]&=\biggl(\prod_{i=1}^n z_it_i\biggr)xyt\biggl(\prod_{i=n+1}^{n+m} z_it_i\biggr)x\biggl(\prod_{i=1}^{n+m+k} z_{i\rho}\biggr)y\biggl(\prod_{i=n+m+1}^{n+m+k} t_iz_i\biggr),\\[-3pt]
\mathbf c_{n,m,k}^\prime[\rho]&=\biggl(\prod_{i=1}^n z_it_i\biggr)yxt\biggl(\prod_{i=n+1}^{n+m} z_it_i\biggr)x\biggl(\prod_{i=1}^{n+m+k} z_{i\rho}\biggr)y\biggl(\prod_{i=n+m+1}^{n+m+k} t_iz_i\biggr).
\end{align*}
Note that $\mathbf c_{n,m,0}[\rho]=\mathbf c_{n,m}[\rho]$ and $\mathbf c_{n,m,0}^\prime[\rho]=\mathbf c_{n,m}^\prime[\rho]$ for all $n,m\in\mathbb N_0$ and $\rho\in S_{n+m}$.

\begin{lemma}
\label{Lem: from pxyqxrys to pyxqxrys}
Let $\mathbf V$ be a monoid variety satisfying the identities~\eqref{xxy=yxx},~\eqref{w_n[pi,tau]=w_n'[pi,tau]} and
\begin{equation}
\label{c_{n,m,k}[rho]=c_{n,m,k}'[rho]}
\mathbf c_{n,m,k}[\rho]\approx\mathbf c_{n,m,k}^\prime[\rho]
\end{equation}
for all $n,m,k\in\mathbb N_0$, $\pi,\tau\in S_n$ and $\rho\in S_{n+m+k}$. If $\mathbf w=\mathbf pxy\mathbf qx\mathbf ry\mathbf s$ and $\con(\mathbf r)\subseteq\mul(\mathbf w)$, then $\mathbf V$ satisfies the identity $\mathbf w\approx\mathbf pyx\mathbf qx\mathbf ry\mathbf s$.
\end{lemma}

\begin{proof}
If $\mathbf r=\lambda$, then the identity $\mathbf w\approx\mathbf pyx\mathbf qxy\mathbf s$ follows from the identity $\mathbf c_{0,0,0}[\varepsilon]\approx\mathbf c_{0,0,0}^\prime[\varepsilon]$, where $\varepsilon$ is the trivial permutation, that is, from the identity
\begin{equation}
\label{xytxy=yxtxy}
xytxy\approx yxtxy.
\end{equation}
Thus, we may assume that $\mathbf r\ne\lambda$. Further considerations are divided into two cases.

\smallskip

\emph{Case 1}: $x,y\notin\con(\mathbf r)$. Suppose that the word $\mathbf r$ is linear and depends on pairwise different letters $z_1,z_2,\dots,z_s$. Then $z_i\in\con(\mathbf p\mathbf q\mathbf s)$ for each $i=1,2,\dots,s$. Fix one occurrence of $z_i$ in $\mathbf p\mathbf q\mathbf s$ for any $i=1,2,\dots,s$. Renaming the letters $z_1,z_2,\dots,z_s$ if necessary, we can achieve the inclusions $z_1,z_2,\dots,z_n\in\con(\mathbf p)$, $z_{n+1},\dots,z_{n+m}\in\con(\mathbf q)$ and $z_{n+m+1},\dots,z_{n+m+k}\in\con(\mathbf s)$ for some $n,m$ and $k$ with $n+m+k=s$. Then $\mathbf w=\mathbf pxy\mathbf qxz_{1\rho}z_{2\rho}\cdots z_{s\rho}y\mathbf s$ for an appropriate permutation $\rho\in S_{n+m+k}$, where
$$
\mathbf p=\mathbf w_0\biggl(\prod_{i=1}^n z_i\mathbf w_i\biggr),\,\mathbf q=\mathbf w_{n+1}\biggl(\prod_{i=n+1}^{n+m} z_i\mathbf w_{i+1}\biggr),\,\mathbf s=\biggl(\prod_{i=n+m+1}^{n+m+k} \mathbf w_{i+1}z_i\biggr)\mathbf w_{n+m+k+2}
$$
for some words $\mathbf w_0,\mathbf w_1,\dots,\mathbf w_{n+m+k+2}$. Then the required identity $\mathbf w\approx\mathbf pyx\mathbf qx\mathbf ry\mathbf s$ follows from the identity~\eqref{c_{n,m,k}[rho]=c_{n,m,k}'[rho]}.

It remains to consider the case when the word $\mathbf r$ is not linear. Then there is a letter $y_1\in\con(\mathbf r)$ such that $\mathbf r=\mathbf v_1y_1\mathbf v_2y_1\mathbf v_3$ for some words $\mathbf v_1$, $\mathbf v_2$ and $\mathbf v_3$. Then we can apply Lemma~\ref{Lem: from pxqxr to px^2qr} and conclude that $\mathbf V$ satisfies the identities
$$
\mathbf w=\mathbf pxy\mathbf qx\mathbf v_1y_1\mathbf v_2y_1\mathbf v_3y\mathbf s\approx\mathbf pxy\mathbf qx\mathbf v_1y_1^2\mathbf v_2\mathbf v_3y\mathbf s\stackrel{\eqref{xxy=yxx}}\approx\mathbf pxy\mathbf qy_1^2x\mathbf v_1\mathbf v_2\mathbf v_3y\mathbf s.
$$
In other words, we can eliminate two occurrences of the letter $y_1$ from the word $\mathbf r$. Repeating these arguments the necessary number of times, we obtain that $\mathbf V$ satisfies the identity $\mathbf w\approx\mathbf pxy\mathbf qy_1^2y_2^2\cdots y_\ell^2x\mathbf r^\prime y\mathbf s$ for some letters $y_1,y_2,\dots,y_\ell$ and some empty or linear word $\mathbf r^\prime$. Then we can repeat arguments from the previous paragraph and obtain that $\mathbf V$ satisfies the identity $\mathbf w\approx\mathbf pyx\mathbf qy_1^2y_2^2\cdots y_\ell^2x\mathbf r^\prime y\mathbf s$. Finally, it remains to return the letters $y_1,y_2,\dots,y_\ell$ to their original places using Lemma~\ref{Lem: from pxqxr to px^2qr} and the identity~\eqref{xxy=yxx}. As a result, we obtain that $\mathbf V$ satisfies the required identity $\mathbf w\approx\mathbf pyx\mathbf qx\mathbf ry\mathbf s$.

\smallskip

\emph{Case 2}: either $x\in\con(\mathbf r)$ or $y\in\con(\mathbf r)$. Then there are the words $\mathbf r_1$, $\mathbf r_2$ and $\mathbf r_3$ such that $x\mathbf ry=\mathbf r_1x\mathbf r_2y\mathbf r_3$ and $x,y\notin\con(\mathbf r_2)$. Then considerations from Case~1 imply that $\mathbf V$ satisfies the identities
$$
\mathbf w=\mathbf pxy\mathbf q\mathbf r_1x\mathbf r_2y\mathbf r_3\mathbf s\approx\mathbf pyx\mathbf q\mathbf r_1x\mathbf r_2y\mathbf r_3\mathbf s=\mathbf pyx\mathbf qx\mathbf ry\mathbf s,
$$
and we are done.
\end{proof}

\begin{lemma}
\label{Lem: c_{n,m}[rho]=c[n,m]'[rho]}
Let $\mathbf V$ be a monoid variety that contains the variety $\mathbf M^\delta$ and does not contain the monoid $S(\mathbf c_{n,m}[\rho])$ for some $n,m\in\mathbb N_0$ and $\rho\in S_{n+m}$. Then $\mathbf V$ satisfies the identity
\begin{equation}
\label{c_{n,m}[rho]=c_{n,m}'[rho]}
\mathbf c_{n,m}[\rho]\approx\mathbf c_{n,m}^\prime[\rho].
\end{equation}
\end{lemma}

\begin{proof}
In view of Lemma~\ref{Lem: S(W) in V}, $\mathbf V$ satisfies a non-trivial identity of the form $\mathbf c_{n,m}[\rho]\approx\mathbf c$ and the word $xsytxy$ is an isoterm for $\mathbf V$. Then the word $xyx$ also is an isoterm for $\mathbf V$. Now Lemma~\ref{Lem: xt_1x...t_kx is isoterm} applies and we conclude that 
$$
\mathbf c=\biggl(\prod_{i=1}^n z_it_i\biggr)\mathbf c^\prime t\biggl(\prod_{i=n+1}^{n+m} z_it_i\biggr)\mathbf c^{\prime\prime},
$$
where $\mathbf c^\prime\in\{xy,yx\}$ and $\mathbf c^{\prime\prime}$ is a linear word that depends on the letters $z_1$, $z_2$, \dots, $z_{n+m}$, $x$ and $y$. The assertion dual to Lemma~2.5 in Gusev~\cite{Gusev-19} implies that 
$$
\mathbf c^{\prime\prime}=x\biggl(\prod_{i=1}^{n+m} z_{i\rho}\biggr)y.
$$
Then the claim that the identity $\mathbf c_{n,m}[\rho]\approx\mathbf c$ is non-trivial implies that $\mathbf c^\prime=yx$, whence $\mathbf c=\mathbf c_{n,m}^\prime[\rho]$. 
\end{proof}

\begin{lemma}
\label{Lem: V notin S(c_{n,m}[rho])}
Let $\mathbf V$ be a monoid variety such that $\mathbf D_2\subseteq\mathbf V$, $\mathbf N\nsubseteq\mathbf V$ and $S(\mathbf c_{n,m}[\rho])\notin\mathbf V$ for some $n,m\in\mathbb N_0$ and $\rho\in S_{n+m}$. Then the identity~\eqref{c_{n,m}[rho]=c_{n,m}'[rho]} holds in $\mathbf V$.
\end{lemma}

\begin{proof}
Let $\mathbf u\approx\mathbf v$ be an arbitrary identity that fails in $\mathbf N$. We denote by $X$ the set of all letters $x$ such that either $\occ_x(\mathbf u)>2$ or $x$ occurs more than one times in some block of the word $\mathbf u$. Lemma~\ref{Lem: S(W) in V} implies that the word $xyx$ is an isoterm for $\mathbf V$. Suppose that there is a letter $x\in X$ such that $\occ_x(\mathbf v)=2$ and $x$ occurs in two different blocks of $\mathbf v$. Let $t$ be a simple in $\mathbf v$ letter that is located between the occurrences of $x$ in $\mathbf v$. Then $\mathbf V$ satisfies the identity $\mathbf v(x,t)\approx\mathbf u(x,t)$ and $\mathbf v(x,t)=xtx$. Therefore, $\mathbf u(x,t)=xtx$, contradicting with the choice of the letter $x$. Therefore, the set $X$ coincides with the set of all letters $x$ such that either $\occ_x(\mathbf v)>2$ or $x$ occurs more than one times in some block of the word $\mathbf v$. 

Let $X=\{x_1,x_2,\dots,x_s\}$. Clearly, $\mathbf N$ satisfies the identities $\mathbf u\approx x_1^2x_2^2\cdots x_s^2\mathbf u_X$ and $\mathbf v\approx x_1^2x_2^2\cdots x_s^2\mathbf v_X$. This means that the identity $\mathbf u_X\approx\mathbf v_X$ fails in the variety $\mathbf N$. Let~\eqref{decomposition of u} be the decomposition of the word $\mathbf u_X$. By Lemma~\ref{Lem: decompositions of u and v}, the decomposition of the word $\mathbf v_X$ has the form~\eqref{decomposition of v}. Since $\mathbf N$ satisfies the identity $\sigma_3$, this variety satisfies also the identities
$$
\mathbf u_X\approx\mathbf p_0\mathbf q_0\biggl(\prod_{i=1}^m t_i\mathbf p_i\mathbf q_i\biggr)\quad\text{and}\quad\mathbf v_X\approx\mathbf p_0^\prime\mathbf q_0^\prime\biggl(\prod_{i=1}^m t_i\mathbf p_i^\prime\mathbf q_i^\prime\biggr),
$$
where $\mathbf p_i$ and $\mathbf p_i^\prime$ consist of the first occurrences of letters in the words $\mathbf u_X$ and $\mathbf v_X$, respectively, while $\mathbf q_i$ and $\mathbf q_i^\prime$ consist of the second occurrences of letters in the words $\mathbf u_X$ and $\mathbf v_X$, respectively. Suppose that $\con(\mathbf p_i)\ne\con(\mathbf p_i^\prime)$ for some $i\in\{0,1,\dots,m\}$. We may assume without loss of generality that there is a letter $x\in\con(\mathbf p_i)\setminus\con(\mathbf p_i^\prime)$. Then the first occurrence of $x$ in $\mathbf v_X$ lies in $\mathbf p_j^\prime$ for some $j\ne i$. We may assume without loss of generality that $i<j$. Then $\mathbf u(x,t_{i+1})=xt_{i+1}x$ but $\mathbf v(x,t_{i+1})\ne xt_{i+1}x$. But this is impossible because $\mathbf V$ satisfies the identity $\mathbf u(x,t_{i+1})\approx\mathbf v(x,t_{i+1})$ and $xtx$ is an isoterm for $\mathbf V$. Thus, $\con(\mathbf p_i)=\con(\mathbf p_i^\prime)$ for all $i=0,1,\dots,m$. Analogous arguments show that $\con(\mathbf q_i)=\con(\mathbf q_i^\prime)$ for all $i=0,1,\dots,m$. Therefore, the identity $\mathbf u_X\approx\mathbf v_X$ is linear-balanced.

If $\mathbf p_i=\mathbf p_i^\prime$ for all $i=0,1,\dots,m$, then
$$
\mathbf u_X\stackrel{\sigma_3}\approx\mathbf p_0\mathbf q_0\biggl(\prod_{i=1}^m t_i\mathbf p_i\mathbf q_i\biggr)\stackrel{\sigma_2}\approx\mathbf p_0\mathbf q_0^\prime\biggl(\prod_{i=1}^m t_i\mathbf p_i\mathbf q_i^\prime\biggr)\stackrel{\sigma_3}\approx\mathbf v_X.
$$
Since the identity $\mathbf u_X\approx\mathbf v_X$ fails in the variety $\mathbf N$ and the identities $\sigma_2$ and $\sigma_3$ hold in this variety, there is $i\in\{0,1,\dots,m\}$ with $\mathbf p_i\ne\mathbf p_i^\prime$. All blocks in the words $\mathbf u_X$ and $\mathbf v_X$ are linear words. Hence there are $i\in\{0,1,\dots,m\}$ and $x,y\in\con(\mathbf p_i)$ such that $x$ precedes $y$ in $\mathbf p_i$ but $y$ precedes $x$ in $\mathbf p_i^\prime$. Since the identity $\mathbf u_X\approx\mathbf v_X$ is linear-balanced and the letters $x$ and $y$ occur in the words $\mathbf u_X$ and $\mathbf v_X$ exactly two times, the identity $\mathbf u(x,y,t_{i+1})\approx\mathbf v(x,y,t_{i+1})$ coincides (up to renaming of letters) with one of the identities
\begin{align}
\label{xytxy=yxtyx}
&xytxy\approx yxtyx,\\
\label{xytyx=yxtxy}
&xytyx\approx yxtxy
\end{align}
or~\eqref{xytxy=yxtxy}. Since $\mathbf u\approx\mathbf v$ is an arbitrary identity that fails in $\mathbf N$, this means that each variety that does not contain $\mathbf N$ satisfies one of the identities~\eqref{xytxy=yxtxy},~\eqref{xytxy=yxtyx} or~\eqref{xytyx=yxtxy}. In particular, one of these identities holds in $\mathbf V$.

Lemma~\ref{Lem: c_{n,m}[rho]=c[n,m]'[rho]} allows us to suppose that $\mathbf M^\delta\nsubseteq\mathbf V$. Then Lemma~\ref{Lem: V contains D_2 but does not contain L or M or M^delta}(ii) applies with the conclusion that $\mathbf V$ satisfies the identity $\sigma_2$. Therefore, if $\mathbf V$ satisfies one of the identities~\eqref{xytxy=yxtyx} or~\eqref{xytyx=yxtxy}, then $\mathbf V$ satisfies also the identity~\eqref{xytxy=yxtxy}. Thus, $\mathbf V$ satisfies the identity~\eqref{xytxy=yxtxy} in any case. Therefore, $\mathbf V$ satisfies the identities
\begin{align*}
\mathbf c_{n,m}[\rho]\stackrel{\sigma_2}\approx{}&\biggl(\prod_{i=1}^n z_it_i\biggr)xyt\biggl(\prod_{i=n+1}^{n+m} z_it_i\biggr)\biggl(\prod_{i=1}^{n+m} z_{i\pi}\biggr)xy\\[-3pt]
\stackrel{\eqref{xytxy=yxtxy}}\approx&\biggl(\prod_{i=1}^n z_it_i\biggr)yxt\biggl(\prod_{i=n+1}^{n+m} z_it_i\biggr)\biggl(\prod_{i=1}^{n+m} z_{i\pi}\biggr)xy\stackrel{\sigma_2}\approx\mathbf c_{n,m}^\prime[\rho].
\end{align*}
Lemma is proved.
\end{proof}

\subsection{The variety $\mathbf A^\prime$}
\label{Subsec: A'}
For any $n,m,k\in\mathbb N_0$ and $\rho\in S_{n+m+k}$, we denote by $\mathbf d_{n,m,k}[\rho]$ and $\mathbf d_{n,m,k}^\prime[\rho]$ the words that are obtained from the words $\mathbf c_{n,m,k}[\rho]$ and $\mathbf c_{n,m,k}^\prime[\rho]$, respectively, when reading the last words from right to left. Put 
$$
\mathbf A^\prime=\mathbf A^\ast\{\mathbf c_{n,m,k}[\rho]\approx\mathbf c_{n,m,k}^\prime[\rho],\,\mathbf d_{n,m,k}[\rho]\approx\mathbf d_{n,m,k}^\prime[\rho]\mid n,m,k\in\mathbb N_0,\,\rho\in S_{n+m+k}\}.
$$
The following statement indicates an important property of the lattice $L(\mathbf A^\prime)$, which will be useful for us in the proof of Corollaries~\ref{Cor: fi-perm distributive} and~\ref{Cor: almost fi-perm distributive}.

\begin{proposition}
\label{Prop: L(A') is distributive}
The lattice $L(\mathbf A^\prime)$ is distributive.
\end{proposition}

\begin{proof}
In view of Lemma~\ref{Lem: L(A)}(i), it suffices to verify that the interval $[\mathbf D_2,\mathbf A^\prime]$ is distributive. We are going to deduce this fact from Lemma~\ref{Lem: smth imply distributivity} with $\mathbf V=\mathbf A^\prime$, $\mathbf W=\mathbf D_2$ and the identity system $\Sigma$ that consists of all identities of the form $\delta_n$ with $2\le n\le\infty$ and all identities of the form~\eqref{pxyq=pyxq} such that the equalities~\eqref{pxyq=pyxq restrict} hold. In view of Corollaries~\ref{Cor: delta_k in X wedge Y} and~\ref{Cor: pxyq=pyxq in X wedge Y}, to do this, it remains to prove only that each monoid variety from the interval $[\mathbf D_2,\mathbf A^\prime]$ may be given within the variety $\mathbf A^\prime$ by the identity $\delta_n$ with some $2\le n\le\infty$ and identities of the form~\eqref{pxyq=pyxq} such that the equalities~\eqref{pxyq=pyxq restrict} hold.

So, let $\mathbf V\in[\mathbf D_2,\mathbf A^\prime]$ and $\mathbf a\approx\mathbf b$ be an identity that holds in $\mathbf D_2$. In view of Lemma~\ref{Lem: L(A)}(ii), $\mathbf A^\prime\{\mathbf a\approx\mathbf b\}\in[\mathbf D_n,\mathbf A^\prime\{\delta_n\}]$ for some $2\le n\le\infty$. Lemma~\ref{Lem: reduction to linear-balanced} with $\mathbf X=\mathbf Y=\mathbf A^\prime\{\mathbf a\approx\mathbf b\}$ implies that $\mathbf A^\prime\{\delta_n\}$ satisfies the identities $\mathbf a\approx\mathbf p\mathbf u$ and $\mathbf p\mathbf v\approx\mathbf b$ for some linear-balanced identity $\mathbf u\approx\mathbf v$ that holds in $\mathbf A^\prime\{\mathbf a\approx\mathbf b\}$ and some word $\mathbf p$ with $\con(\mathbf p)\cap\con(\mathbf u)=\varnothing$. Clearly, $\mathbf A^\prime\{\mathbf a\approx\mathbf b\}=\mathbf A^\prime\{\delta_n,\,\mathbf u\approx\mathbf v\}$. To complete the proof, it suffices to verify that each linear-balanced identity $\mathbf u\approx\mathbf v$ is equivalent within $\mathbf A^\prime$ to some system of identities of the form~\eqref{pxyq=pyxq} such that the equalities~\eqref{pxyq=pyxq restrict} hold. We denote the set of all identities of such a kind by $\Gamma$. 

We call an identity $\mathbf c\approx\mathbf d$ 1-\emph{invertible} if $\mathbf c=\mathbf w^\prime xy\mathbf w^{\prime\prime}$ and $\mathbf d=\mathbf w^\prime yx\mathbf w^{\prime\prime}$ for some words $\mathbf w^\prime,\mathbf w^{\prime\prime}$ and letters $x,y\in\con(\mathbf w^\prime\mathbf w^{\prime\prime})$. Let $n>1$. An identity $\mathbf c\approx\mathbf d$ is called $n$-\emph{invertible} if there is a sequence of words $\mathbf c=\mathbf w_0,\mathbf w_1,\dots,\mathbf w_n=\mathbf d$ such that the identity $\mathbf w_i\approx\mathbf w_{i+1}$ is 1-invertible for each $i=0,1,\dots,n-1$ and $n$ is the least number with such a property. For convenience, we will call the trivial identity 0-\emph{invertible}. 

Since the identity $\mathbf u\approx\mathbf v$ is linear-balanced, it is $r$-invertible for some $r\in\mathbb N_0$. We will use induction by $r$.

\smallskip

\emph{Induction base}. If $r=0$, then $\mathbf u=\mathbf v$, whence $\mathbf A^\prime\{\mathbf u\approx\mathbf v\}=\mathbf A^\prime\{\varnothing\}$.

\smallskip

\emph{Induction step}. Let $r>0$. Suppose that the decompositions of the words $\mathbf u$ and $\mathbf v$ have the forms~\eqref{decomposition of u} and~\eqref{decomposition of v}, respectively. Obviously, $\mathbf u_i\ne\mathbf v_i$ for some $i\in\{0,1,\dots,m\}$. Then the claim that the identity $\mathbf u\approx\mathbf v$ is linear-balanced implies that $\mathbf u_i=\mathbf u_i^\prime yx\mathbf u_i^{\prime\prime}$ for some words $\mathbf u_i^\prime,\mathbf u_i^{\prime\prime}$ and letters $x,y$ such that $x$ precedes $y$ in $\mathbf v_i$. We denote by $\mathbf w$ the word that is obtained from $\mathbf u$ by swapping of the occurrences of $x$ and $y$ in the block $\mathbf u_i$. 

Suppose that $x,y\in\con(\mathbf u_j)=\con(\mathbf v_j)$ for some $j\ne i$. Lemma~\ref{Lem: from pxyqxrys to pyxqxrys} shows that we can swap two adjacent occurrences of $x$ and $y$ if somewhere to the right of these occurrences there are two more occurrences of these letters lying in the same block. Thus, Lemma~\ref{Lem: from pxyqxrys to pyxqxrys} or the statement dual to it implies that $\mathbf A^\prime$ satisfies the identity $\mathbf u\approx\mathbf w$. The identity $\mathbf w\approx\mathbf v$ is \mbox{$(r-1)$}-invertible. By the induction assumption, $\mathbf A^\prime\{\mathbf w\approx\mathbf v\}=\mathbf A^\prime\Phi$ for some $\Phi\subseteq\Gamma$. Then $\mathbf A^\prime\{\mathbf u\approx\mathbf v\}=\mathbf A^\prime\Phi$, and we are done.

Thus, we may assume that at most one of the letters $x$ and $y$ occurs in $\mathbf u_j$ for any $j\ne i$. Let $s$ be the least number such that $\con(\mathbf u_s)\cap\{x,y\}\ne\varnothing$. Put
\begin{equation}
\label{T_{x,y}}
T_{x,y}=\{t_j\mid s<j\le m,\,\con(\mathbf u_j)\cap\{x,y\}\ne\varnothing\}.
\end{equation}
Since the identity $\mathbf u\approx\mathbf v$ is linear-balanced, the identity
\begin{equation}
\label{2 mult letters and their dividers}
\mathbf u(x,y,T_{x,y})\approx\mathbf v(x,y,T_{x,y})
\end{equation}
coincides (up to renaming of letters) with an identity of the form~\eqref{pxyq=pyxq} such that the equalities~\eqref{pxyq=pyxq restrict} hold. It is clear that $\mathbf u\stackrel{\eqref{2 mult letters and their dividers}}\approx\mathbf w$ and the identity $\mathbf w\approx\mathbf v$ is \mbox{$(r-1)$}-invertible. By the induction assumption, $\mathbf A^\prime\{\mathbf w\approx\mathbf v\}=\mathbf A^\prime\Phi$ for some $\Phi\subseteq\Gamma$. Then $\mathbf A^\prime\{\mathbf u\approx\mathbf v\}=\mathbf A^\prime\{\eqref{2 mult letters and their dividers},\,\Phi\}$, and we are done.
\end{proof}

\section{Certain non-permutative fully invariant congruences}
\label{Sec: non-fi-perm}

Here we find a number of pairs of non-permutative congruences of the type $(\theta_{\mathbf X},\theta_{\mathbf Y})$ on the free monoid $\mathfrak X^\ast$. This will be very helpful for us in Section~\ref{Sec: proof of main results} because the variety $\mathbf X\vee\mathbf Y$ is not $fi$-permutable in this case by Lemma~\ref{Lem: lifting}.

For any $n>1$, we denote by $\mathbf A_n$ the variety of Abelian groups of exponent $n$. Put $\mathbf Z_1=\var S(xysxtxhy)$, $\mathbf Z_2=\var S(xysxtyhx)$ and $\mathbf Z_3=\var S(xysytxhx)$.

\begin{lemma}
\label{Lem: non-fi-perm}
The following congruences on the free monoid $\mathfrak X^\ast$ do not permute:
\begin{itemize}
\item[(i)] $\theta_{\mathbf A_n}$ with $n\ge 2$ and $\theta_{\mathbf{SL}}$;
\item[(ii)] $\theta_{\mathbf A_n\vee\mathbf{SL}}$ with $n\ge 2$ and $\theta_{\mathbf C_2}$;
\item[(iii)] $\theta_{\mathbf C_3}$ and $\theta_{\mathbf D_2}$;
\item[(iv)] $\theta_{\mathbf D_2}$ and $\theta_{\mathbf E}$;
\item[(v)] $\theta_{\mathbf E}$ and $\theta_{\mathbf E^\delta}$;
\item[(vi)] $\theta_{\mathbf L}$ and $\theta_{\mathbf M}$;
\item[(vii)] $\theta_{\mathbf N}$ and $\theta_{\mathbf Z_i}$ with $1\le i\le 3$;
\item[(viii)] $\theta_{\mathbf Z_i}$ and $\theta_{\mathbf Z_j}$ with $1\le i<j\le 3$.
\end{itemize}
\end{lemma}

\begin{proof}
(i) Suppose that the congruences $\theta_{\mathbf A_n}$ and $\theta_{\mathbf{SL}}$ permute. It is obvious that $x\,\theta_{\mathbf A_n}\,xy^n\,\theta_{\mathbf{SL}}\,xy$. Then $(x,xy)\in\theta_{\mathbf A_n}\theta_{\mathbf{SL}}=\theta_{\mathbf{SL}}\theta_{\mathbf A_n}$. Whence, there exists a word $\mathbf w$ such that $x\,\theta_{\mathbf{SL}}\,\mathbf w\,\theta_{\mathbf A_n}\,xy$. Now Lemma~\ref{Lem: group variety} applies with the conclusion that $\mathbf w=x^k$ for some $k\in\mathbb N$. But this is impossible because the variety $\mathbf A_n$ violates the identity $x^k\approx xy$.

\smallskip

(ii) Suppose that the congruences $\theta_{\mathbf A_n\vee\mathbf{SL}}$ and $\theta_{\mathbf C_2}$ permute. It is obvious that $x\,\theta_{\mathbf A_n\vee\mathbf{SL}}\,x^{n+1}\,\theta_{\mathbf C_2}\,x^2$. Then $(x,x^2)\in\theta_{\mathbf A_n\vee\mathbf{SL}}\theta_{\mathbf C_2}=\theta_{\mathbf C_2}\theta_{\mathbf A_n\vee\mathbf{SL}}$. Therefore, there exists a word $\mathbf w$ such that $x\,\theta_{\mathbf C_2}\,\mathbf w\,\theta_{\mathbf A_n\vee\mathbf{SL}}\,x^2$. Now Lemma~\ref{Lem: x^n is isoterm} implies that $\mathbf w=x$. But this is not the case because the identity $x\approx x^2$ fails in $\mathbf A_n$.

\smallskip

(iii) Obviously, $xyx\,\theta_{\mathbf C_3}\,x^2y\,\theta_{\mathbf D_2}\,x^3y$. If the congruences $\theta_{\mathbf C_3}$ and $\theta_{\mathbf D_2}$ permute, then $(xyx,x^3y)\in\theta_{\mathbf C_3}\theta_{\mathbf D_2}=\theta_{\mathbf D_2}\theta_{\mathbf C_3}$. Therefore, $xyx\,\theta_{\mathbf D_2}\,\mathbf w\,\theta_{\mathbf C_3}\,x^3y$ for some word $\mathbf w$. Now Lemma~\ref{Lem: S(W) in V} implies that $\mathbf w=xyx$. But then the variety $\mathbf C_3$ satisfies the identity~\eqref{xx=xxx}, a contradiction with Lemma~\ref{Lem: x^n is isoterm}.

\smallskip

(iv) Obviously, $xyx\,\theta_{\mathbf E}\,x^2y\,\theta_{\mathbf D_2}\,yx^2$. Thus, $(xyx,yx^2)\in\theta_{\mathbf E}\theta_{\mathbf D_2}$. If the congruences $\theta_{\mathbf E}$ and $\theta_{\mathbf D_2}$ permute, then there is a word $\mathbf w$ such that $xyx\,\theta_{\mathbf D_2}\,\mathbf w\,\theta_{\mathbf E}\,yx^2$. Then $\mathbf w=xyx$ by Lemma~\ref{Lem: S(W) in V} and, simultaneously, $\mathbf w=yx^\ell$ for some $\ell\ge 2$ by Lemma~\ref{Lem: yx^n=w in E}. We have a contradiction.

\smallskip

(v) Obviously, $x^2y\,\theta_{\mathbf E}\,xyx\,\theta_{\mathbf E^\delta}\,yx^2$, whence $(x^2y,yx^2)\in\theta_{\mathbf E}\theta_{\mathbf E^\delta}$. But Lemma~\ref{Lem: yx^n=w in E} and the dual to it show that there is no word $\mathbf w$ such that the identities $x^2y\approx\mathbf w$ and $\mathbf w\approx yx^2$ hold in the varieties $\mathbf E^\delta$ and $\mathbf E$, respectively. Therefore, the congruences $\theta_{\mathbf E}$ and $\theta_{\mathbf E^\delta}$ do not permute.

\smallskip

(vi) Put $\mathbf u=xsxyztyhz$ and $\mathbf v=xszxytyhz$. Then $\mathbf u\,\theta_{\mathbf L}\,xsxzytyhz\,\theta_{\mathbf M}\,\mathbf v$, whence $(\mathbf u,\mathbf v)\in\theta_{\mathbf L}\theta_{\mathbf M}$. Suppose that the congruences $\theta_{\mathbf L}$ and $\theta_{\mathbf M}$ permute. Then $\mathbf u\,\theta_{\mathbf M}\,\mathbf w\,\theta_{\mathbf L}\,\mathbf v$ for some word $\mathbf w$. It is evident that the monoid $S(xyx)$ lies in the varieties $\mathbf L$ and $\mathbf M$. Then the word $xyx$ is an isoterm for $\mathbf L$ and $\mathbf M$ by Lemma~\ref{Lem: S(W) in V}. Now we can apply Lemma~\ref{Lem: xt_1x...t_kx is isoterm} and conclude that $\mathbf w=xs\mathbf atyhz$, where $\mathbf a$ is a linear word with $\con(\mathbf a)=\{x,y,z\}$. If $\mathbf a\in\{xzy,zxy,zyx\}$, then $\mathbf u(y,z,t,h)=yztyhz$ and $\mathbf w(y,z,t,h)\ne yztyhz$. This contradicts with the claim that the word $yztyhz$ is an isoterm for $\mathbf M$ by Lemma~\ref{Lem: S(W) in V}. If $\mathbf a\in\{xyz,yxz\}$, then $\mathbf w(x,z,s,t)=xsxztz$ and $\mathbf v(x,z,s,t)\ne xsxztz$. Finally, if $\mathbf a=yzx$, then $\mathbf v(x,y,s,t)=xsxyty$ and $\mathbf w(x,y,s,t)\ne xsxyty$. In both the cases we have a contradiction with the fact that the word $xsxyty$ is an isoterm for $\mathbf L$ by Lemma~\ref{Lem: S(W) in V}. 

\smallskip

(vii) We consider the case when $i=3$ only. The other cases can be considered quite analogously. Put $\mathbf u=yxzsyztxhx$ and $\mathbf v=xzysyztxhx$. Then $(\mathbf u,\mathbf v)\in\theta_{\mathbf N}\theta_{\mathbf Z_3}$ because $\mathbf u\,\theta_{\mathbf N}\,xyzsyztxhx\,\theta_{\mathbf Z_3}\,\mathbf v$. Suppose that the congruences $\theta_{\mathbf N}$ and $\theta_{\mathbf Z_3}$ permute. Then $\mathbf u\,\theta_{\mathbf Z_3}\,\mathbf w\,\theta_{\mathbf N}\,\mathbf v$ for some word $\mathbf w$. It is evident that the monoid $S(xyx)$ lies in the varieties $\mathbf N$ and $\mathbf Z_3$. Then the word $xyx$ is an isoterm for $\mathbf N$ and $\mathbf Z_3$ by Lemma~\ref{Lem: S(W) in V}. Now Lemma~\ref{Lem: xt_1x...t_kx is isoterm} applies with the conclusion that $\mathbf w=\mathbf as\mathbf btxhx$, where $\mathbf a$ is a linear word with $\con(\mathbf a)=\{x,y,z\}$ and $\mathbf b\in\{yz,zy\}$. If $\mathbf a\in\{xyz,xzy,zxy\}$, then $\mathbf u_z=yxsytxhx$ and $\mathbf w_z\ne yxsytxhx$. Further, if $\mathbf a\in\{yzx,zyx\}$, then $\mathbf w_y=zxsztxhx$ and $\mathbf u_y\ne zxsztxhx$. In both the cases we have a contradiction with the fact that the word $yxsytxhx$ is an isoterm for $\mathbf Z_3$ by Lemma~\ref{Lem: S(W) in V}. This means that $\mathbf a=yxz$ and $\mathbf w\in\{yxzsyztxhx,yxzszytxhx\}$. Then the identity $\mathbf w(y,z,s)\approx\mathbf v(y,z,s)$ coincides (up to renaming of letters) with one of the identities~\eqref{xytxy=yxtxy} or~\eqref{xytyx=yxtxy}. But this is impossible because the identity $\mathbf w\approx\mathbf v$ holds in the variety $\mathbf N$ and the identities~\eqref{xytxy=yxtxy} and~\eqref{xytyx=yxtxy} fail in this variety.

\smallskip

(viii) We consider the case when $i=1$ and $j=2$. The other cases can be considered quite analogously. For brevity, put $\mathbf p=t_1xt_2yt_3xt_4z$. Let us consider the words $\mathbf u=xyz\mathbf p$, $\mathbf v=yzx\mathbf p$ and $\mathbf w=yxz\mathbf p$. It is a routine to verify that $\mathbf u\,\theta_{\mathbf Z_1}\,\mathbf w\,\theta_{\mathbf Z_2}\,\mathbf v$. Thus, $(\mathbf u,\mathbf v)\in\theta_{\mathbf Z_1}\theta_{\mathbf Z_2}$. Suppose that the congruences $\theta_{\mathbf Z_1}$ and $\theta_{\mathbf Z_2}$ permute. Then $\mathbf u\,\theta_{\mathbf Z_2}\,\mathbf w^\prime\,\theta_{\mathbf Z_1}\,\mathbf v$ for some word $\mathbf w^\prime$. Lemma~\ref{Lem: S(W) in V} implies that the word $xyt_1xt_2yt_3x$ is an isoterm for $\mathbf Z_2$. Therefore, the word $xt_1xt_2x$ is an isoterm for $\mathbf Z_2$ as well. Now we can apply Lemma~\ref{Lem: xt_1x...t_kx is isoterm} and conclude that $\mathbf w^\prime=\mathbf a\mathbf p$, where $\mathbf a$ is a linear word with $\con(\mathbf a)=\{x,y,z\}$. If $\mathbf a\in\{yxz,yzx,zyx\}$, then the identity $\mathbf u_{\{z,t_4\}}\approx\mathbf w^\prime_{\{z,t_4\}}$ is non-trivial and the left-hand side of it coincides with $xyt_1xt_2yt_3x$. Further, if $\mathbf a\in\{xzy,zxy\}$, then we substitute $t_2zt_3$ for $t_2$ in the identity $\mathbf u(y,z,t_2,t_4)\approx\mathbf w^\prime(y,z,t_2,t_4)$ and obtain a non-trivial identity whose right-hand side is $zyt_2zt_3yt_4z$. Both the cases contradict the claim that $xysxtyhx$ is an isoterm for $\mathbf Z_2$ by Lemma~\ref{Lem: S(W) in V}. Therefore, $\mathbf a=xyz$. This means that the variety $\mathbf Z_1$ satisfies the identity $\mathbf u\approx\mathbf v$. But this is not the case because the identity $\mathbf u_{\{y,t_2\}}\approx\mathbf v_{\{y,t_2\}}$ is non-trivial and the left-hand side of it coincides with the word $xzt_1xt_3xt_4z$ which is an isoterm for $\mathbf Z_1$ by Lemma~\ref{Lem: S(W) in V}. 
\end{proof}

We denote the first letter of a word $\mathbf w$ by $h(\mathbf w)$. 

\begin{lemma}
\label{Lem: var S(w_n[pi,tau]) and var S(w_n[xi,eta])}
Let $n\in\mathbb N$ and $\pi,\tau,\xi,\eta\in S_n$. Suppose that $\mathbf w_n[\pi,\tau]\ne\mathbf w_n[\xi,\eta]$ and put $\mathbf X=\var S(\mathbf w_n[\pi,\tau])$ and $\mathbf Y=\var S(\mathbf w_n[\xi,\eta])$. Then the congruences $\theta_{\mathbf X}$ and $\theta_{\mathbf Y}$ on the free monoid $\mathfrak X^\ast$ do not permute.
\end{lemma}

\begin{proof}
Let $\mathbf u=\mathbf p_1x\mathbf q_1x\mathbf q_2\mathbf p_2$ and $\mathbf v=\mathbf p_1\mathbf q_1x\mathbf q_2x\mathbf p_2$, where
$$
\begin{aligned}
&\mathbf p_1=\biggl(\prod_{i=1}^{n\pi-1} z_it_i\biggr)\biggl(\prod_{i=1}^n z_i^\prime t_i^\prime\biggr)\biggl(\prod_{i=n\pi}^n z_it_i\biggr),&\mathbf q_1=\prod_{i=1}^n (z_{i\pi}z_{n+i\tau}),\\[-3pt]
&\mathbf p_2=\biggl(\prod_{i=n+1}^{n+n\tau-1} t_iz_i\biggr)\biggl(\prod_{i=n+1}^{2n} t_i^\prime z_i^\prime\biggr)\biggl(\prod_{i=n+n\tau}^{2n} t_iz_i\biggr),&\mathbf q_2=\prod_{i=1}^n (z_{i\xi}^\prime z_{n+i\eta}^\prime).
\end{aligned}
$$
It is clear that $\mathbf X$ and $\mathbf Y$ satisfy the identity~\eqref{xxy=yxx}. Lemma~\ref{Lem: w_n[pi,tau]=w_n'[pi,tau] in S(w_n[xi,eta])} implies that the variety $\mathbf Y$ satisfies the identity~\eqref{w_n[pi,tau]=w_n'[pi,tau]}. The variety $\mathbf Y$ satisfies the identity $\mathbf u\approx\mathbf p_1\mathbf q_1x^2\mathbf q_2\mathbf p_2$ because
$$
\mathbf u\stackrel{\eqref{w_n[pi,tau]=w_n'[pi,tau]}}\approx\mathbf p_1x^2\mathbf q_1\mathbf q_2\mathbf p_2\stackrel{\eqref{xxy=yxx}}\approx\mathbf p_1\mathbf q_1x^2\mathbf q_2\mathbf p_2.
$$
Analogous arguments show that the identity $\mathbf v\approx\mathbf p_1\mathbf q_1x^2\mathbf q_2\mathbf p_2$ holds in $\mathbf X$. Then $(\mathbf u,\mathbf v)\in\theta_{\mathbf Y}\theta_{\mathbf X}$. Suppose that the congruences $\theta_{\mathbf X}$ and $\theta_{\mathbf Y}$ permute. Then $\mathbf u\,\theta_{\mathbf X}\,\mathbf w\,\theta_{\mathbf Y}\,\mathbf v$ for some word $\mathbf w$.

The variety $\mathbf X$ satisfies the identity $\mathbf u_x\approx\mathbf w_x$. Lemma~\ref{Lem: xt_1x...t_kx is isoterm} implies that $\mathbf w_x=\mathbf p_1\mathbf q^\prime\mathbf p_2$, where $\mathbf q^\prime$ is a linear word with $\con(\mathbf q^\prime)=\{z_i,z_i^\prime\mid 1\le i\le 2n\}$. One can verify that $\mathbf q^\prime=\mathbf q_1\mathbf q_2$. It is proved in Gusev and Vernikov~\cite[p.~28]{Gusev-Vernikov-18} that the word $xzxyty$ is an isoterm for $\mathbf X$ and $\mathbf Y$. It can be checked directly that if $h(\mathbf q^\prime)\ne h(\mathbf q_1\mathbf q_2)$, then we can deduce from $\mathbf u_x\approx\mathbf w_x$ a non-trivial identity of the form $xzxyty\approx\mathbf a$, resulting in a contradiction. Indeed, suppose that $h(\mathbf q^\prime)=z_{n+i\tau}$. Then 
\begin{align*}
&\mathbf u(z_{1\pi},t_{1\pi},z_{n+i\tau},t_{n+i\tau})=z_{1\pi}t_{1\pi}z_{1\pi}z_{n+i\tau}t_{n+i\tau}z_{n+i\tau}\\
\text{and}\quad&\mathbf w(z_{1\pi},t_{1\pi},z_{n+i\tau},t_{n+i\tau})=z_{1\pi}t_{1\pi}z_{n+i\tau}z_{1\pi}t_{n+i\tau}z_{n+i\tau},
\end{align*} 
whence $\mathbf X$ satisfies the identity $xzxyty\approx xzyxty$. Analogous arguments show that $h(\mathbf q^\prime)$ differs from the letters $z_{i\pi}$ with $i>1$, $z_{i\xi}^\prime$ and $z_{1+i\eta}^\prime$. Thus, $h(\mathbf q^\prime)=h(\mathbf q_1\mathbf q_2)$. Arguing in a similar way and moving step by step from left to right by the word $\mathbf q^\prime$, we can establish that $\mathbf q^\prime=\mathbf q_1\mathbf q_2$, whence $\mathbf u_x=\mathbf w_x=\mathbf p_1\mathbf q_1\mathbf q_2\mathbf p_2$. Lemma~\ref{Lem: S(W) in V} implies that
\begin{align*}
&\mathbf w(x,t_1,z_1,t_2,z_2,\dots,t_n,z_n)=\mathbf w_n[\pi,\tau]\quad\text{and}\\[-3pt]
&\mathbf w(x,t_1^\prime,z_1^\prime,t_2^\prime,z_2^\prime,\dots,t_n^\prime,z_n^\prime)=\biggl(\prod_{i=1}^n z_i^\prime t_i^\prime\biggr)x\biggl(\prod_{i=1}^n z_{i\xi}^\prime z_{n+i\eta}^\prime\biggr)x\biggl(\prod_{i=n+1}^{2n} t_i^\prime z_i^\prime\biggr).
\end{align*}
These two equalities together with the claim that $\mathbf u_x=\mathbf w_x=\mathbf p_1\mathbf q_1\mathbf q_2\mathbf p_2$ are possible only in the case when $\mathbf w=\mathbf p_1x\mathbf q_1\mathbf q_2x\mathbf p_2$. 

Put
$$
X=\{x,z_i,t_i,z_{n\xi}^\prime,z_{n+n\eta}^\prime,t_{n\xi}^\prime,t_{n+n\eta}^\prime\mid 1\le i\le 2n,i\ne n\pi,i\ne n+n\tau\}.
$$
Then $\mathbf X$ satisfies the identity $\mathbf u(X)\approx\mathbf w(X)$, that is,
$$
\mathbf p_1(X)x\mathbf q_1(X)x\mathbf q_2(X)\mathbf p_2(X)\approx\mathbf p_1(X)x\mathbf q_1(X)\mathbf q_2(X)x\mathbf p_2(X).
$$
The right-hand side of this identity, that is, the word
\begin{align*}
&\biggl(\prod_{i=1}^{n\pi-1} z_it_i\biggr)z_{n\xi}^\prime t_{n\xi}^\prime\biggl(\prod_{i=n\pi+1}^n z_it_i\biggr)x\biggl(\prod_{i=1}^{n-1} z_{i\pi}z_{n+i\tau}\biggr)z_{n\xi}^\prime z_{n+n\eta}^\prime x\\[-3pt]
\cdot&\biggl(\prod_{i=n+1}^{n+n\tau-1} t_iz_i\biggr)t_{n+n\eta}^\prime z_{n+n\eta}^\prime\biggl(\prod_{i=n+n\tau+1}^{2n} t_iz_i\biggr),
\end{align*}
coincides (up to renaming of letters) with the word $\mathbf w_n[\pi,\tau]$ which is an isoterm for $\mathbf X$ by Lemma~\ref{Lem: S(W) in V}. We have a contradiction.
\end{proof}

\begin{lemma}
\label{Lem: var S(c_{n+1,m}[rho]) and S(c_{n,m+1}[tau])}
Let $n,m\in\mathbb N_0$, $n+m>0$, $\pi\in S_{n+m}$ and $\mathbf V=\var S(\mathbf c_{n,m}[\pi])$. Then there are permutations $\rho,\tau\in S_{n+m+1}$ such that the varieties $\mathbf X=\var S(\mathbf c_{n+1,m}[\rho])$ and $\mathbf Y=\var S(\mathbf c_{n,m+1}[\tau])$ are contained in $\mathbf V$ and the congruences $\theta_{\mathbf X}$ and $\theta_{\mathbf Y}$ on the free monoid $\mathfrak X^\ast$ do not permute.
\end{lemma}

\begin{proof}
Let $\mathbf X=\var S(\mathbf c_{n+1,m}[\rho])$ and $\mathbf Y=\var S(\mathbf c_{n,m+1}[\tau])$, where
\begin{align*}
&\rho=
\begin{pmatrix}
1&2&\dots&n+m&n+m+1\\
1\pi&2\pi&\dots&(n+m)\pi&n+m+1
\end{pmatrix}
\\
\text{and}\quad&\tau=
\begin{pmatrix}
1&2&3&\dots&n+m+1\\
1&1\pi+1&2\pi+1&\dots&(n+m)\pi+1
\end{pmatrix}
.
\end{align*}
For any $k,\ell\in\mathbb N_0$ and $\xi\in S_{k+\ell}$, the identity $\sigma_2$ implies a non-trivial identity of the form $\mathbf c_{k,\ell}[\xi]\approx\mathbf a$. Then Lemmas~\ref{Lem: S(W) in V} and~\ref{Lem: V contains D_2 but does not contain L or M or M^delta}(ii) imply that all the varieties $\mathbf V$, $\mathbf X$ and $\mathbf Y$ contain $\mathbf M^\delta$, whence the word $xzytxy$ is an isoterm for $\mathbf V$, $\mathbf X$ and $\mathbf Y$ by Lemma~\ref{Lem: S(W) in V}. If $\mathbf X\nsubseteq\mathbf V$, then $\mathbf V$ satisfies the identity $\mathbf c_{n+1,m}[\rho]\approx\mathbf c_{n+1,m}^\prime[\rho]$ by Lemma~\ref{Lem: c_{n,m}[rho]=c[n,m]'[rho]}. Since 
\begin{align*}
&(\mathbf c_{n+1,m}[\rho])_{\{t_{n+m+1},z_{n+m+1}\}}=\mathbf c_{n,m}[\pi]\\ 
\text{and}\quad&(\mathbf c_{n+1,m}^\prime[\rho])_{\{t_{n+m+1},z_{n+m+1}\}}=\mathbf c_{n,m}^\prime[\pi],
\end{align*}
we have a contradiction with Lemma~\ref{Lem: S(W) in V}. Therefore, $\mathbf X\subseteq\mathbf V$. It can be verified similarly that $\mathbf Y\subseteq\mathbf V$. 

Arguments similar to ones from the proof of Lemma~\ref{Lem: w_n[pi,tau]=w_n'[pi,tau] in S(w_n[xi,eta])} show that the variety $\mathbf X$ satisfies the identity 
\begin{equation}
\label{c_{n,m+1}[tau]=c_{n,m+1}'[tau]}
\mathbf c_{n,m+1}[\tau]\approx\mathbf c_{n,m+1}^\prime[\tau],
\end{equation}
while the variety $\mathbf Y$ satisfies the identity
\begin{equation}
\label{c_{n+1,m}[rho]=c_{n+1,m}'[rho]}
\mathbf c_{n+1,m}[\rho]\approx\mathbf c_{n+1,m}^\prime[\rho].
\end{equation}
Let $\mathbf u=\mathbf pyxz\mathbf q$ and $\mathbf v=\mathbf pxzy\mathbf q$, where
\begin{align*}
&\mathbf p=\biggl(\prod_{i=1}^n z_it_i\biggr)\biggl(\prod_{i=1}^{n+1} z_i^\prime t_i^\prime\biggr)\quad\text{and}\\[-3pt]
&\mathbf q=t\biggl(\prod_{i=n+1}^{n+m+1} z_it_i\biggr)\biggl(\prod_{i=n+2}^{n+m+1} z_i^\prime t_i^\prime\biggr)x\biggl(\prod_{i=1}^{n+m+1} z_{i\tau}\biggr)y\biggl(\prod_{i=1}^{n+m+1} z_{i\rho}^\prime\biggr)z.
\end{align*}
It is easy to see that $\mathbf u\stackrel{\eqref{c_{n,m+1}[tau]=c_{n,m+1}'[tau]}}\approx\mathbf pxyz\mathbf q\stackrel{\eqref{c_{n+1,m}[rho]=c_{n+1,m}'[rho]}}\approx\mathbf v$. Thus, $\mathbf u\,\theta_{\mathbf X}\,\mathbf pxyz\mathbf q\,\theta_{\mathbf Y}\,\mathbf v$, whence $(\mathbf u,\mathbf v)\in\theta_{\mathbf X}\theta_{\mathbf Y}$. Suppose that the congruences $\theta_{\mathbf X}$ and $\theta_{\mathbf Y}$ permute. Then there is a word $\mathbf w$ such that $\mathbf u\,\theta_{\mathbf Y}\,\mathbf w\,\theta_{\mathbf X}\,\mathbf v$. Since $xzytxy$ and therefore, $xtx$ are isoterms for $\mathbf X$ and $\mathbf Y$, Lemma~\ref{Lem: xt_1x...t_kx is isoterm} and the assertion dual to Lemma~2.5 in Gusev~\cite{Gusev-19} imply that $\mathbf w=\mathbf p\mathbf a\mathbf q$, where $\mathbf a$ is a linear word with $\con(\mathbf a)=\{x,y,z\}$. It is easy to see that:
\begin{itemize}
\item[(i)] if $\mathbf a\in\{xyz,xzy,zxy\}$, then $\mathbf Y$ satisfies the identity~\eqref{c_{n,m+1}[tau]=c_{n,m+1}'[tau]}; 
\item[(ii)] if $\mathbf a\in\{yzx,zyx\}$, then $\mathbf Y$ satisfies the identity 
\begin{align*}
&\biggl(\prod_{i=1}^n z_it_i\biggr)xzt\biggl(\prod_{i=n+1}^{n+m+1} z_it_i\biggr)x\biggl(\prod_{i=1}^{n+m+1} z_{i\tau}\biggr)z\\[-3pt]
\approx{}&\biggl(\prod_{i=1}^n z_it_i\biggr)zxt\biggl(\prod_{i=n+1}^{n+m+1} z_it_i\biggr)x\biggl(\prod_{i=1}^{n+m+1} z_{i\tau}\biggr)z
\end{align*}
which coincides (up to renaming of letters) with the identity~\eqref{c_{n,m+1}[tau]=c_{n,m+1}'[tau]}; 
\item[(iii)] if $\mathbf a=yxz$, then $\mathbf X$ satisfies the identity
\begin{align*}
&\biggl(\prod_{i=1}^{n+1} z_i^\prime t_i^\prime\biggr)yzt\biggl(\prod_{i=n+2}^{n+m+1} z_i^\prime t_i^\prime\biggr)y\biggl(\prod_{i=1}^{n+m+1} z_{i\rho}^\prime\biggr)z\\[-3pt]
\approx{}&\biggl(\prod_{i=1}^{n+1} z_i^\prime t_i^\prime\biggr)zyt\biggl(\prod_{i=n+2}^{n+m+1} z_i^\prime t_i^\prime\biggr)y\biggl(\prod_{i=1}^{n+m+1} z_{i\rho}^\prime\biggr)z
\end{align*}
which coincides (up to renaming of letters) with the identity~\eqref{c_{n+1,m}[rho]=c_{n+1,m}'[rho]}. 
\end{itemize}
In either case we have a contradiction with Lemma~\ref{Lem: S(W) in V}.
\end{proof}

\section{Certain $fi$-permutable varieties}
\label{Sec: fi-perm}

In view of Lemma~\ref{Lem: lifting}, to verify that a monoid variety $\mathbf V$ is $fi$-permutable, it suffices to prove that if $\mathbf X,\mathbf Y\subseteq\mathbf V$ and $(\mathbf u,\mathbf v)\in\theta_{\mathbf X}\vee\theta_{\mathbf Y}$ (equivalently, an identity $\mathbf u\approx\mathbf v$ holds in $\mathbf X\wedge\mathbf Y$), then $(\mathbf u,\mathbf v)\in\theta_{\mathbf X}\theta_{\mathbf Y}$. We will use this argument many times throughout the rest of the article without explicitly mentioning it.

The section is divided into four subsections each of which is devoted to proving the $fi$-permutability of one of the varieties $\mathbf D_\infty\vee\mathbf N$, $\mathbf P_n$, $\mathbf Q_{r,s}$ or $\mathbf R$.

\subsection{The variety $\mathbf D_\infty\vee\mathbf N$}
\label{Subsec: D_infty vee N is fi-perm}
The lattice $L(\mathbf D_\infty\vee\mathbf N)$ has a very simple structure and admits of an exhaustive description, which will be obtained below (see Corollary~\ref{Cor: L(D_infty vee N)}).

\begin{lemma}
\label{Lem: subvarieties of D_infty vee N}
Any variety from the interval $[\mathbf D_2,\mathbf D_\infty\vee\mathbf N]$ may be given within the variety $\mathbf D_\infty\vee\mathbf N$ by the identities~\eqref{xytxy=yxtxy}, $\sigma_1$ or $\delta_k$ with $k\ge 2$.
\end{lemma}

\begin{proof}
It is evident that $\mathbf D_\infty\vee\mathbf N\subseteq\mathbf A\wedge\mathbf O\{\alpha_1,\,\alpha_2,\,\alpha_3\}$. Lemmas~\ref{Lem: subvarieties of O} and~\ref{Lem: Straubing identities in A} imply that any variety from the interval $[\mathbf D_2,\mathbf D_\infty\vee\mathbf N]$ may be given within the variety $\mathbf D_\infty\vee\mathbf N$ by the identity $\delta_k$ with $k\ge 2$ and the identities of the form~\eqref{two letters in a block} with $r\in\mathbb N_0$, $e_0,f_0\in\mathbb N$, $e_1,f_1,\dots,e_r,f_r\in\mathbb N_0$, $\sum_{i=0}^r e_i\ge 2$ and $\sum_{i=0}^r f_i\ge 2$. If either at least one of the numbers $e_0,f_0,e_1,f_1,\dots,e_r,f_r$ is greater than~1 or $\sum_{i=0}^r e_i>2$ or $\sum_{i=0}^r f_i>2$, then the identity~\eqref{two letters in a block} holds in $\mathbf D_\infty\vee\mathbf N$. Therefore, if an identity of the form~\eqref{two letters in a block} fails in $\mathbf D_\infty\vee\mathbf N$, then this identity coincides (up to renaming of letters) with either $\sigma_1$ or~\eqref{xytxy=yxtxy}. 
\end{proof}

Since the identities $\delta_\ell$ with $\ell<k$, $\sigma_1$ and~\eqref{xytxy=yxtxy} fail in the varieties $\mathbf D_k$, $\mathbf M$ and $\mathbf N$, respectively, and $\sigma_1$ implies~\eqref{xytxy=yxtxy}, Lemma~\ref{Lem: subvarieties of D_infty vee N} implies the following assertion.

\begin{corollary}
\label{Cor: D_k vee N and D_k vee M}
If $2\le k\le\infty$, then $\mathbf D_k\vee\mathbf N=\mathbf A\wedge\mathbf O\{\alpha_1,\,\alpha_2,\,\alpha_3,\,\delta_k\}$ and $\mathbf D_k\vee\mathbf M=\mathbf A\wedge\mathbf O\{\alpha_1,\,\alpha_2,\,\alpha_3,\,\delta_k,\,\eqref{xytxy=yxtxy}\}$.\qed
\end{corollary}

\begin{corollary}
\label{Cor: L(D_infty vee N)}
The lattice $L(\mathbf D_\infty\vee\mathbf N)$ has the form shown in Fig.~\ref{Fig: L(D_infty vee N)}.
\end{corollary}

\begin{proof}
In view of Lemma~\ref{Lem: L(A)}(i), it suffices to check that the interval $[\mathbf D_2,\mathbf D_\infty\vee\mathbf N]$ is as in Fig.~\ref{Fig: L(D_infty vee N)}. According to Lemma~\ref{Lem: L(A)}(ii), the interval $[\mathbf D_2,\mathbf D_\infty\vee\mathbf N]$ is a disjoint union of the intervals of the form $[\mathbf D_k,(\mathbf D_\infty\vee\mathbf N)\{\delta_k\}]$, where $2\le k\le\infty$. Now Lemmas~\ref{Lem: subvarieties of D_infty vee N} and~\ref{Lem: basis for D_k} and Corollary~\ref{Cor: D_k vee N and D_k vee M} apply and we conclude that the interval $[\mathbf D_k,(\mathbf D_\infty\vee\mathbf N)\{\delta_k\}]$ is the chain $\mathbf D_k\subseteq\mathbf D_k\vee\mathbf M\subseteq\mathbf D_k\vee\mathbf N$. All the varieties in this chain are distinct because $\sigma_1$ holds in $\mathbf D_k$ but fails in $\mathbf M$, while~\eqref{xytxy=yxtxy} holds in $\mathbf D_k\vee\mathbf M$ but fails in $\mathbf N$.
\end{proof}

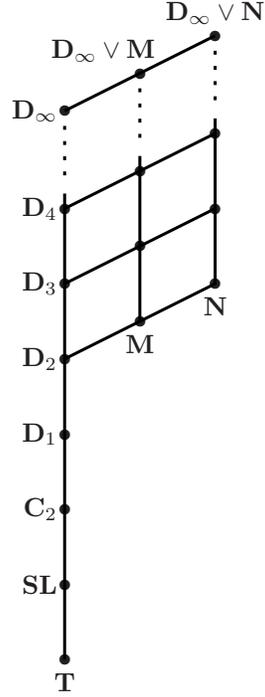
\begin{figure}[htb]
\unitlength=1mm
\linethickness{0.4pt}
\begin{center}
\begin{picture}(34,93)
\put(7,5){\circle*{1.33}}
\put(7,15){\circle*{1.33}}
\put(7,25){\circle*{1.33}}
\put(7,35){\circle*{1.33}}
\put(7,45){\circle*{1.33}}
\put(7,55){\circle*{1.33}}
\put(7,65){\circle*{1.33}}
\put(7,78){\circle*{1.33}}
\put(17,50){\circle*{1.33}}
\put(17,60){\circle*{1.33}}
\put(17,70){\circle*{1.33}}
\put(17,83){\circle*{1.33}}
\put(27,55){\circle*{1.33}}
\put(27,65){\circle*{1.33}}
\put(27,75){\circle*{1.33}}
\put(27,88){\circle*{1.33}}
\gasset{AHnb=0,linewidth=0.4}
\drawline(7,5)(7,67)
\drawline(7,45)(27,55)(27,77)
\drawline(7,55)(27,65)
\drawline(7,65)(27,75)
\drawline(7,78)(27,88)
\drawline(17,50)(17,72)
\drawline[dash={0.52 1.5}{0}](7,68)(7,76)
\drawline[dash={0.52 1.5}{0}](17,73)(17,81)
\drawline[dash={0.52 1.5}{0}](27,78)(27,86)
\put(6,25){\makebox(0,0)[rc]{$\mathbf C_2$}}
\put(6,35){\makebox(0,0)[rc]{$\mathbf D_1$}}
\put(6,45){\makebox(0,0)[rc]{$\mathbf D_2$}}
\put(6,55){\makebox(0,0)[rc]{$\mathbf D_3$}}
\put(6,65){\makebox(0,0)[rc]{$\mathbf D_4$}}
\put(6,78){\makebox(0,0)[rc]{$\mathbf D_\infty$}}
\put(19,86){\makebox(0,0)[rc]{$\mathbf D_\infty\vee\mathbf M$}}
\put(27,91){\makebox(0,0)[cc]{$\mathbf D_\infty\vee\mathbf N$}}
\put(17,47){\makebox(0,0)[cc]{$\mathbf M$}}
\put(27,52){\makebox(0,0)[cc]{$\mathbf N$}}
\put(6,15){\makebox(0,0)[rc]{$\mathbf{SL}$}}
\put(7,2){\makebox(0,0)[cc]{$\mathbf T$}}
\end{picture}
\end{center}
\caption{The lattice $L(\mathbf D_\infty\vee\mathbf N)$}
\label{Fig: L(D_infty vee N)}
\end{figure}

\begin{proposition}
\label{Prop: D_infty vee N is fi-perm}
The variety $\mathbf D_\infty\vee\mathbf N$ is $fi$-permutable.
\end{proposition}

\begin{proof}
Let $\mathbf X,\mathbf Y\subseteq\mathbf D_\infty\vee\mathbf N$ and an identity $\mathbf u\approx\mathbf v$ holds in $\mathbf X\wedge\mathbf Y$. We need to check that $(\mathbf u,\mathbf v)\in\theta_{\mathbf X}\theta_{\mathbf Y}$. 

In view of Lemma~\ref{Lem: comparable is fi-permut} and Corollary~\ref{Cor: L(D_infty vee N)}, we may assume without loss of generality that $\mathbf X=\mathbf D_k\vee\mathbf U$ and $\mathbf Y=\mathbf D_\ell\vee\mathbf W$, where $2\le\ell<k\le\infty$, $\mathbf U,\mathbf W\in\{\mathbf D_2,\mathbf M,\mathbf N\}$ and $\mathbf U\subset\mathbf W$. Lemma~\ref{Lem: smth imply w_n[pi,tau]=w_n'[pi,tau]} implies that $\mathbf D_\infty\vee\mathbf N\subseteq\mathbf A^\ast$. Now Lemmas~\ref{Lem: L(A)}(ii) and~\ref{Lem: reduction to linear-balanced} apply and we conclude that there are a linear-balanced identity $\mathbf u^\prime\approx\mathbf v^\prime$ that holds in $\mathbf X\wedge\mathbf Y$ and a word $\mathbf p$ such that $\mathbf u\,\theta_{\mathbf X}\,\mathbf p\mathbf u^\prime$ and $\mathbf p\mathbf v^\prime\,\theta_{\mathbf Y}\,\mathbf v$. It follows from Corollary~\ref{Cor: L(D_infty vee N)} that $\mathbf X\wedge\mathbf Y=\mathbf D_\ell\vee\mathbf U$. It can be easily verified that an arbitrary linear-balanced identity holds in $\mathbf D_\infty$ and therefore, in $\mathbf D_k$. Therefore, the identity $\mathbf u^\prime\approx\mathbf v^\prime$ holds in $\mathbf D_k\vee(\mathbf D_\ell\vee\mathbf U)=\mathbf D_k\vee\mathbf U=\mathbf X$, whence $\mathbf u^\prime\,\theta_{\mathbf X}\,\mathbf v^\prime$. Then $\mathbf u\,\theta_{\mathbf X}\,\mathbf p\mathbf u^\prime\,\theta_{\mathbf X}\,\mathbf p\mathbf v^\prime\,\theta_{\mathbf Y}\,\mathbf v$, whence $(\mathbf u,\mathbf v)\in\theta_{\mathbf X}\theta_{\mathbf Y}$, and we are done.
\end{proof}

\subsection{The variety $\mathbf P_n$}
\label{Subsec: P_n is fi-perm}
We put $\mathbf P=\var\{\delta_1\}$.

\begin{lemma}
\label{Lem: subvarieties of P_n}
Any subvariety of the variety $\mathbf P$ can be given within $\mathbf P$ by a finite number of the following identities:
\begin{align}
\label{x^k=x^ell}
x^k&\approx x^\ell,\\
\label{x^ky^ell=y^ell x^k}
x^ky^\ell&\approx y^\ell x^k,
\end{align} 
where $k,\ell\in\mathbb N$.
\end{lemma}

\begin{proof}
By results of Head~\cite{Head-68}, every commutative monoid variety can be given by the identities $xy\approx yx$ and~\eqref{x^k=x^ell} for some $k,\ell\in\mathbb N$. It remains to prove the required assertion for non-commutative subvarieties of the variety $\mathbf P$. To achieve this goal, it suffices to verify that if an identity $\mathbf u\approx\mathbf v$ does not imply the commutative law, then it is equivalent in the variety $\mathbf P$ to some system of identities of the form either~\eqref{x^k=x^ell} or~\eqref{x^ky^ell=y^ell x^k}. In view of Lemma~\ref{Lem: subvarieties of O} and the inclusion $\mathbf P\subseteq\mathbf O$, we may assume that one of the following two statements holds:
\begin{itemize}
\item[(a)] the identity $\mathbf u\approx\mathbf v$ coincides with an efficient identity of the form~\eqref{one letter in a block} with $r,e_0,f_0,e_1,f_1,\dots,e_r,f_r\in\mathbb N_0$;
\item[(b)] the identity $\mathbf u\approx\mathbf v$ coincides with an efficient identity of the form~\eqref{two letters in a block} with $r\in\mathbb N_0$, $e_0,f_0\in\mathbb N$, $e_1,f_1,\dots,e_r,f_r\in\mathbb N_0$, $\sum_{i=0}^r e_i\ge 2$ and $\sum_{i=0}^r f_i\ge 2$.
\end{itemize}
Put $e=\sum_{i=0}^r e_i$ and $f=\sum_{i=0}^r f_i$.

Suppose that the claim~(a) holds. Since the identity~\eqref{one letter in a block} is efficient, we may assume without loss of generality that $e_0>0$. The variety $\mathbf P\{\eqref{one letter in a block}\}$ satisfies the identity $x^e\approx x^f$. If $f_0>0$, then $\mathbf P\{\eqref{one letter in a block}\}=\mathbf P\{x^e\approx x^f\}$, and we are done. Finally, if $f_0=0$, then $\mathbf P\{\eqref{one letter in a block}\}$ satisfies the identity $x^et\approx tx^f$. It is easy to see that in this case $\mathbf P\{\eqref{one letter in a block}\}=\mathbf P_n\{x^e\approx x^f,\,x^et\approx tx^f\}$, and we are done again.

Suppose now that the claim~(b) holds. Since 
$$
x^ey^f\stackrel{\delta_1}\approx x^{e_0}y^{f_0}\biggl(\prod_{i=1}^r x^{e_i}y^{f_i}\biggr)\stackrel{\eqref{two letters in a block}}\approx y^{f_0}x^{e_0}\biggl(\prod_{i=1}^r x^{e_i}y^{f_i}\biggr)\stackrel{\delta_1}\approx y^fx^e,
$$
the variety $\mathbf P\{\eqref{two letters in a block}\}$ satisfies the identity $x^ey^f\approx y^fx^e$. On the other hand, the identity~\eqref{two letters in a block} holds in $\mathbf P\{x^ey^f\approx y^fx^e\}$ because this variety satisfies the identities
$$
x^{e_0}y^{f_0}\biggl(\prod_{i=1}^r t_ix^{e_i}y^{f_i}\biggr)\stackrel{\delta_1}\approx x^ey^f\cdot\prod_{i=1}^r t_i\approx y^fx^e\cdot\prod_{i=1}^r t_i\stackrel{\delta_1}\approx y^{f_0}x^{e_0}\biggl(\prod_{i=1}^r t_ix^{e_i}y^{f_i}\biggr).
$$
Therefore, $\mathbf P\{\eqref{two letters in a block}\}=\mathbf P\{x^ey^f\approx y^fx^e\}$, and we are done.
\end{proof}

\begin{lemma}
\label{Lem: P_n{x^ky^ell=y^ell x^k} subset P_n{x^py^q=y^qx^p}}
If $k\le\ell<n$ and $p\le q<n$, then $\mathbf P_n\{\eqref{x^ky^ell=y^ell x^k}\}\subseteq\mathbf P_n\{x^py^q\approx y^q x^p\}$ if and only if $k\le p$ and $\ell\le q$.
\end{lemma}

\begin{proof}
\emph{Necessity}. Suppose that $\mathbf P_n\{\eqref{x^ky^ell=y^ell x^k}\}\subseteq\mathbf P_n\{x^py^q\approx y^q x^p\}$. Consider an identity $x^py^q\approx\mathbf w$ of $\mathbf P_n$. Clearly, $\con(\mathbf w)=\{x,y\}$, $\occ_x(\mathbf w)=p$, $\occ_y(\mathbf w)=q$ and $h(\mathbf w)=x$. It follows that $\var\{\eqref{x^ky^ell=y^ell x^k}\}$ satisfies a non-trivial identity $\mathbf w\approx\mathbf w^\prime$. But it is easy to see that if either $p<k$ or $q<\ell$, then $\mathbf w$ is an isoterm for $\var\{\eqref{x^ky^ell=y^ell x^k}\}$. Therefore, $k\le p$ and $\ell\le q$.

\smallskip

\emph{Sufficiency} follows from the fact that the identities
$$
x^py^q\stackrel{\delta_1}\approx x^ky^\ell x^{p-k}y^{q-\ell}\stackrel{\eqref{x^ky^ell=y^ell x^k}}\approx y^\ell x^k x^{p-k}y^{q-\ell}\stackrel{\delta_1}\approx y^qx^p
$$
hold in $\mathbf P_n\{\eqref{x^ky^ell=y^ell x^k}\}$.
\end{proof}

\begin{lemma}
\label{Lem: x^ky^ell=y^ell x^k in X wedge Y}
Let $\mathbf X$ and $\mathbf Y$ be subvarieties of the variety $\mathbf P_n$. If the variety $\mathbf X\wedge\mathbf Y$ satisfies the identity~\eqref{x^ky^ell=y^ell x^k} for some $1\le k,\ell<n$, then this identity holds in either $\mathbf X$ or $\mathbf Y$.
\end{lemma}

\begin{proof}
Suppose that $x^k$ is not an isoterm for $\mathbf X$. Then Lemma~\ref{Lem: x^n is isoterm} applies with the conclusion that $\mathbf X$ satisfies the identity $x^k\approx x^r$ for some $r>k$. Since $\mathbf X\subseteq\mathbf P_n$, this implies that $\mathbf X$ satisfies the identities $x^ky^\ell\approx x^ny^\ell\approx y^\ell x^n\approx y^\ell x^k$, and we are done. Thus, we may assume that $x^k$ is an isoterm for $\mathbf X$. Analogously, we may assume that $y^\ell$ is an isoterm for $\mathbf X$ as well. By symmetry, $x^k$ and $y^\ell$ are isoterms for $\mathbf Y$ too.

Since the identity~\eqref{x^ky^ell=y^ell x^k} holds in $\mathbf X\wedge\mathbf Y$, there is a sequence of words $\mathbf w_0$, $\mathbf w_1$, \dots, $\mathbf w_m$ such that $\mathbf w_0=x^ky^\ell$, $\mathbf w_m=y^\ell x^k$ and, for each $i=0,1,\dots,m-1$, the identity $\mathbf w_i\approx\mathbf w_{i+1}$ holds in either $\mathbf X$ or $\mathbf Y$. The claim that $x^k$ and $y^\ell$ are isoterms for $\mathbf X$ and $\mathbf Y$ imply that $\con(\mathbf w_i)=\{x,y\}$, $\occ_x(\mathbf w_i)=k$ and $\occ_y(\mathbf w_i)=\ell$ for any $i=0,1,\dots,m$. Evidently, there is $j\in\{0,1,\dots,m-1\}$ such that $h(\mathbf w_j)=x$ but $h(\mathbf w_{j+1})=y$. The identity $\mathbf w_j\approx\mathbf w_{j+1}$ holds in either $\mathbf X$ or $\mathbf Y$. Taking into account that the varieties $\mathbf X$ and $\mathbf Y$ satisfy the identity $\delta_1$, we have that the identity $\mathbf w_j\approx\mathbf w_{j+1}$ is equivalent to~\eqref{x^ky^ell=y^ell x^k} in one of the varieties $\mathbf X$ or $\mathbf Y$.
\end{proof}

\begin{corollary}
\label{Cor: L(P_n) is distributive}
For any $n\in\mathbb N$, the lattice $L(\mathbf P_n)$ is distributive.
\end{corollary}

\begin{proof}
Lemmas~\ref{Lem: subvarieties of P_n},~\ref{Lem: x^n=x^m in X wedge Y} and~\ref{Lem: x^ky^ell=y^ell x^k in X wedge Y} show that it suffices to refer to Lemma~\ref{Lem: smth imply distributivity} with $\mathbf V=\mathbf P_n$, $\mathbf W=\mathbf T$ and the identity system $\Sigma$ that consists of the identities~\eqref{x^k=x^ell} and~\eqref{x^ky^ell=y^ell x^k} for all $k,\ell\in\mathbb N$.
\end{proof}

\begin{proposition}
\label{Prop: P_n is fi-perm}
For any $n\in\mathbb N$, the variety $\mathbf P_n$ is $fi$-permutable.
\end{proposition}

\begin{proof}
Let $\mathbf X,\mathbf Y\subseteq\mathbf P_n$ and an identity $\mathbf u\approx\mathbf v$ holds in $\mathbf X\wedge\mathbf Y$. We need to check that $(\mathbf u,\mathbf v)\in\theta_{\mathbf X}\theta_{\mathbf Y}$.

We may assume that both $\mathbf X$ and $\mathbf Y$ are non-trivial because the required conclusion is evident otherwise. Lemma~\ref{Lem: group variety} and the fact that $\mathbf P_n$ is aperiodic imply that $\con(\mathbf u)=\con(\mathbf v)$. Then, since $\mathbf P_n$ satisfies the identity $\delta_1$, we may assume that $\mathbf u=x_1^{s_1}x_2^{s_2}\cdots x_m^{s_m}$ and $\mathbf v=x_{1\pi}^{t_{1\pi}}x_{2\pi}^{t_{2\pi}}\cdots x_{m\pi}^{t_{m\pi}}$ for some $m\in\mathbb N$ and some permutation $\pi\in S_m$. The variety $\mathbf X\wedge\mathbf Y$ satisfies the identity $x_i^{s_i}\approx x_i^{t_i}$ for each $i=1,2,\dots,m$. Lemma~\ref{Lem: x^n=x^m in X wedge Y} implies that this identity holds in either $\mathbf X$ or $\mathbf Y$. We put $\mathbf u^\prime=x_1^{r_1}x_2^{r_2}\cdots x_m^{r_m}$ and $\mathbf v^\prime=x_{1\pi}^{r_{1\pi}}x_{2\pi}^{r_{2\pi}}\cdots x_{m\pi}^{r_{m\pi}}$, where
$$
r_i=
\begin{cases}
s_i&\text{if the identity}\ x^{s_i}\approx x^{t_i}\ \text{holds in}\ \mathbf Y,\\
t_i&\text{if the identity}\ x^{s_i}\approx x^{t_i}\ \text{holds in}\ \mathbf X\ \text{but does not hold in}\ \mathbf Y
\end{cases}
$$
for any $i=1,2,\dots,m$. Then $\mathbf u\,\theta_{\mathbf X}\,\mathbf u^\prime$ and $\mathbf v^\prime\,\theta_{\mathbf Y}\,\mathbf v$. It remains to prove that $(\mathbf u^\prime,\mathbf v^\prime)\in\theta_{\mathbf X}\theta_{\mathbf Y}$. Indeed, in this case $\mathbf u\,\theta_{\mathbf X}\,\mathbf u^\prime\,\theta_{\mathbf X}\,\mathbf w\,\theta_{\mathbf Y}\mathbf v^\prime\,\theta_{\mathbf Y}\,\mathbf v$ for some word $\mathbf w$ and therefore, $(\mathbf u,\mathbf v)\in\theta_{\mathbf X}\theta_{\mathbf Y}$.

By Lemma~\ref{Lem: x^ky^ell=y^ell x^k in X wedge Y}, for any $i<j$, the identity $\mathbf u^\prime(x_i,x_j)\approx\mathbf v^\prime(x_i,x_j)$ holds in either $\mathbf X$ or $\mathbf Y$. Let us step by step as long as possible apply to the word $\mathbf v^\prime $ non-trivial identities of the kind $\mathbf u^\prime(x_i,x_j)\approx\mathbf v^\prime(x_i,x_j)$ that hold in $\mathbf Y$ as follows: at each step, we will replace some subword of the form $x_j^{ r_j}x_i^{r_i}$ of the word $\mathbf v^\prime$ to the subword $x_i^{r_i}x_j^{r_j}$ whenever $i<j$ and the identity $x_j^{ r_j}x_i^{r_i}\approx x_i^{r_i}x_j^{r_j}$ holds in $\mathbf Y$. As a result, we obtain the word $\mathbf w_1=x_{1\tau}^{r_{1\tau}}x_{2\tau}^{r_{2\tau}}\cdots x_{m\tau}^{r_{m\tau}}$ for some permutation $\tau\in S_m$. Clearly, $\mathbf w_1\,\theta_{\mathbf Y}\,\mathbf v^\prime$. If $\mathbf w_1=\mathbf u^\prime$, then $(\mathbf u^\prime,\mathbf v^\prime)\in\theta_{\mathbf Y}\subseteq\theta_{\mathbf X}\theta_{\mathbf Y}$. Suppose now that $\mathbf w_1\ne\mathbf u^\prime$. 

Now we will step by step as long as possible apply to the word $\mathbf w_1$ non-trivial identities of the kind $\mathbf u^\prime(x_i,x_j)\approx\mathbf v^\prime(x_i,x_j)$ that hold in $\mathbf X$ by the same way as above: at each step, we will replace some subword of the form $x_j^{ r_j}x_i^{r_i}$ of the word $\mathbf w_1$ to the subword $x_i^{r_i}x_j^{r_j}$ whenever $i<j$ and the identity $x_j^{ r_j}x_i^{r_i}\approx x_i^{r_i}x_j^{r_j}$ holds in $\mathbf X$. As a result, we obtain some word $\mathbf w_2$. It suffices to verify that $\mathbf w_2=\mathbf u^\prime$ because $\mathbf u^\prime=\mathbf w_2\,\theta_{\mathbf X}\,\mathbf w_1\,\theta_{\mathbf Y}\,\mathbf v^\prime$ and therefore, $(\mathbf u^\prime,\mathbf v^\prime)\in\theta_{\mathbf X}\theta_{\mathbf Y}$ in this case. 

Arguing by contradiction, we suppose that $\mathbf w_2\ne\mathbf u^\prime$. Then there are indexes $a$ and $b$ such that $a<b$ and $x_b^{r_b}x_a^{r_a}$ is a subword of the word $\mathbf w_2$. Then $b=p\tau$ and $a=q\tau$ for some $p<q$. Therefore, 
$$
\mathbf w_1=\biggl(\prod_{i=1}^{p-1} x_{i\tau}^{r_{i\tau}}\biggr)x_b^{r_b}\biggl(\prod_{i=p+1}^{q-1} x_{i\tau}^{r_{i\tau}}\biggr)x_a^{r_a}\biggl(\prod_{i=q+1}^m x_{i\tau}^{r_{i\tau}}\biggr).
$$
The identity $x^{r_a}y^{r_b}\approx y^{r_b}x^{r_a}$ fails in $\mathbf X$ by the definition of the word $\mathbf w_2$. Then it holds in $\mathbf Y$ by Lemma~\ref{Lem: x^ky^ell=y^ell x^k in X wedge Y}. Suppose that $q=p+1$. Then we can apply the identity $x^{r_a}y^{r_b}\approx y^{r_b}x^{r_a}$ to the word $\mathbf w_1$ and replace the subword $x_b^{r_b}x_a^{r_a}$ to the subword $x_a^{r_a}x_b^{r_b}$. As a result, we obtain the word
$$
\mathbf w_1^\prime=\biggl(\prod_{i=1}^{p-1} x_{i\tau}^{r_{i\tau}}\biggr)x_a^{r_a}x_b^{r_b}\biggl(\prod_{i=q+1}^m x_{i\tau}^{r_{i\tau}}\biggr)
$$
such that the identity $\mathbf v^\prime\approx\mathbf w_1^\prime$ holds in $\mathbf Y$. But this is impossible by the definition of the word $\mathbf w_1$. Therefore, $p+1<q$.

We will assume without loss of generality that $r_b\le r_a$ (the case when $r_a<r_b$ can be considered quite analogously). For each $i=p+1,p+2,\dots,q-1$, the variety $\mathbf X$ satisfies either the identity $x^{r_{i\tau}}y^{r_b}\approx y^{r_b}x^{r_{i\tau}}$ or the identity $x^{r_a}y^{r_{i\tau}}\approx y^{r_{i\tau}}x^{r_a}$. This claim, Lemma~\ref{Lem: P_n{x^ky^ell=y^ell x^k} subset P_n{x^py^q=y^qx^p}} and the inequality $r_b\le r_a$ imply that $\mathbf X$ satisfies the identity $x^{r_a}y^{r_{i\tau}}\approx y^{r_{i\tau}}x^{r_a}$ for each $i=p+1,p+2,\dots,q-1$. Since the identity $x^{r_a}y^{r_b}\approx y^{r_b}x^{r_a}$ fails in $\mathbf X$, we can apply Lemma~\ref{Lem: P_n{x^ky^ell=y^ell x^k} subset P_n{x^py^q=y^qx^p}} again and obtain that $r_{i\tau}>r_b$ for each $i=p+1,p+2,\dots,q-1$. Since the identity $x^{r_a}y^{r_b}\approx y^{r_b}x^{r_a}$ holds in the variety $\mathbf Y$, Lemma~\ref{Lem: P_n{x^ky^ell=y^ell x^k} subset P_n{x^py^q=y^qx^p}} implies that this variety satisfies also the identity $x^{r_a}y^{r_{i\tau}}\approx y^{r_{i\tau}}x^{r_a}$ for each $i=p+1,p+2,\dots,q-1$. This implies that $(q-1)\tau<a$ because we have a contradiction with the definition of the word $\mathbf w_1$ and the fact that $\mathbf Y$ satisfies the identity $x^{r_a}y^{r_{(q-1)\tau}}\approx y^{r_{(q-1)\tau}}x^{r_a}$ otherwise.

Suppose that $i\tau<a$ for each $i=p+1,p+2,\dots,q-1$. Then $(p+1)\tau<a<b$. By the definition of the word $\mathbf w_1$, this implies that $\mathbf Y$ violates the identity $x^{r_{(p+1)\tau}}y^{r_b}\approx y^{r_b}x^{r_{(p+1)\tau}}$. Then $r_{(p+1)\tau}<r_a$ by Lemma~\ref{Lem: P_n{x^ky^ell=y^ell x^k} subset P_n{x^py^q=y^qx^p}}. In view of Lemma~\ref{Lem: x^ky^ell=y^ell x^k in X wedge Y}, the identity $x^{r_{(p+1)\tau}}y^{r_b}\approx y^{r_b}x^{r_{(p+1)\tau}}$ holds in $\mathbf X$. But this is impossible because this identity implies $x^{r_a}y^{r_b}\approx y^{r_b}x^{r_a}$ by Lemma~\ref{Lem: P_n{x^ky^ell=y^ell x^k} subset P_n{x^py^q=y^qx^p}}.

Finally, suppose that $j\tau>a$ for some $j\in\{p+1,p+2,\dots,q-1\}$. Let $j$ be the largest number with such a property. Put $d=j\tau$. Then there is $c\le a$ such that $x_d^{r_d}x_c^{r_c}$ is a subword of the word $\mathbf w_1$. The definition of the word $\mathbf w_1$ implies that the variety $\mathbf Y$ does not satisfy the identity $x^{r_d}y^{r_c}\approx y^{r_c}x^{r_d}$. Since $r_d\ge r_b$ and the identity $x^{r_a}y^{r_b}\approx y^{r_b}x^{r_a}$ is true in $\mathbf Y$, Lemma~\ref{Lem: P_n{x^ky^ell=y^ell x^k} subset P_n{x^py^q=y^qx^p}} implies that $r_c<r_a$. Therefore, $c<a$. Then the definition of the word $\mathbf w_2$ and the inequalities $c<a<b$ imply that the variety $\mathbf X$ satisfies the identity $x^{r_c}y^{r_b}\approx y^{r_b}x^{r_c}$. Now Lemma~\ref{Lem: P_n{x^ky^ell=y^ell x^k} subset P_n{x^py^q=y^qx^p}} applies and we obtain a contradiction with the inequality $r_c<r_a$ and the fact that the identity $x^{r_a}y^{r_b}\approx y^{r_b}x^{r_a}$ fails in $\mathbf X$. This contradiction completes the proof.
\end{proof}

\subsection{The variety $\mathbf Q_{r,s}$}
\label{Subsec: Q_{r,s} is fi-perm}

\begin{proposition}
\label{Prop: Q_{r,s} is fi-perm}
For any $1\le r,s\le 3$, the variety $\mathbf Q_{r,s}$ is $fi$-permutable.
\end{proposition}

\begin{proof}
Let $\mathbf X,\mathbf Y\subseteq\mathbf Q_{r,s}$ and an identity $\mathbf u\approx\mathbf v$ holds in $\mathbf X\wedge\mathbf Y$. We need to check that $(\mathbf u,\mathbf v)\in\theta_{\mathbf X}\theta_{\mathbf Y}$.

In view of Lemmas~\ref{Lem: comparable is fi-permut} and~\ref{Lem: L(A)}(i), we may assume that $\mathbf X,\mathbf Y\in[\mathbf D_2,\mathbf Q_{r,s}]$. Lemma~\ref{Lem: smth imply w_n[pi,tau]=w_n'[pi,tau]} implies that $\mathbf Q_{r,s}\subseteq\mathbf A^\ast$. Then, by Lemmas~\ref{Lem: L(A)}(ii) and~\ref{Lem: reduction to linear-balanced}, there are a linear-balanced identity $\mathbf u^\prime\approx\mathbf v^\prime$ that holds in $\mathbf X\wedge\mathbf Y$ and a word $\mathbf p$ such that $\mathbf u\,\theta_{\mathbf X}\,\mathbf p\mathbf u^\prime$ and $\mathbf p\mathbf v^\prime\,\theta_{\mathbf Y}\,\mathbf v$. It suffices to verify that $(\mathbf u^\prime,\mathbf v^\prime)\in\theta_{\mathbf X}\theta_{\mathbf Y}$. Indeed, in this case there is a word $\mathbf w$ such that $\mathbf u^\prime\,\theta_{\mathbf X}\,\mathbf w\,\theta_{\mathbf Y}\,\mathbf v^\prime$. Then $\mathbf u\,\theta_{\mathbf X}\,\mathbf p\mathbf u^\prime\,\theta_{\mathbf X}\,\mathbf p\mathbf w\,\theta_{\mathbf Y}\,\mathbf p\mathbf v^\prime\,\theta_{\mathbf Y}\,\mathbf v$, whence $(\mathbf u,\mathbf v)\in\theta_{\mathbf X}\theta_{\mathbf Y}$, and we are done. This allows us to suppose below that the identity $\mathbf u\approx\mathbf v$ is linear-balanced.

Clearly, we may assume that $\mathbf u\ne\mathbf v$ because the required conclusion is evident otherwise. Let~\eqref{decomposition of u} and~\eqref{decomposition of v} be the decompositions of the words $\mathbf u$ and $\mathbf v$, respectively. The varieties $\mathbf X$ and $\mathbf Y$ satisfy the identity $\sigma_3$. Therefore, we may assume that, for each $i=0,1,\dots,m$, $\mathbf u_i=\mathbf p_i\mathbf q_i\mathbf r_i$ and $\mathbf v_i=\mathbf p_i^\prime\mathbf q_i^\prime\mathbf r_i^\prime$, where $\mathbf p_i$ and $\mathbf p_i^\prime$ consist of the first occurrences of letters in the words $\mathbf u$ and $\mathbf v$, respectively; $\mathbf q_i$ and $\mathbf q_i^\prime$ consist of non-first and non-last occurrences of letters in the words $\mathbf u$ and $\mathbf v$, respectively; finally, $\mathbf r_i$ and $\mathbf r_i^\prime$ consist of the last occurrences of letters in the words $\mathbf u$ and $\mathbf v$, respectively. 

Let $i\in\{0,1,\dots,m\}$. Clearly, $\con(\mathbf p_i)=\con(\mathbf p_i^\prime)$. Suppose that $\mathbf p_i\ne\mathbf p_i^\prime$. Then there are letters $x,y\in\con(\mathbf p_i)$ such that $x$ precedes $y$ in $\mathbf p_i$ but $y$ precedes $x$ in $\mathbf p_i^\prime$. Let 
$$
T_{x,y}=\{t_j\mid 1\le j\le m,\,\con(\mathbf u_j)\cap\{x,y\}\ne\varnothing\}.
$$
Clearly, the variety $\mathbf X\wedge\mathbf Y$ satisfies the identity~\eqref{2 mult letters and their dividers}. Suppose that $\mathbf a_0t_{i_1}\mathbf a_1t_{i_2}\mathbf a_2\cdots t_{i_k}\mathbf a_k$ and $\mathbf b_0t_{i_1}\mathbf b_1t_{i_2}\mathbf b_2\cdots t_{i_k}\mathbf b_k$ are decompositions of the words $\mathbf u(x,y,T_{x,y})$ and $\mathbf v(x,y,T_{x,y})$, respectively. Then $\con(\mathbf a_i)=\con(\mathbf b_i)$ for all $i=0,1,\dots,k$. Further, $x,y\in\con(\mathbf a_0)$ and $x,y\in\con(\mathbf b_0)$ by the choice of letters $x$ and $y$. We may assume without loss of generality that $\mathbf a_0=xy$. Then $\mathbf b_0=yx$ by the choice of $x$ and $y$. The variety $\mathbf Q_{r,s}$ satisfies the identity $\mathbf d_{0,0}[\pi]\approx\mathbf d_{0,0}^\prime[\pi]$, that is, the identity $xytxy\approx xytyx$. Thus, if $x,y\in\con(\mathbf a_i)=\con(\mathbf b_i)$ for some $i>0$, then we may assume that $\mathbf a_i=\mathbf b_i=xy$. Therefore, the identity~\eqref{2 mult letters and their dividers} coincides with an identity of the form $xyt_1\mathbf a_1t_2\mathbf a_2\cdots t_k\mathbf a_k\approx yxt_1\mathbf a_1t_2\mathbf a_2\cdots t_k\mathbf a_k$, where $\mathbf a_i\in\{x,y,xy\}$ for all $i=1,2,\dots,k$.{\sloppy

}
For any $r=1,2,3$, we put $\Phi_r=\{\eqref{xytxy=yxtxy},\,\alpha_i\mid i=1,2,3,\,i\ne r\}$. The identity~\eqref{xytxy=yxtxy} is nothing but the identity $\mathbf c_{0,0}[\pi]\approx\mathbf c_{0,0}^\prime[\pi]$. Hence the variety $\mathbf Q_{r,s}$ satisfies the identity system $\Phi_r$. Put
$$
\Sigma_r=
\begin{cases}
\{xyt_1xt_2x\cdots t_nxty\approx yxt_1xt_2x\cdots t_nxty\mid n\in\mathbb N\}&\text{if}\ r=1,\\
\{\sigma_1,\,\alpha_2\}&\text{if}\ r=2,\\ 
\{xytyt_1xt_2x\cdots t_nx\approx yxtyt_1xt_2x\cdots t_nx\mid n\in\mathbb N\}&\text{if}\ r=3.
\end{cases}
$$
It is easy to see that, for each $r=1,2,3$, an identity of the form~\eqref{2 mult letters and their dividers} either follows from the identity system $\Phi_r$ (and therefore, holds in $\mathbf Q_{r,s}$) or coincides (up to renaming of letters) with some identity from the identity system $\Sigma_r$. 

The set of all identities of the form~\eqref{2 mult letters and their dividers} implies the identities
\begin{equation}
\label{two letters and their dividers imply smth}
\mathbf u\approx\mathbf p_0^\prime\mathbf q_0\mathbf r_0\biggl(\prod_{i=1}^m t_i\mathbf p_i^\prime\mathbf q_i\mathbf r_i\biggr)\quad\text{and}\quad\mathbf v\approx\mathbf p_0\mathbf q_0^\prime\mathbf r_0^\prime\biggl(\prod_{i=1}^m t_i\mathbf p_i\mathbf q_i^\prime\mathbf r_i^\prime\biggr).
\end{equation}
Thus, these two identities follow from the identities that hold in the variety $\mathbf Q_{r,s}$ and some (possibly empty) subsystem $\Gamma$ of the system $\Sigma_r$. Suppose that $\Gamma\ne\varnothing$. Then the definition of $\Sigma_r$ implies that the set of all varieties of the form $\var\{\gamma\}$ with $\gamma\in\Gamma$ forms a chain with the least element. Therefore, $\Gamma$ is equivalent to a single identity $\gamma\in\Gamma$, whence the identities~\eqref{two letters and their dividers imply smth} hold in the variety $\mathbf Q_{r,s}\{\gamma\}$. 

The identity $\gamma$ coincides with some identity of the form~\eqref{2 mult letters and their dividers}, whence it holds in $\mathbf X\wedge\mathbf Y$. Now Corollary~\ref{Cor: pxyq=pyxq in X wedge Y} applies and we conclude that the identity $\gamma$ holds in either $\mathbf X$ or $\mathbf Y$, say, in $\mathbf X$. Then $\mathbf X\subseteq\mathbf Q_{r,s}\{\gamma\}$, whence $\mathbf X$ satisfies~\eqref{two letters and their dividers imply smth}. The same is evidently true whenever $\Gamma=\varnothing$. Thus, $\mathbf X$ satisfies~\eqref{two letters and their dividers imply smth} in either case. Analogous arguments show that one of the varieties $\mathbf X$ or $\mathbf Y$ satisfies the identity 
$$
\mathbf p_0^\prime\mathbf q_0\mathbf r_0\biggl(\prod_{i=1}^m t_i\mathbf p_i^\prime\mathbf q_i\mathbf r_i\biggr)\approx\mathbf p_0^\prime\mathbf q_0\mathbf r_0^\prime\biggl(\prod_{i=1}^m t_i\mathbf p_i^\prime\mathbf q_i\mathbf r_i^\prime\biggr).
$$ 
Finally, since the variety $\mathbf Y$ satisfies the identity $\sigma_3$, this variety satisfies also the identity
$$
\mathbf p_0^\prime\mathbf q_0\mathbf r_0^\prime\biggl(\prod_{i=1}^m t_i\mathbf p_i^\prime\mathbf q_i\mathbf r_i^\prime\biggr)\approx\mathbf v.
$$ 
Therefore,
$$
\mathbf u\,\theta_{\mathbf X}\,\mathbf p_0^\prime\mathbf q_0\mathbf r_0\biggl(\prod_{i=1}^m t_i\mathbf p_i^\prime\mathbf q_i\mathbf r_i\biggr)\,\theta_{\mathbf Z}\,\mathbf p_0^\prime\mathbf q_0\mathbf r_0^\prime\biggl(\prod_{i=1}^m t_i\mathbf p_i^\prime\mathbf q_i\mathbf r_i^\prime\biggr)\,\theta_{\mathbf Y}\,\mathbf v,
$$
where $\mathbf Z\in\{\mathbf X,\mathbf Y\}$. Hence $(\mathbf u,\mathbf v)\in\theta_{\mathbf X}\theta_{\mathbf Y}$, and we are done.
\end{proof}

\subsection{The variety $\mathbf R$}
\label{Subsec: R is fi-perm}
For any $k,\ell\in\mathbb N$, we fix the following notation for an identity:
$$
\gamma_{k,\ell}:\enskip\biggl(\prod_{i=1}^k xt_i\biggr)xy\biggl(\prod_{i=k+1}^{k+\ell} t_iy\biggr)\approx\biggl(\prod_{i=1}^k xt_i\biggr)yx\biggl(\prod_{i=k+1}^{k+\ell} t_iy\biggr).
$$

The following evident observation will be very useful.

\begin{lemma}
\label{Lem: gamma_{e,f} implies pxyq=pyxq}
Let $\mathbf p$ and $\mathbf q$ be words, $x\in\con(\mathbf p)$ and $y\in\con(\mathbf q)$. If $e\le\occ_x(\mathbf p)$ and $f\le\occ_y(\mathbf q)$, then the identity $\gamma_{e,f}$ implies the identity $\mathbf pxy\mathbf q\approx\mathbf pyx\mathbf q$.\qed
\end{lemma}

\begin{lemma}
\label{Lem: R{gamma_{k,ell}} subset R{gamma_{p,q}}}
For any $k,\ell,p,q\in\mathbb N$, the inclusion $\mathbf R\{\gamma_{k,\ell}\}\subseteq\mathbf R\{\gamma_{p,q}\}$ holds if and only if $k\le p$ and $\ell\le q$.
\end{lemma}

\begin{proof}
\emph{Necessity} follows from the fact that if either $p<k$ or $q<\ell$, then the word 
$$
\biggl(\prod_{i=1}^p xt_i\biggr)xy\biggl(\prod_{i=p+1}^{p+q} t_iy\biggr)
$$
is an isoterm for the variety $\mathbf R\{\gamma_{k,\ell}\}$.

\smallskip

\emph{Sufficiency} follows from Lemma~\ref{Lem: gamma_{e,f} implies pxyq=pyxq}.
\end{proof}

\emph{Throughout Lemmas~\ref{Lem: smth imply gamma_{e,f}}--\ref{Lem: v'av''=v'bv''} and their proofs, $\mathbf u\approx\mathbf v$ is a fixed linear-balanced identity,~\eqref{decomposition of u} and~\eqref{decomposition of v} are decompositions of the words $\mathbf u$ and $\mathbf v$, respectively, $i$ is a fixed number with $1\le i\le m$,}
\begin{equation}
\label{u',u'',v',v''}
\mathbf u^\prime=\prod_{j=0}^{i-1}\mathbf u_j,\ \mathbf u^{\prime\prime}=\prod_{j=i+1}^m\mathbf u_j,\ \mathbf v^\prime=\prod_{j=0}^{i-1}\mathbf v_j\enskip \text{and}\enskip\mathbf v^{\prime\prime}=\prod_{j=i+1}^m\mathbf v_j.
\end{equation}
Thus,
\begin{equation}
\label{u=u'u_iu'',v=v'v_iv''}
\mathbf u=\mathbf u^\prime\mathbf u_i\mathbf u^{\prime\prime}\quad\text{and}\quad\mathbf v=\mathbf v^\prime\mathbf v_i\mathbf v^{\prime\prime}.
\end{equation}

\begin{lemma}
\label{Lem: smth imply gamma_{e,f}}
Let $\mathbf X$ be a monoid variety and the identity $\mathbf u\approx\mathbf v$ holds in $\mathbf X$. Suppose that a letter $x$ precedes a letter $y$ in the block $\mathbf u_i$ but $y$ precedes $x$ in the block $\mathbf v_i$. If $x\in\con(\mathbf u^\prime)\setminus\con(\mathbf u^{\prime\prime})$ and $y\in\con(\mathbf u^{\prime\prime})\setminus\con(\mathbf u^\prime)$, then $\mathbf X$ satisfies the identity $\gamma_{e,f}$, where $e=\occ_x(\mathbf u)-1$ and $f=\occ_y(\mathbf u)-1$.
\end{lemma}

\begin{proof}
Let $s$ be the least number such that $\con(\mathbf u_s)\cap\{x,y\}\ne\varnothing$ and $T_{x,y}$ be the set of letters defined by the equality~\eqref{T_{x,y}}. Clearly, the variety $\mathbf X$ satisfies the identity~\eqref{2 mult letters and their dividers}. Since the identity $\mathbf u\approx\mathbf v$ is linear-balanced, the identity~\eqref{2 mult letters and their dividers} coincides (up to renaming of letters) with the identity $\gamma_{e,f}$, where $e=\occ_x(\mathbf u)-1$ and $f=\occ_y(\mathbf u)-1$, and we are done.
\end{proof}

\begin{lemma}
\label{Lem: from r_1r_2xr_3 to r_1xr_2r_3}
Let $\mathbf X,\mathbf Y\in[\mathbf D_2,\mathbf R]$ and the identity $\mathbf u\approx\mathbf v$ holds in $\mathbf X\wedge\mathbf Y$. Suppose that $\mathbf v_i=\mathbf r_1\mathbf r_2x\mathbf r_3$ for some words $\mathbf r_1$, $\mathbf r_2$, $\mathbf r_3$ and a letter $x$. If the letter $x$ precedes each letter from $\con(\mathbf r_2)$ in the block $\mathbf u_i$, then one of the varieties $\mathbf X$ or $\mathbf Y$ satisfies the identity $\mathbf v\approx\mathbf v^\prime\mathbf r_1x\mathbf r_2\mathbf r_3\mathbf v^{\prime\prime}$.
\end{lemma}

\begin{proof}
If $\mathbf r_2=\lambda$, then the required conclusion is evident. Let now $\mathbf r_2\ne\lambda$. Suppose that $x\in\con(\mathbf v^\prime)\cap\con(\mathbf v^{\prime\prime})$. Since letters from $\mathbf r_2$ are multiple in $\mathbf v$, $\con(\mathbf r_2)\subseteq\con(\mathbf v^\prime\mathbf v^{\prime\prime})$. Then the identity $\mathbf v\approx\mathbf v^\prime\mathbf r_1x\mathbf r_2\mathbf r_3\mathbf v^{\prime\prime}$ follows from $\sigma_1$ and $\sigma_2$, whence it holds in $\mathbf X$. Therefore, we may assume without loss of generality that $x\in\con(\mathbf v^\prime)\setminus\con(\mathbf v^{\prime\prime})$. We can write the word $\mathbf r_2$ in the form $\mathbf r_2=\mathbf p_1\mathbf q_1\mathbf p_2\mathbf q_2\cdots\mathbf p_n\mathbf q_n$ for some words $\mathbf p_1,\mathbf q_1,\mathbf p_2,\mathbf q_2,\dots,\mathbf p_n,\mathbf q_n$ such that $\con(\mathbf p_1\mathbf p_2\cdots\mathbf p_n)\subseteq\con(\mathbf v^\prime)$ and $\con(\mathbf q_1\mathbf q_2\cdots\mathbf q_n)\cap\con(\mathbf v^\prime)=\varnothing$. Let $y$ be a letter from $\con(\mathbf q_1\mathbf q_2\cdots\mathbf q_n)$ with the least number of occurrences in the word $\mathbf v^{\prime\prime}$ among all letters from $\con(\mathbf q_1\mathbf q_2\cdots\mathbf q_n)$. According to Lemma~\ref{Lem: smth imply gamma_{e,f}}, $\mathbf X\wedge\mathbf Y$ satisfies $\gamma_{e,f}$, where $e=\occ_x(\mathbf v)-1$ and $f=\occ_y(\mathbf v)-1$. In view of Corollary~\ref{Cor: pxyq=pyxq in X wedge Y}, the identity $\gamma_{e,f}$ holds in either $\mathbf X$ or $\mathbf Y$, say, in $\mathbf X$. In view of Lemma~\ref{Lem: gamma_{e,f} implies pxyq=pyxq}, the choice of the letter $y$ allows us, using the identity $\gamma_{e, f}$, to swap the letter $x$ and the letter immediately to the left of $x$ whenever that adjacent letter lies in $\con(\mathbf q_1\mathbf q_2\cdots\mathbf q_n)$. The identity $\sigma_2$ allows us to do the same whenever the letter immediately to the left of $x$ lies in $\con(\mathbf p_1\mathbf p_2\cdots\mathbf p_n)$. Therefore, $\mathbf X$ satisfies the identities
\begin{align*}
\mathbf v\stackrel{\gamma_{e,f}}\approx&\mathbf v^\prime\mathbf r_1\biggl(\prod_{j=1}^{n-1} \mathbf p_j\mathbf q_j\biggr)\mathbf p_nx\mathbf q_n\mathbf r_3\mathbf v^{\prime\prime}\stackrel{\sigma_2}\approx\mathbf v^\prime\mathbf r_1\biggl(\prod_{j=1}^{n-1} \mathbf p_j\mathbf q_j\biggr)x\mathbf p_n\mathbf q_n\mathbf r_3\mathbf v^{\prime\prime}\\[-3pt]
\stackrel{\gamma_{e,f}}\approx&\cdots\stackrel{\gamma_{e,f}}\approx\mathbf v^\prime\mathbf r_1\mathbf p_1x\mathbf q_1\biggl(\prod_{j=2}^n \mathbf p_j\mathbf q_j\biggr)\mathbf r_3\mathbf v^{\prime\prime}\stackrel{\sigma_2}\approx\mathbf v^\prime\mathbf r_1x\mathbf r_2\mathbf r_3\mathbf v^{\prime\prime}.
\end{align*}
This completes the proof.
\end{proof}

Let $\mathbf X$ be a monoid variety. A pair of letters $(a,b)$ is called $\mathbf X$-\emph{invertible in the block $\mathbf u_i$ of the word $\mathbf u$ in the identity $\mathbf u\approx\mathbf v$} (or \mbox{$(\mathbf X,\mathbf u,\mathbf v,i)$}-\emph{invertible}, for short) if $\mathbf u_i=\mathbf cab\mathbf d$ for some words $\mathbf c$ and $\mathbf d$, the letter $b$ precedes the letter $a$ in the block $\mathbf v_i$ and $\mathbf X$ satisfies the identity $\mathbf u\approx\mathbf u^\prime\mathbf cba\mathbf d\mathbf u^{\prime\prime}$. 

\begin{lemma}
\label{Lem: v'av''=v'bv''}
Let $\mathbf X$ be a subvariety of the variety $\mathbf R$. Suppose that, for some letters $x$ and $y$, the word $xy$ is a subword of $\mathbf u_i$, while $y$ precedes $x$ in $\mathbf v_i$. Let $\mathbf a$ and $\mathbf b$ be linear words with $\con(\mathbf a)=\con(\mathbf b)=\con(\mathbf u_i)$. If $x$ precedes $y$ in $\mathbf a$ but $y$ precedes $x$ in $\mathbf b$ and the identity $\mathbf v^\prime\mathbf a\mathbf v^{\prime\prime}\approx\mathbf v^\prime\mathbf b\mathbf v^{\prime\prime}$ holds in $\mathbf X$, then the pair $(x,y)$ is \mbox{$(\mathbf X,\mathbf u,\mathbf v,i)$}-invertible.
\end{lemma}

\begin{proof}
By the hypothesis, $\mathbf u_i=\mathbf p_1xy\mathbf p_2$ for some words $\mathbf p_1$ and $\mathbf p_2$. If $x,y\in\con(\mathbf u^\prime)$ [respectively, $x,y\in\con(\mathbf u^{\prime\prime})$], then the identity $\mathbf u\approx\mathbf u^\prime\mathbf p_1yx\mathbf p_2\mathbf u^{\prime\prime}$ follows from $\sigma_2$ [respectively, $\sigma_1$], whence it holds in $\mathbf X$ and, therefore, the pair $(x,y)$ is \mbox{$(\mathbf X,\mathbf u,\mathbf v,i)$}-invertible. Thus, we may assume without loss of generality that $x\in\con(\mathbf u^\prime)\setminus\con(\mathbf u^{\prime\prime})$ and $y\in\con(\mathbf u^{\prime\prime})\setminus\con(\mathbf u^\prime)$. Then $x\in\con(\mathbf v^\prime)\setminus\con(\mathbf v^{\prime\prime})$ and $y\in\con(\mathbf v^{\prime\prime})\setminus\con(\mathbf v^\prime)$ because the identity $\mathbf u\approx\mathbf v$ is linear-balanced. Since the identity $\mathbf v^\prime\mathbf a\mathbf v^{\prime\prime}\approx\mathbf v^\prime\mathbf b\mathbf v^{\prime\prime}$ is linear-balanced and holds in $\mathbf X$, Lemma~\ref{Lem: smth imply gamma_{e,f}} implies that $\mathbf X$ satisfies the identity $\gamma_{e,f}$, where $e=\occ_x(\mathbf u)-1$ and $f=\occ_y(\mathbf u)-1$. Then $\mathbf X$ satisfies $\mathbf u\stackrel{\gamma_{e,f}}\approx\mathbf u^\prime\mathbf p_1yx\mathbf p_2\mathbf u^{\prime\prime}$ by Lemma~\ref{Lem: gamma_{e,f} implies pxyq=pyxq}, whence the pair $(x,y)$ is \mbox{$(\mathbf X,\mathbf u,\mathbf v,i)$}-invertible.
\end{proof}

\begin{proposition}
\label{Prop: R is fi-perm}
The variety $\mathbf R$ is $fi$-permutable.
\end{proposition}

\begin{proof}
Let $\mathbf X,\mathbf Y\subseteq\mathbf R$ and an identity $\mathbf u\approx\mathbf v$ holds in $\mathbf X\wedge\mathbf Y$. We need to check that $(\mathbf u,\mathbf v)\in\theta_{\mathbf X}\theta_{\mathbf Y}$.

Since $\mathbf R\subseteq\mathbf A^\ast$, we can repeat literally arguments from the second paragraph of the proof of Proposition~\ref{Prop: Q_{r,s} is fi-perm} and reduce our considerations to the case when $\mathbf X,\mathbf Y\in[\mathbf D_2,\mathbf R]$ and the identity $\mathbf u\approx\mathbf v$ is linear-balanced.

Let~\eqref{decomposition of u} be the decomposition of the word $\mathbf u$. Since the identity $\mathbf u\approx\mathbf v$ is linear-balanced, the decomposition of the word $\mathbf v$ has the form~\eqref{decomposition of v} and the identity $\mathbf u\approx\mathbf v$ is $n$-invertible for some $n\in\mathbb N_0$ (the notion of $n$-invertible identity was introduced in Subsection~\ref{Subsec: A'}). We will use induction on $n$.

\smallskip

\emph{Induction base}. If $n=0$, then $\mathbf u=\mathbf v$, and we are done.

\smallskip

\emph{Induction step}. Let now $n>0$. Since $\mathbf u\ne\mathbf v$, there is $i\in\{0,1,\dots,m\}$ with $\mathbf u_i\ne\mathbf v_i$. Let us fix a number $i$ with such a property. Let $\mathbf u^\prime$, $\mathbf u^{\prime\prime}$, $\mathbf v^\prime$ and $\mathbf v^{\prime\prime}$ be words defined by the equalities~\eqref{u',u'',v',v''}. Thus, the equalities~\eqref{u=u'u_iu'',v=v'v_iv''} are valid. Suppose that the block $\mathbf u_i$ of the word $\mathbf u$ contains an \mbox{$(\mathbf X,\mathbf u,\mathbf v,i)$}-invertible pair of letters $(a,b)$. This means that $\mathbf u_i=\mathbf cab\mathbf d$ for some words $\mathbf c$ and $\mathbf d$, the identity $\mathbf u\approx\mathbf u^\prime\mathbf cba\mathbf d\mathbf u^{\prime\prime}$ holds in $\mathbf X$ and the letter $b$ precedes the letter $a$ in the block $\mathbf v_i$. The identity $\mathbf u^\prime\mathbf cba\mathbf d\mathbf u^{\prime\prime}\approx\mathbf v$ is \mbox{$(n-1)$}-invertible and holds in $\mathbf X\wedge\mathbf Y$. By the induction assumption, $(\mathbf u^\prime\mathbf cba\mathbf d\mathbf u^{\prime\prime},\mathbf v)\in\theta_{\mathbf X}\theta_{\mathbf Y}$. This implies that $\mathbf u\,\theta_{\mathbf X}\,\mathbf u^\prime\mathbf cba\mathbf d\mathbf u^{\prime\prime}\,\theta_{\mathbf X}\,\mathbf w\,\theta_{\mathbf Y}\,\mathbf v$ for some word $\mathbf w$, whence $(\mathbf u,\mathbf v)\in\theta_{\mathbf X}\theta_{\mathbf Y}$. 

Thus, we may assume that
\begin{itemize}
\item[(a)] the block $\mathbf u_i$ does not contain \mbox{$(\mathbf X,\mathbf u,\mathbf v,i)$}-invertible pairs of letters.
\end{itemize}
Analogous arguments allows us to assume that 
\begin{itemize}
\item[(b)] the block $\mathbf v_i$ does not contain \mbox{$(\mathbf Y,\mathbf v,\mathbf u,i)$}-invertible pairs of letters.
\end{itemize}

In order to facilitate understanding of further considerations, we outline their general scheme. We want to find the words $\mathbf u^\sharp$ and $\mathbf v^\sharp$ such that the identity $\mathbf u^\sharp\approx\mathbf v^\sharp$ is $t$-invertible for some $t<n$ and holds in the variety $\mathbf X\wedge\mathbf Y$, while the varieties $\mathbf X$ and $\mathbf Y$ satisfy the identities $\mathbf u\approx\mathbf u^\sharp$ and $\mathbf v\approx\mathbf v^\sharp$, respectively. This is sufficient for our aims. Indeed, if such words $\mathbf u^\sharp$ and $\mathbf v^\sharp$ exist, then $(\mathbf u^\sharp,\mathbf v^\sharp)\in\theta_{\mathbf X}\theta_{\mathbf Y}$ by the induction assumption. Then $\mathbf u\,\theta_{\mathbf X}\,\mathbf u^\sharp\,\theta_{\mathbf X}\,\mathbf w\,\theta_{\mathbf Y}\,\mathbf v^\sharp\,\theta_{\mathbf Y}\,\mathbf v$ for some word $\mathbf w$, whence $(\mathbf u,\mathbf v)\in\theta_{\mathbf X}\theta_{\mathbf Y}$, and we are done.

To understand, how to construct the sought words $\mathbf u^\sharp$ and $\mathbf v^\sharp$ from the words $\mathbf u$ and $\mathbf v$, we need to clarify the structure of the words $\mathbf u$ and $\mathbf v$ whenever the claims~(a) and~(b) hold. As we will see, these two claims impose rather strict restrictions on the words $\mathbf u$ and $\mathbf v$. After we find that the claims~(a) and~(b) make many cases impossible, we will eventually come to the conclusion that the words $\mathbf u$ and $\mathbf v$ must have some well-defined form, which will allows us to define the words $\mathbf u^\sharp$ and $\mathbf v^\sharp$ with the desired properties.

Let $\mathbf q$ be the largest common prefix of the words $\mathbf u_i$ and $\mathbf v_i$. Then $\mathbf u_i=\mathbf qx_0\mathbf u_i^\prime y_0\mathbf u_i^{\prime\prime}$ and $\mathbf v_i=\mathbf qy_0\mathbf v_i^\prime x_0\mathbf v_i^{\prime\prime}$ for some different letters $x_0,y_0$ and some words $\mathbf u_i^\prime,\mathbf u_i^{\prime\prime},\mathbf v_i^\prime$ and $\mathbf v_i^{\prime\prime}$. Lemma~\ref{Lem: from r_1r_2xr_3 to r_1xr_2r_3} implies that one of the varieties $\mathbf X$ and $\mathbf Y$ satisfies the identity 
\begin{equation}
\label{u=u'qy_0x_0u_i''u''}
\mathbf u\approx\mathbf u^\prime\mathbf qy_0x_0\mathbf u_i^\prime\mathbf u_i^{\prime\prime}\mathbf u^{\prime\prime}.
\end{equation} 
Suppose that this identity holds in $\mathbf X$. Let $z=x_0$ whenever $\mathbf u_i^\prime=\lambda$ and $z$ be the last letter of the word $\mathbf u_i^\prime$ otherwise. Then the pair $(y_0,z)$ is \mbox{$(\mathbf X,\mathbf u,\mathbf v,i)$}-invertible. This is impossible by the claim~(a). Therefore, the identity~\eqref{u=u'qy_0x_0u_i''u''} holds in $\mathbf Y$. If $\mathbf v_i^\prime=\lambda$, then Lemma~\ref{Lem: v'av''=v'bv''} implies that the pair $(x_0,y_0)$ is \mbox{$(\mathbf Y,\mathbf v,\mathbf u,i)$}-invertible, contradicting with the claim~(b). Therefore, $\mathbf v_i^\prime\ne\lambda$. Analogous arguments show that $\mathbf u_i^\prime\ne\lambda$. Let $\mathbf u_i^\prime=x_1x_2\cdots x_p$ and $\mathbf v_i^\prime=y_1y_2\cdots y_q$. Thus,
$$
\mathbf u=\mathbf u^\prime\mathbf qx_0x_1\cdots x_py_0\mathbf u_i^{\prime\prime}\mathbf u^{\prime\prime}\quad\text{and}\quad\mathbf v=\mathbf v^\prime\mathbf qy_0y_1\cdots y_qx_0\mathbf v_i^{\prime\prime}\mathbf v^{\prime\prime}.
$$ 

Put $X_j=\{x_1,x_2,\dots,x_j\}$ for all $j=1,2,\dots,p$ and $Y_j=\{y_1,y_2,\dots,y_j\}$ for all $j=1,2,\dots,q$. We are going to verify that either $y_j$ precedes $y_{j+1}$ in $\mathbf u_i$ for any $j=0,1,\dots,q-1$ or $x_j$ precedes $x_{j+1}$ in $\mathbf v_i$ for any $j=0,1,\dots,p-1$. Suppose that this is not the case. Then there are $r\in\{0,1,\dots,p-1\}$ and $s\in\{0,1,\dots,q-1\}$ such that $x_{r+1}$ precedes $x_r$ in $\mathbf v_i$ and $y_{s+1}$ precedes $y_s$ in $\mathbf u_i$. We may assume without any loss that $r$ and $s$ are the least numbers with such properties. It follows that $X_r\subseteq\con(\mathbf v_i^{\prime\prime})$ and $Y_s\subseteq\con(\mathbf u_i^{\prime\prime})$. Then we apply Lemma~\ref{Lem: from r_1r_2xr_3 to r_1xr_2r_3} $s+r+2$ times and obtain that the variety $\mathbf X\wedge\mathbf Y$ satisfies the identities
\begin{align*}
\mathbf u&\approx\mathbf u^\prime\mathbf qy_0x_0\mathbf u_i^\prime\mathbf u_i^{\prime\prime}\mathbf u^{\prime\prime}\approx\mathbf u^\prime\mathbf qy_0y_1x_0\mathbf u_i^\prime(\mathbf u_i^{\prime\prime})_{Y_1}\mathbf u^{\prime\prime}\approx\mathbf u^\prime\mathbf qy_0y_1y_2x_0\mathbf u_i^\prime(\mathbf u_i^{\prime\prime})_{Y_2}\mathbf u^{\prime\prime}\\
&\approx\cdots\approx\mathbf u^\prime\mathbf qy_0y_1y_2\cdots y_sx_0\mathbf u_i^\prime (\mathbf u_i^{\prime\prime})_{Y_s}\mathbf u^{\prime\prime}\quad\text{and}\\
\mathbf v&\approx\mathbf v^\prime\mathbf qx_0y_0\mathbf v_i^\prime\mathbf v_i^{\prime\prime}\mathbf v^{\prime\prime}\approx\mathbf v^\prime\mathbf qx_0x_1y_0\mathbf v_i^\prime(\mathbf v_i^{\prime\prime})_{X_1}\mathbf v^{\prime\prime}\approx\mathbf v^\prime\mathbf qx_0x_1x_2y_0\mathbf v_i^\prime(\mathbf v_i^{\prime\prime})_{X_2}\mathbf v^{\prime\prime}\\
&\approx\cdots\approx\mathbf v^\prime\mathbf qx_0x_1x_2\cdots x_ry_0\mathbf v_i^\prime (\mathbf v_i^{\prime\prime})_{X_r}\mathbf v^{\prime\prime}.
\end{align*}
In particular, the variety $\mathbf X\wedge\mathbf Y$ satisfies the identities
\begin{align}
\label{s identity u}
&\mathbf u^\prime\mathbf qy_0y_1\cdots y_{s-1}x_0\mathbf u_i^\prime(\mathbf u_i^{\prime\prime})_{Y_{s-1}}\mathbf u^{\prime\prime}\approx\mathbf u^\prime\mathbf qy_0y_1\cdots y_sx_0\mathbf u_i^\prime(\mathbf u_i^{\prime\prime})_{Y_s}\mathbf u^{\prime\prime},\\
\label{r identity v}
&\mathbf v^\prime\mathbf qx_0x_1\cdots x_{r-1}y_0\mathbf v_i^\prime(\mathbf v_i^{\prime\prime})_{X_{r-1}}\mathbf v^{\prime\prime}\approx\mathbf v^\prime\mathbf qx_0x_1\cdots x_ry_0\mathbf v_i^\prime(\mathbf v_i^{\prime\prime})_{X_r}\mathbf v^{\prime\prime}.
\end{align}

If the identity~\eqref{s identity u} holds in the variety $\mathbf Y$, then the pair $(y_s,y_{s+1})$ is \mbox{$(\mathbf Y,\mathbf v,\mathbf u,i)$}-invertible by Lemma~\ref{Lem: v'av''=v'bv''}, contradicting with the claim~(b). Therefore, the identity~\eqref{s identity u} fails in $\mathbf Y$ and so holds in $\mathbf X$ by Lemma~\ref{Lem: from r_1r_2xr_3 to r_1xr_2r_3}. Analogously, the identity~\eqref{r identity v} holds in $\mathbf Y$ but fails in $\mathbf X$.

Further, if $y_s,y_{s+1}\in\con(\mathbf v^\prime)$ or $y_s,y_{s+1}\in\con(\mathbf v^{\prime\prime})$, then, since $\mathbf Y$ satisfies $\sigma_1$ and $\sigma_2$, the pair $(y_s,y_{s+1})$ is \mbox{$(\mathbf Y,\mathbf v,\mathbf u,i)$}-invertible, contradicting with the claim~(b). Therefore, we may assume without loss of generality that $y_s\in\con(\mathbf v^\prime)\setminus\con(\mathbf v^{\prime\prime})$ and $y_{s+1}\in\con(\mathbf v^{\prime\prime})\setminus\con(\mathbf v^\prime)$. 

In view of Lemma~\ref{Lem: smth imply gamma_{e,f}}, $\mathbf X$ satisfies the identity $\gamma_{f_s,f_{s+1}}$. If this identity holds in the variety $\mathbf Y$, then, by Lemma~\ref{Lem: gamma_{e,f} implies pxyq=pyxq}, the pair $(y_s,y_{s+1})$ is \mbox{$(\mathbf Y,\mathbf v,\mathbf u,i)$}-invertible contradicting with the claim~(b). Therefore, the identity $\gamma_{f_s,f_{s+1}}$ fails in $\mathbf Y$.

Analogously, we can verify that one of the following two claims holds: 
\begin{itemize}
\item[(c)]$x_r\in\con(\mathbf v^\prime)\setminus\con(\mathbf v^{\prime\prime})$, $x_{r+1}\in\con(\mathbf v^{\prime\prime})\setminus\con(\mathbf v^\prime)$ and the identity $\gamma_{e_r,e_{r+1}}$ fails in $\mathbf X$ and holds in $\mathbf Y$;
\item[(d)]$x_{r+1}\in\con(\mathbf v^\prime)\setminus\con(\mathbf v^{\prime\prime})$, $x_r\in\con(\mathbf v^{\prime\prime})\setminus\con(\mathbf v^\prime)$ and the identity $\gamma_{e_{r+1},e_r}$ fails in $\mathbf X$ and holds in $\mathbf Y$. 
\end{itemize}

First, suppose that the claim~(c) is true. Since $\mathbf X$ satisfies the identity~\eqref{s identity u}, the identity $\gamma_{f_s,e_{r+1}}$ holds in $\mathbf X$ by Lemma~\ref{Lem: smth imply gamma_{e,f}}. Then Lemma~\ref{Lem: R{gamma_{k,ell}} subset R{gamma_{p,q}}} and the fact that $\mathbf X$ violates $\gamma_{e_r,e_{r+1}}$ imply that $e_r<f_s$. Since $\mathbf Y$ satisfies the identity~\eqref{r identity v}, the identity $\gamma_{e_r,f_{s+1}}$ holds in $\mathbf Y$ by Lemma~\ref{Lem: smth imply gamma_{e,f}}. Then Lemma~\ref{Lem: R{gamma_{k,ell}} subset R{gamma_{p,q}}} and the fact that $\mathbf Y$ violates $\gamma_{f_s,f_{s+1}}$ imply that $f_s<e_r$, a contradiction. So, the claim~(c) is false.

Suppose now that the claim~(d) is true. Since $\mathbf X$ satisfies the identity~\eqref{s identity u}, the identity $\gamma_{f_s,e_r}$ holds in $\mathbf X$ by Lemma~\ref{Lem: smth imply gamma_{e,f}}. Then Lemma~\ref{Lem: R{gamma_{k,ell}} subset R{gamma_{p,q}}} and the fact that $\mathbf X$ violates $\gamma_{e_{r+1},e_r}$ imply that $e_{r+1}<f_s$. Since $\mathbf Y$ satisfies the identity~\eqref{r identity v}, the identity $\gamma_{f_s,e_r}$ holds in $\mathbf Y$ by Lemma~\ref{Lem: smth imply gamma_{e,f}}. Then Lemma~\ref{Lem: R{gamma_{k,ell}} subset R{gamma_{p,q}}} and the fact that $\mathbf Y$ violates $\gamma_{f_s,f_{s+1}}$ imply that $f_{s+1}<e_r$.

Further, if $x_{r+1}\in\con(\mathbf v_i^\prime)$, then $x_{r+1}=y_c$ for some $c\in\{1,2,\dots,q\}$. This means that $y_c\in\con(\mathbf u_i^\prime)$. Then $s+1<c$ because $Y_s\subseteq\con(\mathbf u_i^{\prime\prime})$ and $x_{r+1}\ne y_{s+1}$. It follows that $y_{s+1}$ precedes $x_{r+1}$ in $\mathbf v_i$. If $x_{r+1}\notin\con(\mathbf v_i^\prime)$, then, evidently, $y_{s+1}$ precedes $x_{r+1}$ in $\mathbf v_i$. So, $y_{s+1}$ precedes $x_{r+1}$ in $\mathbf v_i$ in either case. Analogously, we can verify that $x_{r+1}$ precedes $y_{s+1}$ in $\mathbf u_i$. Then $\mathbf X\wedge\mathbf Y$ satisfies the identity $\gamma_{e_{r+1},f_{s+1}}$ by Lemma~\ref{Lem: smth imply gamma_{e,f}}. According to Corollary~\ref{Cor: pxyq=pyxq in X wedge Y}, either $\mathbf X$ or $\mathbf Y$ satisfies this identity. If $\gamma_{e_{r+1},f_{s+1}}$ holds in $\mathbf X$, then $\gamma_{e_{r+1},e_r}$ holds in $\mathbf X$ too by Lemma~\ref{Lem: R{gamma_{k,ell}} subset R{gamma_{p,q}}}, a contradiction. If $\gamma_{e_{r+1},f_{s+1}}$ holds in $\mathbf Y$, then $\gamma_{f_s,f_{s+1}}$ holds in $\mathbf Y$ too by Lemma~\ref{Lem: R{gamma_{k,ell}} subset R{gamma_{p,q}}}, a contradiction again. So, the claim~(d) is false too.

Thus, the hypothesis that there are $r\in\{0,1,\dots,p-1\}$ and $s\in\{0,1,\dots,q-1\}$ such that $x_{r+1}$ precedes $x_r$ in $\mathbf v_i$ and $y_{s+1}$ precedes $y_s$ in $\mathbf u_i$ is false in either case. Then either $y_j$ precedes $y_{j+1}$ in $\mathbf u_i$ for any $j=0,1,\dots,q-1$ or $x_j$ precedes $x_{j+1}$ in $\mathbf v_i$ for any $j=0,1,\dots,p-1$. By symmetry, we may assume that $y_j$ precedes $y_{j+1}$ in $\mathbf u_i$ for any $j=0,1,\dots,q-1$. In particular, $X_p\cap Y_q=\varnothing$. It follows that $X_p\subseteq\con(\mathbf v_i^{\prime\prime})$ and $Y_q\subseteq\con(\mathbf u_i^{\prime\prime})$. Considerations similar to the deduction of the identities~\eqref{s identity u} and~\eqref{r identity v} allow us to conclude that the variety $\mathbf X\wedge\mathbf Y$ satisfies the identities
\begin{align}
\label{q identity u}
&\mathbf u^\prime\mathbf qy_0y_1\cdots y_{q-1}x_0\mathbf u_i^\prime(\mathbf u_i^{\prime\prime})_{Y_{q-1}}\mathbf u^{\prime\prime}\approx\mathbf u^\prime\mathbf qy_0y_1\cdots y_qx_0\mathbf u_i^\prime(\mathbf u_i^{\prime\prime})_{Y_q}\mathbf u^{\prime\prime},\\
\label{p identity v}
&\mathbf v^\prime\mathbf qx_0x_1\cdots x_{p-1}y_0\mathbf v_i^\prime(\mathbf v_i^{\prime\prime})_{X_{p-1}}\mathbf v^{\prime\prime}\approx\mathbf v^\prime\mathbf qx_0x_1\cdots x_py_0\mathbf v_i^\prime(\mathbf v_i^{\prime\prime})_{X_p}\mathbf v^{\prime\prime}.
\end{align}

If the identity~\eqref{q identity u} holds in the variety $\mathbf Y$, then the pair $(y_{q-1},y_q)$ is \mbox{$(\mathbf Y,\mathbf v,\mathbf u,i)$}-invertible by Lemma~\ref{Lem: v'av''=v'bv''}, contradicting with the claim~(b). Therefore, the identity~\eqref{q identity u} fails in $\mathbf Y$ and so holds in $\mathbf X$ by Lemma~\ref{Lem: from r_1r_2xr_3 to r_1xr_2r_3}. Analogously, the identity~\eqref{p identity v} holds in $\mathbf Y$ but fails in $\mathbf X$.

If either $x_p,y_0\in\con(\mathbf u^\prime)$ or $x_p,y_0\in\con(\mathbf u^{\prime\prime})$, then, since $\mathbf X$ satisfies $\sigma_1$ and $\sigma_2$, the pair $(y_0,x_p)$ is \mbox{$(\mathbf X,\mathbf u,\mathbf v,i)$}-invertible. This contradicts the claim~(a). Therefore, we may assume without loss of generality that $y_0\in\con(\mathbf u^\prime)\setminus\con(\mathbf u^{\prime\prime})$ and $x_p\in\con(\mathbf u^{\prime\prime})\setminus\con(\mathbf u^\prime)$. 

Analogously, either $y_q\in\con(\mathbf v^\prime)\setminus\con(\mathbf v^{\prime\prime})$ and $x_0\in\con(\mathbf v^{\prime\prime})\setminus\con(\mathbf v^\prime)$ or $x_0\in\con(\mathbf v^\prime)\setminus\con(\mathbf v^{\prime\prime})$ and $y_q\in\con(\mathbf v^{\prime\prime})\setminus\con(\mathbf v^\prime)$. Further considerations are divided into two cases.

\smallskip

\emph{Case 1}: $y_q\in\con(\mathbf v^\prime)\setminus\con(\mathbf v^{\prime\prime})$ and $x_0\in\con(\mathbf v^{\prime\prime})\setminus\con(\mathbf v^\prime)$. If the identity $\gamma_{f_0,e_p}$ holds in $\mathbf X$, then Lemma~\ref{Lem: gamma_{e,f} implies pxyq=pyxq} implies that the pair $(x_p,y_0)$ is $\mbox(\mathbf X,\mathbf u,\mathbf v,i)$-invertible, contradicting with the claim~(a). Therefore, $\gamma_{f_0,e_p}$ fails in $\mathbf X$. But, since the identity~\eqref{q identity u} holds in the variety $\mathbf X$, this variety satisfies the identity $\gamma_{f_q,e_p}$ by Lemma~\ref{Lem: smth imply gamma_{e,f}}. This fact and Lemma~\ref{Lem: R{gamma_{k,ell}} subset R{gamma_{p,q}}} imply that $f_0<f_q$. Further, since $\mathbf Y$ satisfies the identity~\eqref{u=u'qy_0x_0u_i''u''}, $\mathbf Y$ satisfies also the identity $\gamma_{f_0,e_0}$ by Lemma~\ref{Lem: smth imply gamma_{e,f}}. Now Lemma~\ref{Lem: R{gamma_{k,ell}} subset R{gamma_{p,q}}} applies with the conclusion that the identity $\gamma_{f_q,e_0}$ holds in $\mathbf Y$. According to Lemma~\ref{Lem: gamma_{e,f} implies pxyq=pyxq}, the pair $(y_q,x_0)$ is \mbox{$(\mathbf Y,\mathbf v,\mathbf u,i)$}-invertible. But this is not the case by the claim~(b). So, Case~1 is impossible.

\smallskip

\emph{Case 2}: $x_0\in\con(\mathbf v^\prime)\setminus\con(\mathbf v^{\prime\prime})$ and $y_q\in\con(\mathbf v^{\prime\prime})\setminus\con(\mathbf v^\prime)$. Further considerations are divided into three subcases.

\smallskip

\emph{Subcase 2.1}: there are $k\in\{0,1,\dots,p\}$ and $\ell\in\{0,1,\dots,q\}$ such that $x_k,y_\ell\in\con(\mathbf v^\prime)\setminus\con(\mathbf v^{\prime\prime})$ and the identities $\gamma_{f_\ell,f_q}$ and $\gamma_{e_k,e_p}$ fail in $\mathbf X$ and $\mathbf Y$, respectively. Since the identities~\eqref{q identity u} and~\eqref{p identity v} hold in the varieties $\mathbf X$ and $\mathbf Y$, respectively, Lemma~\ref{Lem: smth imply gamma_{e,f}} implies that $\mathbf X$ satisfies $\gamma_{e_k,f_q}$ and $\mathbf Y$ satisfies $\gamma_{f_\ell,e_p}$. Then $f_\ell<e_k<f_\ell$ by Lemma~\ref{Lem: R{gamma_{k,ell}} subset R{gamma_{p,q}}}, a contradiction. So, this subcase is impossible.

\smallskip

\emph{Subcase 2.2}: $\mathbf X$ satisfies the identity $\gamma_{f_j,f_q}$ for any $j$ such that $y_j\in\con(\mathbf v^\prime)\setminus\con(\mathbf v^{\prime\prime})$. Clearly, $\mathbf u_i^{\prime\prime}=\mathbf a_iy_q\mathbf b_i$ for some words $\mathbf a_i$ and $\mathbf b_i$ and $y_1,y_2,\dots,y_{q-1}\in\con(\mathbf a_i)$. Let $\mathbf w=x_0\mathbf u_i^\prime y_0\mathbf a_i$. Then $\mathbf u=\mathbf u^\prime\mathbf q\mathbf wy_q\mathbf b_i\mathbf u^{\prime\prime}$. For convenience, we rename letters from $\con(\mathbf w)$ and put $\mathbf w=z_1z_2\cdots z_a$. We are going to check that the variety $\mathbf X$ satisfies the identities
\begin{equation}
\label{move y_q to q}
\left\{\!\!
\begin{array}{ll}
\mathbf u\!\!\!\!&=\mathbf u^\prime\mathbf qz_1z_2\cdots z_ay_q\mathbf b_i\mathbf u^{\prime\prime}\approx\mathbf u^\prime\mathbf qz_1z_2\cdots z_{a-1}y_qz_a\mathbf b_i\mathbf u^{\prime\prime}\\
&\approx\mathbf u^\prime\mathbf qz_1z_2\cdots z_{a-2}y_qz_{a-1}z_a\mathbf b_i\mathbf u^{\prime\prime}\approx\cdots\approx\mathbf u^\prime\mathbf qy_qz_1z_2\cdots z_a\mathbf b_i\mathbf u^{\prime\prime}\\
&=\mathbf u^\prime\mathbf qy_qx_0\mathbf u_i^\prime y_0\mathbf a_i\mathbf b_i\mathbf u^{\prime\prime}=\mathbf u^\prime\mathbf qy_qx_0\mathbf u_i^\prime y_0(\mathbf u_i^{\prime\prime})_{y_q}\mathbf u^{\prime\prime}.
\end{array}
\right.
\end{equation}
Let $r\in\{1,2,\dots,a\}$. If $z_r\in\con(\mathbf v^{\prime\prime})$, then we can swap $y_q$ with $z_r$ using the identity $\sigma_2$. Let now $z_r\notin\con(\mathbf v^{\prime\prime})$ and therefore, $z_r\in\con(\mathbf v^\prime)$. If $z_r=y_j$ for some $j\in\{1,2,\dots,q-1\}$, then we can swap $y_q$ with $z_r$ using the identity $\gamma_{f_j,f_q}$ that holds in $\mathbf X$ by the hypothesis of Subcase~2.2. Finally, suppose that $z_r\notin\{y_1,y_2,\dots,y_{q-1}\}$. Since $\mathbf X$ satisfies the identity~\eqref{q identity u}, we can apply Lemma~\ref{Lem: smth imply gamma_{e,f}} and conclude that $\mathbf X$ satisfies the identity $\gamma_{h,f_q}$, where $h=\occ_{z_r}(\mathbf u)-1$. Then we can swap $y_q$ with $z_r$ by Lemma~\ref{Lem: gamma_{e,f} implies pxyq=pyxq}. Thus, we really can step by step swap $y_q$ with $z_a$, $z_{a-1}$, \dots, $z_1$ using identities that hold in $\mathbf X$ on each step. This implies that $\mathbf X$ satisfies the identities~\eqref{move y_q to q} and in particular, the identity $\mathbf u\approx\mathbf u^\sharp$, where 
$$
\mathbf u^\sharp=\mathbf u^\prime\mathbf qy_qx_0\mathbf u_i^\prime y_0(\mathbf u_i^{\prime\prime})_{y_q}\mathbf u^{\prime\prime}.
$$

Further, $\mathbf X$ violates $\gamma_{f_0,e_p}$ because the pair $(y_0,x_p)$ is \mbox{$(\mathbf X,\mathbf u,\mathbf v,i)$}-invertible by Lemma~\ref{Lem: gamma_{e,f} implies pxyq=pyxq} otherwise. Recall that $y_0\in\con(\mathbf u^\prime)\setminus\con(\mathbf u^{\prime\prime})$. Since the identity $\mathbf u\approx\mathbf v$ is linear-balanced, $y_0\in\con(\mathbf v^\prime)\setminus\con(\mathbf v^{\prime\prime})$. Then $\mathbf X$ satisfies $\gamma_{f_0,f_q}$ by the hypothesis, whence Lemma~\ref{Lem: R{gamma_{k,ell}} subset R{gamma_{p,q}}} implies that $e_p<f_q$. Since the identity~\eqref{p identity v} holds in the variety $\mathbf Y$, this variety satisfies $\gamma_{f_j,e_p}$ for any $j$ such that $y_j\in\con(\mathbf v^\prime)\setminus\con(\mathbf v^{\prime\prime})$ by Lemma~\ref{Lem: smth imply gamma_{e,f}}. Taking into account Lemma~\ref{Lem: R{gamma_{k,ell}} subset R{gamma_{p,q}}}, we get that $\gamma_{f_j,f_q}$ holds in $\mathbf Y$ for any $j$ such that $y_j\in\con(\mathbf v^\prime)\setminus\con(\mathbf v^{\prime\prime})$. Then arguments similar to ones from the previous paragraph imply that $\mathbf Y$ satisfies the identity $\mathbf v\approx\mathbf v^\sharp$, where 
$$
\mathbf v^\sharp=\mathbf v^\prime\mathbf qy_qy_0(\mathbf v_i^\prime)_{y_q}x_0\mathbf v_i^{\prime\prime}\mathbf v^{\prime\prime}.
$$

To obtain $\mathbf v$ from $\mathbf u$ by swapping of adjacent occurrences of multiple letters, we need to transform the word $x_0\mathbf u_i^\prime y_0\mathbf u_i^{\prime\prime}$ to the word $y_0\mathbf v_i^\prime x_0\mathbf v_i^{\prime\prime}$. By the hypothesis, this can be done in $n$ steps. Further, to obtain $\mathbf v^\sharp$ from $\mathbf u^\sharp$ by the same way, we need to transform the word $x_0\mathbf u_i^\prime y_0(\mathbf u_i^{\prime\prime})_{y_q}$ to the word $y_0(\mathbf v_i^\prime)_{y_q}x_0\mathbf v_i^{\prime\prime}$. We see that in the second case we need to replace one letter less than in the first one and the mutual location of all letters in the word $x_0\mathbf u_i^\prime y_0(\mathbf u_i^{\prime\prime})_{y_q}$ [respectively, $y_0(\mathbf v_i^\prime)_{y_q}x_0\mathbf v_i^{\prime\prime}$] coincides with the mutual location of all letters except $y_q$ in the word $x_0\mathbf u_i^\prime y_0\mathbf u_i^{\prime\prime}$ [respectively, $y_0\mathbf v_i^\prime x_0\mathbf v_i^{\prime\prime}$]. Therefore, the identity $\mathbf u^\sharp\approx\mathbf v^\sharp$ is $t$-invertible for some $t<n$. It is clear that this identity holds in $\mathbf X\wedge\mathbf Y$. Thus, the words $\mathbf u^\sharp$ and $\mathbf v^\sharp$ have the properties indicated in the paragraph after the claim~(b). As we have seen there, this implies the desired conclusion.

\smallskip

\emph{Subcase 2.3}: $\mathbf Y$ satisfies the identity $\gamma_{e_j,e_p}$ for any $j$ such that $x_j\in\con(\mathbf v^\prime)\setminus\con(\mathbf v^{\prime\prime})$. This subcase is similar to Subcase~2.2.

\smallskip

Proposition~\ref{Prop: R is fi-perm} is proved. 
\end{proof}

\section{Proof of main results}
\label{Sec: proof of main results}

\begin{proof}[Proof of Theorem~\ref{Th: fi-perm}]
Throughout the proof, we will use Lemma~\ref{Lem: lifting} many times without explicitly specifying this. 

\smallskip

\emph{Necessity}. Let $\mathbf V$ be a non-group $fi$-permutable variety of monoids. Then $\mathbf{SL}\subseteq\mathbf V$ by Lemma~\ref{Lem: group variety}. Now Lemma~\ref{Lem: non-fi-perm}(i) applies and we conclude that $\mathbf A_n\nsubseteq\mathbf V$ for any $n>1$. In other words, the variety $\mathbf V$ is aperiodic, whence it satisfies the identity~\eqref{x^n=x^{n+1}} for some $n\in\mathbb N$. Let $n$ be the least number with such a property. 

If $n=1$, then the variety $\mathbf V$ is completely regular. But every completely regular aperiodic variety consists of idempotent monoids, and we are done. 

Suppose now that $n>2$. Then Lemma~\ref{Lem: x^n is isoterm} implies that $\mathbf C_3\subseteq\mathbf V$. It is verified by Gusev~\cite[Lemma~5]{Gusev-20+} that the lattice $L(\mathbf C_3\vee\mathbf E)$ is non-modular. Clearly, the lattice $L(\mathbf C_3\vee\mathbf E^\delta)$ is non-modular too. Now Lemma~\ref{Lem: fi-perm implies modularity} applies and we conclude that $\mathbf E,\mathbf E^\delta\nsubseteq\mathbf V$. Then Lemma~\ref{Lem: does not contain E} and the statement dual to it imply that $\mathbf V$ satisfies the identities~\eqref{x^nyx^n=yx^n} and $x^nyx^n\approx x^ny$. Then the identity $x^ny\approx yx^n$ holds in $\mathbf V$. Further, $\mathbf D_2\nsubseteq\mathbf V$ by Lemma~\ref{Lem: non-fi-perm}(iii). Then Lemmas~\ref{Lem: x^n is isoterm} and~\ref{Lem: does not contain D_{k+1}} imply that one of the identities $\delta_1$ or $yx^2\approx xyx$ is true in $\mathbf V$. This means that either $\mathbf V\subseteq\mathbf P_n$ or $\mathbf V\subseteq\mathbf P_n^\delta$, and we are done.

Finally, suppose that $n=2$. Then Lemma~\ref{Lem: x^n is isoterm} implies that $\mathbf C_2\subseteq\mathbf V$. Suppose that $\mathbf E\subseteq\mathbf V$. Then Lemma~\ref{Lem: non-fi-perm}(v) implies that $\mathbf E^\delta\nsubseteq\mathbf V$. Now the statement dual to Lemma~\ref{Lem: does not contain E} shows that $\mathbf V$ satisfies the identity
\begin{equation}
\label{xxyxx=xxy}
x^2yx^2\approx x^2y.
\end{equation}

Put $\mathbf{LRB}=\var\{xy\approx xyx\}$. By Lee~\cite[Proposition~4.1(i) and Lemma~3.3(iv)]{Lee-12}, the lattice $L(\mathbf C_2\vee\mathbf{LRB})$ is non-modular. Then Lemma~\ref{Lem: fi-perm implies modularity} implies that $\mathbf{LRB}\nsubseteq\mathbf V$. Hence there is an identity $\mathbf u\approx\mathbf v$ that holds in $\mathbf V$ but fails in $\mathbf{LRB}$. For any word $\mathbf w$, we denote by $\ini(\mathbf w)$ the word obtained from $\mathbf w$ by retaining only the first occurrence of each letter. It is evident that an identity $\mathbf a\approx\mathbf b$ holds in the variety $\mathbf{LRB}$ if and only if $\ini(\mathbf a)=\ini(\mathbf b)$. Thus, $\ini(\mathbf u)\ne\ini(\mathbf v)$. Lemma~\ref{Lem: group variety} implies that $\con(\mathbf u)=\con(\mathbf v)$. Therefore, we may assume that there are letters $x,y\in\con(\mathbf u)$ such that $\mathbf u(x,y)=x^sy\mathbf w_1$ and $\mathbf v(x,y)=y^tx\mathbf w_2$, where $s,t>0$ and $\con(\mathbf w_1)=\con(\mathbf w_2)=\{x,y\}$. Let us substitute $x^2$ and $y^2$ for $x$ and $y$, respectively, in the identity $\mathbf u(x,y)\approx\mathbf v(x,y)$. After that we apply the identities~\eqref{xx=xxx} and~\eqref{xxyxx=xxy}, resulting in the identity $x^2y^2\approx y^2x^2$. 

Further, $\mathbf D_2\nsubseteq\mathbf V$ by Lemma~\ref{Lem: non-fi-perm}(iv). Then Lemma~\ref{Lem: does not contain D_{k+1}} implies that $\mathbf V$ satisfies the identity
\begin{equation}
\label{xyx=x^qyx^r}
xyx\approx x^qyx^r,
\end{equation}
where either $q>1$ or $r>1$. If $q>1$, then $\mathbf V$ satisfies the identities
$$
xyx\stackrel{\eqref{xyx=x^qyx^r}}\approx x^qyx^r\stackrel{\eqref{xxyxx=xxy}}\approx x^qyx^{r+2}\stackrel{\eqref{xx=xxx}}\approx x^2yx^2\stackrel{\eqref{xxyxx=xxy}}\approx x^2y,
$$
whence $\mathbf V\subseteq\mathbf E\subset\mathbf K$. Let now $q\le 1$. Then $r>1$. If $q=0$, then the claim that $\mathbf V$ satisfies the identity~\eqref{xx=xxx} implies that the identity $xyx\approx yx^2$ holds in $\mathbf V$. Besides that, $\mathbf V$ satisfies the identity~\eqref{xxyxx=xxy}. Now Lemma~\ref{Lem: basis for D_k} implies that $\mathbf V\subseteq\mathbf D_1\subset\mathbf K$. Finally, let $q=1$. Since the variety $\mathbf V$ satisfies the identity~\eqref{xx=xxx}, it satisfies also the identity
\begin{equation}
\label{xyx=xyxx}
xyx\approx xyx^2.
\end{equation}
Besides that, the identities $x^2yx\stackrel{\eqref{xyx=xyxx}}\approx x^2yx^2\stackrel{\eqref{xxyxx=xxy}}\approx x^2y$ hold in $\mathbf V$. Therefore, $\mathbf V\subseteq\mathbf K$. 

Thus, if $\mathbf E\subseteq\mathbf V$, then $\mathbf V\subseteq\mathbf K$. By symmetry, if $\mathbf E^\delta\subseteq\mathbf V$, then $\mathbf V\subseteq\mathbf K^\delta$. We are done in both the cases. 

Below we assume that $\mathbf E,\mathbf E^\delta\nsubseteq\mathbf V$. Then Lemma~\ref{Lem: does not contain E} and the dual statement imply that $\mathbf V\subseteq\mathbf A$. In view of Lemmas~\ref{Lem: comparable is fi-permut} and~\ref{Lem: L(A)}(i), we may assume that $\mathbf D_2\subseteq\mathbf V$. If $\mathbf V$ does not contain the varieties $\mathbf L$, $\mathbf M$ and $\mathbf M^\delta$, then Lemmas~\ref{Lem: basis for D_k} and~\ref{Lem: V contains D_2 but does not contain L or M or M^delta} imply that $\mathbf V\subseteq\mathbf D_\infty\subset\mathbf D_\infty\vee\mathbf N$, and we are done. In view of symmetry, we may assume that either $\mathbf M\subseteq\mathbf V$ or $\mathbf L\subseteq\mathbf V$.

Suppose at first that $\mathbf M\subseteq\mathbf V$. Then it follows from Lemma~\ref{Lem: non-fi-perm}(vi) that $\mathbf L\nsubseteq\mathbf V$. Then Lemma~\ref{Lem: V contains D_2 but does not contain L or M or M^delta}(iii) implies that $\mathbf V$ satisfies the identity $\sigma_3$. 

Suppose that $\mathbf N\subseteq\mathbf V$. It is verified by Gusev~\cite[Theorem~1.1 and Fig.~1]{Gusev-19} that the lattice $L(\mathbf N\vee\mathbf M^\delta)$ is non-modular. Now Lemma~\ref{Lem: fi-perm implies modularity} applies and we have that $\mathbf M^\delta\nsubseteq\mathbf V$. Further, $\mathbf Z_i\nsubseteq\mathbf V$ for each $i=1,2,3$ by Lemma~\ref{Lem: non-fi-perm}(vii). In view of the above and Corollary~\ref{Cor: pxyq=pyxq in X wedge Y}, we have that $\mathbf V$ satisfies the identities $\sigma_2$ and $\alpha_i$ with $i=1,2,3$. Then Corollary~\ref{Cor: D_k vee N and D_k vee M} implies that $\mathbf V\subseteq\mathbf D_\infty\vee\mathbf N$, and we are done. 

Whence, we may assume that $\mathbf N\nsubseteq\mathbf V$. Taking into account Lemma~\ref{Lem: var S(c_{n+1,m}[rho]) and S(c_{n,m+1}[tau])}, we have that $S(\mathbf c_{n,m}[\pi])\notin\mathbf V$ for all $n,m\in\mathbb N_0$ and $\pi\in S_{n+m}$. Then Lemma~\ref{Lem: V notin S(c_{n,m}[rho])} applies with the conclusion that $\mathbf V$ satisfies the identity~\eqref{c_{n,m}[rho]=c_{n,m}'[rho]} for all $n,m\in\mathbb N_0$ and $\rho\in S_{n+m}$. The lattice $L(\mathbf M\vee\mathbf N^\delta)$ is non-modular because the lattice $L(\mathbf N\vee\mathbf M^\delta)$ is so. Now Lemma~\ref{Lem: fi-perm implies modularity} applies and we have that $\mathbf N^\delta\nsubseteq\mathbf V$. Then the dual to Lemma~\ref{Lem: V notin S(c_{n,m}[rho])} implies that $\mathbf V$ satisfies the identity $\mathbf d_{n,m}[\rho]\approx\mathbf d_{n,m}^\prime[\rho]$ for all $n,m\in\mathbb N_0$ and $\rho\in S_{n+m}$.

Finally, Lemma~\ref{Lem: non-fi-perm}(viii) implies that there are $r,s\in\{1,2,3\}$ such that $\mathbf Z_i,\mathbf Z_j\nsubseteq\mathbf V$ for all $i\in\{1,2,3\}\setminus\{r\}$ and $j\in\{1,2,3\}\setminus\{s\}$. Then $\mathbf V$ satisfies the identities $\alpha_i$ and $\beta_j$ with $i\in\{1,2,3\}\setminus\{r\}$ and $j\in\{1,2,3\}\setminus\{s\}$ by Lemma~\ref{Lem: pxyq=pyxq}. Summarizing all we say above, we have that $\mathbf V\subseteq\mathbf Q_{r,s}$, and we are done.

Finally, suppose that $\mathbf L\subseteq\mathbf V$. Then Lemma~\ref{Lem: non-fi-perm}(vi) and the dual to it imply that $\mathbf M,\mathbf M^\delta\nsubseteq\mathbf V$. Then Lemma~\ref{Lem: V contains D_2 but does not contain L or M or M^delta}(i),(ii) implies that $\mathbf V$ satisfies the identities $\sigma_1$ and $\sigma_2$. It is proved in Gusev and Vernikov~\cite[p.~32]{Gusev-Vernikov-18} that every variety of the form $\var S(\mathbf w_n[\pi,\tau])$ contains two non-comparable subvarieties of the same form. This claim and Lemma~\ref{Lem: var S(w_n[pi,tau]) and var S(w_n[xi,eta])} imply that $S(\mathbf w_n[\pi,\tau])\notin\mathbf V$ for all $n\in\mathbb N$ and $\pi,\tau\in S_n$. 
 
It is checked in the paragraph starting on p.~32 and ending on p.~33 in~\cite{Gusev-Vernikov-18} that if $\mathbf L\subseteq\mathbf X\subseteq\mathbf A$, $\mathbf X$ satisfies the identities $\sigma_1$, $\sigma_2$ and $\delta_2$ and $S(\mathbf w_n[\pi,\tau])\notin\mathbf X$ for all $n\in\mathbb N$ and $\pi,\tau\in S_n$, then $\mathbf X$ satisfies the identity~\eqref{w_n[pi,tau]=w_n'[pi,tau]} for all $n\in\mathbb N$ and $\pi,\tau\in S_n$. Repeating arguments from that paragraph in~\cite{Gusev-Vernikov-18} but referring to Lemma~\ref{Lem: S(w_n[pi,tau]) notin V} rather than to~\cite[Lemma~4.10]{Gusev-Vernikov-18}, we can verify that the same conclusion is true without the hypothesis that $\mathbf X$ satisfies the identity $\delta_2$. Together with the saying in the previous paragraph, we see that the identity~\eqref{w_n[pi,tau]=w_n'[pi,tau]} holds in $\mathbf V$ for any $n\in\mathbb N$ and $\pi,\tau\in S_n$. Therefore, $\mathbf V\subseteq\mathbf R$, and we are done.

\smallskip

\emph{Sufficiency}. An arbitrary group variety is $fi$-permutable because it is congruence permutable. Lemmas~\ref{Lem: group variety} and~\ref{Lem: cr is fi-perm} imply that an arbitrary variety of idempotent monoids is $fi$-permutable too. By symmetry, it remains to consider the varieties $\mathbf K$, $\mathbf D_\infty\vee\mathbf N$, $\mathbf P_n$, $\mathbf Q_{r,s}$ and $\mathbf R$. The variety $\mathbf K$ is $fi$-permutable by Lemmas~\ref{Lem: comparable is fi-permut} and~\ref{Lem: L(K) is a chain}. The same conclusion for the varieties $\mathbf D_\infty\vee\mathbf N$, $\mathbf P_n$, $\mathbf Q_{r,s}$ and $\mathbf R$ follows from Propositions~\ref{Prop: D_infty vee N is fi-perm},~\ref{Prop: P_n is fi-perm},~\ref{Prop: Q_{r,s} is fi-perm} and~\ref{Prop: R is fi-perm}, respectively. 

\smallskip

Theorem~\ref{Th: fi-perm} is proved.
\end{proof}

\begin{proof}[Proof of Theorem~\ref{Th: almost fi-perm}]
\emph{Necessity}. Let $\mathbf V$ be a non-completely regular almost $fi$-permutable variety of monoids. Then $\mathbf C_2\subseteq\mathbf V$ by Corollary~\ref{Cor: cr}. Now the inclusion $\mathbf{SL}\subset\mathbf C_2$ and Lemma~\ref{Lem: non-fi-perm}(ii) imply that $\mathbf A_n\nsubseteq\mathbf V$ for any $n>1$. Therefore, $\mathbf V$ is an aperiodic variety. Lemma~\ref{Lem: group variety} shows that for aperiodic varieties the properties to be $fi$-permutable and almost $fi$-permutable are equivalent. Now Theorem~\ref{Th: fi-perm} implies that $\mathbf V$ is contained in one of the varieties listed in the item~(iii) of this theorem.

\smallskip

\emph{Sufficiency} immediately follows from Lemma~\ref{Lem: cr is fi-perm} and Theorem~\ref{Th: fi-perm}. 

\smallskip

Theorem~\ref{Th: almost fi-perm} is proved.
\end{proof}

\begin{proof}[Proof of Corollaries~\ref{Cor: fi-perm distributive} and~\ref{Cor: almost fi-perm distributive}]
The lattice of all varieties of idempotent monoids is completely described by Wismath~\cite{Wismath-86}. In particular, it turns out to be distributive. Note that, in actual fact, this claim follows from Lemma~\ref{Lem: embedding} and the assertion that the lattice of all varieties of idempotent semigroups is distributive; this fact was independently discovered by Biryukov, Fennemore and Gerhard in early 1970s (see Evans~\cite[Section~XI]{Evans-71}, for instance).

Theorems~\ref{Th: fi-perm} and~\ref{Th: almost fi-perm} show that, up to duality, it remains to check that the varieties $\mathbf K$, $\mathbf D_\infty\vee\mathbf N$, $\mathbf P_n$, $\mathbf Q_{r,s}$ and $\mathbf R$ have distributive subvariety lattices. The lattices $L(\mathbf K)$, $L(\mathbf D_\infty\vee\mathbf N)$ and $L(\mathbf P_n)$ with any $n\in\mathbb N$ are distributive by Lemma~\ref{Lem: L(K) is a chain} and Corollaries~\ref{Cor: L(D_infty vee N)} and~\ref{Cor: L(P_n) is distributive}, respectively. In view of Proposition~\ref{Prop: L(A') is distributive}, to complete the proof, it suffices to verify that $\mathbf Q_{r,s}\subseteq\mathbf A^\prime$ and $\mathbf R\subseteq\mathbf A^\prime$. The second inclusion is evident because the identities~\eqref{c_{n,m,k}[rho]=c_{n,m,k}'[rho]} and $\mathbf d_{n,m,k}[\pi]\approx\mathbf d_{n,m,k}^\prime[\pi]$ follow from the identities $\sigma_1$ and $\sigma_2$, respectively. Lemma~\ref{Lem: smth imply w_n[pi,tau]=w_n'[pi,tau]} implies that the variety $\mathbf Q_{r,s}$ satisfies the identity~\eqref{w_n[pi,tau]=w_n'[pi,tau]} for any $n\in\mathbb N_0$ and $\pi,\tau\in S_n$. By symmetry, it remains to verify that $\mathbf Q_{r,s}$ satisfies also the identity~\eqref{c_{n,m,k}[rho]=c_{n,m,k}'[rho]} for any $n,m,k\in\mathbb N_0$ and $\rho\in S_{n+m+k}$. If $k=0$, then this claim is evident. Finally, let $k\ge 1$. Put
$$
X = \{t_i,z_i\mid n+m+1\le i\le n+m+k\}.
$$ 
Clearly, the identity $(\mathbf c_{n,m,k}[\rho])_X\approx(\mathbf c_{n,m,k}^\prime[\rho])_X$ coincides with
\begin{equation}
\label{c_{n,m}[rho']=c_{n,m}'[rho']}
\mathbf c_{n,m}[\rho^\prime]\approx \mathbf c_{n,m}^\prime[\rho^\prime]
\end{equation}
for some $\rho^\prime\in S_{n+m}$. Then $\mathbf Q_{r,s}$ satisfies the identities
$$
\mathbf c_{n,m,k}[\rho]\stackrel{\sigma_3}\approx\mathbf pxy\mathbf qx\mathbf ry\mathbf s\stackrel{\eqref{c_{n,m}[rho']=c_{n,m}'[rho']}}\approx\mathbf pyx\mathbf qx\mathbf ry\mathbf s\stackrel{\sigma_3}\approx\mathbf c_{n,m,k}^\prime[\rho],
$$
where
\begin{align*}
&\mathbf p=\prod_{i=1}^n (z_it_i),\ \mathbf q=t\biggl(\prod_{i=n+1}^{n+m} z_it_i\biggr),\ \mathbf r=\prod_{i\rho\le n+m}z_{i\rho}\quad\text{and}\\[-3pt]
&\mathbf s=\biggl(\prod_{i\rho>n+m} z_{i\rho}\biggr)\biggl(\prod_{i=n+m+1}^{n+m+k} t_iz_i\biggr).
\end{align*}

Corollaries~\ref{Cor: fi-perm distributive} and~\ref{Cor: almost fi-perm distributive} are proved.
\end{proof} 

\section{Generalizations of $fi$-permutability}
\label{Sec: generalizations}

Let $\alpha$ and $\beta$ be congruences on an algebra $A$. For any $n$, we put
$$
\alpha\circ_n\beta=\underbrace{\alpha\beta\alpha\beta\cdots}_{n\ \text{letters}}.
$$
Congruences $\alpha$ and $\beta$ are said to $n$-\emph{permute} if $\alpha\circ_n\beta=\beta\circ_n\alpha$. Clearly, 2-permutative congruences are nothing but simply permutative ones. 3-permutative congruences, that is, congruences $\alpha$ and $\beta$ such that $\alpha\beta\alpha=\beta\alpha\beta$ are usually called \emph{weakly permutative}. It is evident that if congruences $\alpha$ and $\beta$ on some algebra $n$-permute then $\alpha\vee\beta=\alpha\circ_n\beta$. A variety of algebras $\mathbf V$ is called \emph{congruence $n$-permutable} [\emph{weakly congruence permutable}] if, on any member of $\mathbf V$, every two congruences $n$-permute [respectively, weakly permute]. 

As well as congruence permutability, the property to be congruence $n$-permutable for some $n$ is very rigid for semigroup or monoid varieties. According to Lipparini~\cite[Corollary~0]{Lipparini-95}, a semigroup variety is congruence $n$-permutable for some $n$ if and only if it is a periodic group variety (this fact with $n=3$ follows also from Jones~\cite[Theorem~1.2(iii)]{Jones-88}). The same is true for monoid varieties. This fact can be easily deduced from Lemma~\ref{Lem: group variety} and results by Freese and Nation~\cite{Freese-Nation-73} and Lipparini~\cite{Lipparini-95}, as well as from Lemma~\ref{Lem: group variety} and~\cite[Lemma~9.13]{Hobby-McKenzie-88}.

For arbitrary $n$, we call a variety of algebras $\mathbf V$ $fi$-$n$-\emph{permutable} if any two fully invariant congruences on every $\mathbf V$-free object $n$-permute. Clearly, $fi$-2-permutable varieties is simply $fi$-permutable ones; $fi$-3-permutable varieties are called \emph{weakly $fi$-permutable}. It is proved by Vernikov~\cite{Vernikov-04c} that every weakly $fi$-permutable semigroup variety is either completely regular or a nil-variety. We provide here an analogous, in a sense, fact concerning monoid varieties.

\begin{lemma}
\label{Lem: weakly fi-permut}
If $\mathbf V$ is a weakly $fi$-permutable variety of monoids, then $\mathbf V$ is either completely regular or aperiodic.
\end{lemma}

\begin{proof}
Suppose that $\mathbf V$ is neither completely regular nor aperiodic. Then $\mathbf C_2\subseteq\mathbf V$ by Corollary~\ref{Cor: cr} and $\mathbf A_n\subseteq\mathbf V$ for some $n>1$. We denote by $\nabla$ the universal relation on the free monoid $\mathfrak X^\ast$. Then $\theta_{\mathbf A_n}\vee\theta_{\mathbf C_2}=\nabla$ because $\mathbf A_n\wedge\mathbf C_2=\mathbf T$. By the analog of Lemma~\ref{Lem: lifting} for weakly permutative equivalences, the congruences $\theta_{\mathbf A_n}$ and $\theta_{\mathbf C_2}$ weakly permute, whence $\nabla=\theta_{\mathbf A_n}\vee\theta_{\mathbf C_2}=\theta_{\mathbf C_2}\theta_{\mathbf A_n}\theta_{\mathbf C_2}$. Therefore, $(x,y)\in\theta_{\mathbf C_2}\theta_{\mathbf A_n}\theta_{\mathbf C_2}$ for any letters $x$ and $y$. Thus, $x\,\theta_{\mathbf C_2}\,\mathbf w_1\,\theta_{\mathbf A_n}\,\mathbf w_2\,\theta_{\mathbf C_2}\,y$ for some words $\mathbf w_1$ and $\mathbf w_2$. Now Lemma~\ref{Lem: x^n is isoterm} applies with the conclusion that $\mathbf w_1=x$ and $\mathbf w_2=y$. Therefore, $x\,\theta_{\mathbf A_n}\,y$ that is evidently not the case.
\end{proof}

Weakly $fi$-permutable completely regular semigroup varieties are completely determined in Vernikov~\cite{Vernikov-04c}. The following problem naturally arises.

\begin{problem}
\label{Prob: weakly fi-permut}
Describe 
\begin{itemize}
\item[(i)] weakly $fi$-permutable completely regular monoid varieties;
\item[(ii)] weakly $fi$-permutable aperiodic monoid varieties.
\end{itemize}
\end{problem}

We do not know even whether there exists a completely regular monoid variety that is not weakly $fi$-permutable.

It follows from Lipparini~\cite[Theorem~1]{Lipparini-95} that if a variety of algebras $\mathbf V$ is $fi$-$n$-permutable for some $n$, then the lattice $L(\mathbf V)$ satisfies some non-trivial lattice identity. There are no any additional information about $fi$-$n$-permutable semigroup or monoid varieties with $n>3$ so far. The following observation shows that the analog of Lemma~\ref{Lem: weakly fi-permut} for $fi$-$n$-permutable monoid varieties with $n>3$ is false.

\begin{remark}
\label{fi-4-permut variety}
If $p$ is a prime number, then the variety $\mathbf A_p\vee\mathbf C_2$ is $fi$-$4$-permutable.
\end{remark} 

\begin{proof}
Let $\mathbf X,\mathbf Y\subseteq\mathbf A_p\vee\mathbf C_2$ and $(\mathbf u,\mathbf v)\in\theta_{\mathbf X}\vee\theta_{\mathbf Y}$. By the analog of Lemma~\ref{Lem: lifting} for 4-permutative equivalences, it suffices to verify that $(\mathbf u,\mathbf v)\in\theta_{\mathbf X}\theta_{\mathbf Y}\theta_{\mathbf X}\theta_{\mathbf Y}$. Results by Head~\cite{Head-68} imply that the lattice $L(\mathbf A_p\vee\mathbf C_2)$ has the form shown in Fig.~\ref{Fig: L(A_p vee C_2)}. In view of Lemma~\ref{Lem: comparable is fi-permut}, we may assume without loss of generality that either one of the varieties $\mathbf X$ or $\mathbf Y$ coincides with $\mathbf A_p$, while another one lies in the set $\{\mathbf{SL},\mathbf C_2\}$ or one of these two varieties coincides with $\mathbf A_p\vee\mathbf{SL}$, while another one equals $\mathbf C_2$. Suppose that $\mathbf X=\mathbf A_p$ and $\mathbf Y\in\{\mathbf{SL},\mathbf C_2\}$. Then
\begin{equation}
\label{from u to v}
\mathbf u\,\theta_{\mathbf X}\,\mathbf u^{p+1}\mathbf v^p\,\theta_{\mathbf Y}\,\mathbf u^p\mathbf v^{p+1}\,\theta_{\mathbf X}\,\mathbf v,
\end{equation}
whence $(\mathbf u,\mathbf v)\in\theta_{\mathbf X}\theta_{\mathbf Y}\theta_{\mathbf X}$. Further, if $\mathbf X=\mathbf A_p\vee\mathbf{SL}$ and $\mathbf Y=\mathbf C_2$ then the identity $\mathbf u\approx\mathbf v$ holds in $\mathbf X\wedge\mathbf Y=\mathbf{SL}$. Then Lemma~\ref{Lem: group variety} implies that $\con(\mathbf u)=\con(\mathbf v)$. Therefore,~\eqref{from u to v} holds, whence $(\mathbf u,\mathbf v)\in\theta_{\mathbf X}\theta_{\mathbf Y}\theta_{\mathbf X}$ again. Finally, if either  $\mathbf Y=\mathbf A_p$ and $\mathbf X\in\{\mathbf{SL},\mathbf C_2\}$ or $\mathbf Y=\mathbf A_p\vee\mathbf{SL}$ and $\mathbf X=\mathbf C_2$, then the same arguments as above show that $(\mathbf u,\mathbf v)\in\theta_{\mathbf Y}\theta_{\mathbf X}\theta_{\mathbf Y}$. Thus,  
$$
(\mathbf u,\mathbf v)\in\theta_{\mathbf X}\theta_{\mathbf Y}\theta_{\mathbf X}\cup\theta_{\mathbf Y}\theta_{\mathbf X}\theta_{\mathbf Y}\subseteq\theta_{\mathbf X}\theta_{\mathbf Y}\theta_{\mathbf X}\theta_{\mathbf Y}
$$ 
in either case, and we are done.
\end{proof}

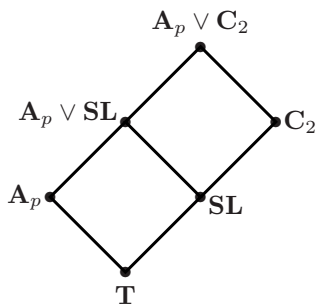
\begin{figure}[htb]
\unitlength=1mm
\linethickness{0.4pt}
\begin{center}
\begin{picture}(42,40)
\put(6,15){\circle*{1.33}}
\put(16,5){\circle*{1.33}}
\put(16,25){\circle*{1.33}}
\put(26,15){\circle*{1.33}}
\put(26,35){\circle*{1.33}}
\put(36,25){\circle*{1.33}}
\gasset{AHnb=0,linewidth=0.4}
\drawline(16,25)(6,15)(16,5)(26,15)(16,25)(26,35)(36,25)(26,15)
\put(5,15){\makebox(0,0)[rc]{$\mathbf A_p$}}
\put(26,38){\makebox(0,0)[cc]{$\mathbf A_p\vee\mathbf C_2$}}
\put(15,26){\makebox(0,0)[rc]{$\mathbf A_p\vee\mathbf{SL}$}}
\put(37,25){\makebox(0,0)[lc]{$\mathbf C_2$}}
\put(27,14){\makebox(0,0)[lc]{$\mathbf{SL}$}}
\put(16,2){\makebox(0,0)[cc]{$\mathbf T$}}
\end{picture}
\end{center}
\caption{The lattice $L(\mathbf A_p\vee\mathbf C_2)$}
\label{Fig: L(A_p vee C_2)}
\end{figure}

It is natural to define \emph{almost $fi$-$n$-permutable} [\emph{almost weakly $fi$-permutable}] varieties of semigroups or monoids as varieties on whose free objects any two fully invariant congruences contained in the least semilattice congruence $n$-permute [respectively, weakly permute]. A classification of almost weakly $fi$-permutable semigroup varieties in some wide partial case was provided in Vernikov~\cite{Vernikov-04b}. A minor inaccuracy in this result is fixed by Vernikov and Shaprynski\v{\i}~\cite{Vernikov-Shaprynskii-14}. Almost $fi$-$n$-permutable semigroup varieties with $n>3$ as well as almost $fi$-$n$-permutable monoid varieties with $n>2$ are not examined so far.


\begin{thebibliography}{99}
\bibitem{Almeida-94} 
J. Almeida, Finite Semigroups and Universal Algebra, World Scientific, Singapore, 1994.
\bibitem{Evans-71}
T. Evans, \emph{The lattice of semigroup varieties}, Semigroup Forum, \textbf 2 (1971), 1--43.
\bibitem{Freese-Nation-73}
R. Freese and J. B. Nation, \emph{Congruence lattices of semilattices}, Pacif. J. Math., \textbf{49} (1973), 51--58.
\bibitem{Gratzer-11}
G. Gr\"atzer, Lattice Theory: Foundation, Springer Basel AG, 2011.
\bibitem{Gusev-18}
S. V. Gusev, \emph{Special elements of the lattice of monoid varieties}, Algebra Universalis, \textbf{79} (2018), Article 29, 1--12.
\bibitem{Gusev-19} 
S. V. Gusev, \emph{On the ascending and descending chain conditions in the lattice of monoid varieties}, Siberian Electronic Math. Reports, \textbf{16} (2019), 983--997.
\bibitem{Gusev-20+}
S. V. Gusev, \emph{Standard elements of the lattice of monoid varieties}, Algebra i Logika, submitted [Russian, Engl. translation available at: http://arxiv.org/abs/1909.13013].
\bibitem{Gusev-Vernikov-18}
S. V. Gusev and B. M. Vernikov, \emph{Chain varieties of monoids}, Dissert. Math., \textbf{534} (2018), 1--73.
\bibitem{Head-68}
T. J. Head, \emph{The varieties of commutative monoids}, Nieuw Arch. Wiskunde. III Ser., \textbf{16} (1968), 203--206.
\bibitem{Hobby-McKenzie-88}
D. Hobby and R. N. McKenzie, The Structure of Finite Algebras, Amer. Math. Soc. Providence, Rhode Island. Contemp. Math. Vol.~76 (1988).
\bibitem{Jackson-05}
M. Jackson, \emph{Finiteness properties of varieties and the restriction to finite algebras}, Semigroup Forum, \textbf{70} (2005), 154--187.
\bibitem{Jackson-Lee-18}
M. Jackson and E. W. H. Lee, \emph{Monoid varieties with extreme properties}, Trans. Amer. Math. Soc., \textbf{370} (2018), 4785--4812.
\bibitem{Jones-88}
P. R. Jones, \emph{Congruence semimodular varieties of semigroups}, Lect. Notes Math., \textbf{1320} (1988), 162--171.
\bibitem{Jonsson-53}
B. J\'onsson, \emph{On the representation of lattices}, Math. Scand., \textbf 1 (1953), 193--206.
\bibitem{Jonsson-72}
B. J\'onsson, \emph{The class of Arguesian lattices is self-dual}, Algebra Universalis, \textbf 2 (1972), 396.
\bibitem{Lee-12}
E. W. H. Lee, \emph{Varieties generated by $2$-testable monoids}, Studia Sci. Math. Hungar., \textbf{49} (2012), 366--389.
\bibitem{Lee-13} 
E. W. H. Lee, \emph{Almost Cross varieties of aperiodic monoids with central idempotents}, Beitr. Algebra Geom., \textbf{54} (2013), 121--129.
\bibitem{Lee-14}
E. W. H. Lee, \emph{Inherently non-finitely generated varieties of aperiodic monoids with central idempotents}, Zapiski Nauchnykh Seminarov POMI (Notes of Sci. Seminars of the St~Petersburg Branch of the Math. Institute of the Russ. Acad. Sci.), \textbf{423} (2014), 166--182; see also J. Math. Sci., \textbf{209} (2015), 588--599.
\bibitem{Lipparini-95}
P. Lipparini, $n$-\emph{permutable varieties satisfy non trivial congruence identities}, Algebra Universalis, \textbf{33} (1995), 159--168.
\bibitem{Pastijn-91}
F. Pastijn, \emph{Commuting fully invariant congruences on free completely regular semigroups}, Trans. Amer. Math. Soc., \textbf{323} (1990), 79--92.
\bibitem{Petrich-Reilly-90}
M. Petrich and N. R. Reilly, \emph{The modularity of the lattice of varieties of completely regular semigroups and related representations}, Glasgow Math. J., \textbf{32} (1990), 137--152.
\bibitem{Tully-64}
E. J. Tully, \emph{The equivalence, for semigroup varieties, of two properties concerning congruence relations}, Bull. Amer. Math. Soc., \textbf{70} (1964), 399--400.
\bibitem{Vernikov-04a}
B. M. Vernikov, \emph{On weaker variant of congruence permutability for semigroup varieties}, Algebra i Logika, \textbf{43} (2004), 3--31 [Russian; Engl. translation: Algebra and Logic, \textbf{43} (2004), 1--16].
\bibitem{Vernikov-04b}
B. M. Vernikov, \emph{On semigroup varieties on whose free objects almost all fully invariant congruences are weakly permutable}, Algebra i Logika, \textbf{43} (2004), 635--649 [Russian; Engl. translation: Algebra and Logic, \textbf{43} (2004), 357--364].
\bibitem{Vernikov-04c}
B. M. Vernikov, \emph{Completely regular semigroup varieties whose free objects have weakly permutable fully invariant congruences}, Semigroup Forum, \textbf{68} (2004), 154--158.
\bibitem{Vernikov-Shaprynskii-14}
B. M. Vernikov and V. Yu. Shaprynski\v{\i}, \emph{Three weaker variants of congruence permutability for semigroup varieties}, Siberian Electronic Math. Reports, \textbf{11} (2014), 567--604 [Russian].
\bibitem{Vernikov-Volkov-97}
B. M. Vernikov and M. V. Volkov, \emph{Permutability of fully invariant congruences on relatively free semigroups}, Acta Sci. Math. (Szeged), \textbf{63} (1997), 437--461.
\bibitem{Vernikov-Volkov-00}
B. M. Vernikov and M. V. Volkov, \emph{Commuting fully invariant congruences on free semigroups}, Contrib. General Algebra, \textbf{12} (2000), 391--417. 
\bibitem{Volkov-05}
M. V. Volkov, \emph{Modular elements of the lattice of semigroup varieties}, Contrib. General Algebra, \textbf{16} (2005), 275--288.
\bibitem{Wismath-86}
S. L. Wismath, \emph{The lattice of varieties and pseudovarieties of band monoids}, Semigroup Forum, \textbf{33} (1986), 187--198.
\end{thebibliography}
\end{document}